\documentclass[reqno,12pt]{amsart}
\usepackage{amsmath, amssymb, amsthm,amsfonts}
\usepackage[english]{babel}
\usepackage{bbm}
\usepackage{graphicx}
\usepackage{url}
\usepackage{soul}
\usepackage{epstopdf}
\usepackage[ruled,vlined]{algorithm2e}
\usepackage[a4paper,bindingoffset=0.5cm,left=2cm,right=2cm,top=2.5cm,bottom=2cm,footskip=.8cm]{geometry}

\usepackage{thmtools}
\usepackage{thm-restate}

\usepackage{tikz}
\usetikzlibrary{arrows, automata,positioning,calc,decorations.pathreplacing,decorations.markings,shapes.misc,petri,topaths}
\usepackage{pgfplots}
\pgfplotsset{compat=newest}
\usetikzlibrary{plotmarks}
\usepackage{grffile}
\newlength\figureheight
\newlength\figurewidth

\pgfplotsset{%
        tick label style={font=\scriptsize},
        label style={font=\footnotesize},
        legend style={font=\footnotesize},
        every axis plot/.append style={very thick}
}

\usepackage{url}

\usepackage{rotating}
\usepackage{amsbsy,enumerate}
\usepackage{graphicx}
\usepackage{ccaption}
\usepackage{comment}
\usepackage{mathrsfs}
\renewcommand{\leq}{\leqslant}
\renewcommand{\geq}{\geqslant}

\newcommand{\vb}{\vspace{3.2mm}}

\newcommand{\vertiii}[1]{{\left\vert\kern-0.25ex\left\vert\kern-0.25ex\left\vert #1
                \right\vert\kern-0.25ex\right\vert\kern-0.25ex\right\vert}}

\allowdisplaybreaks

\newtheorem{lemma}{Lemma}
\newtheorem{corollary}{Corollary}
\newtheorem{assumption}{Assumption}
\newtheorem{theorem}{Theorem}
\newtheorem{remark}{Remark}
\newtheorem{example}{Example}
\newtheorem{definition}{Definition}
\newtheorem{proposition}{Proposition}
\newtheorem{conj}{Conjecture}
\makeatletter
\renewcommand{\fnum@figure}[1]{\textbf{\figurename~\thefigure}. }
\renewcommand{\fnum@table}[1]{\textbf{\tablename~\thetable}. }
\makeatother

\usepackage[utf8]{inputenc}

\usepackage{type1cm}        

\usepackage{graphicx}        
\usepackage{multicol}        
\usepackage[bottom]{footmisc}

\usepackage{newtxtext}       %
\usepackage{newtxmath}       

\usepackage{tikz}
\usetikzlibrary{plotmarks}
\usetikzlibrary{decorations.pathreplacing}
\usepackage{pgfplots}
\pgfplotsset{width=10cm,compat=1.9}
\usepgfplotslibrary{external}
\usepackage{subcaption}

\renewcommand{\fnum@figure}[1]{\textbf{\figurename~\thefigure}. }
\renewcommand{\fnum@table}[1]{\textbf{\tablename~\thetable}. }
\makeatother

\allowdisplaybreaks

\usepackage{hyperref}

\begin{document}

\title[Delayed Hawkes birth-death processes]{Delayed Hawkes birth-death processes}       
\author{Justin Baars, Roger J.~A.~Laeven, and Michel Mandjes}
       
\begin{abstract}
We introduce, and formally establish, a variant of the Hawkes-fed birth-death process --- the \textit{delayed Hawkes birth-death process} --- in which the conditional intensity does not increase at arrivals but at departures from the system. 
In a scaling limit where sojourn times are stretched out by a factor $\sqrt T$, after which time gets contracted by a factor $T$, the delayed Hawkes process behaves markedly differently from its classical counterpart. 
We design a family of models admitting a cluster representation 
and containing the Hawkes and delayed Hawkes processes as special cases. 
The cluster representation allows for transform characterizations by a fixed-point equation and for analysis of heavy-tailed asymptotics.
We compare the delayed Hawkes process to the classical Hawkes process using stochastic ordering, which enables us to describe stationary distributions and heavy-traffic behavior.
In the Markovian network case, a recursive procedure is presented to calculate the $d$th-order moments analytically.

\vb


\vb


\vb

\noindent
{\sc Affiliations.} 
JB and RL are with the Dept.~of Quantitative Economics, University of Amsterdam. 
RL is also with E{\sc urandom}, Eindhoven University of Technology, and with C{\sc ent}ER, Tilburg University. 
MM is with the Mathematical Institute, Leiden University, and is also affiliated with the Korteweg-de Vries Institute for Mathematics, University of Amsterdam; E{\sc urandom}, Eindhoven University of Technology, Eindhoven; Amsterdam Business School, University of Amsterdam. 
The research of JB and RL is 
funded in part by the Netherlands Organization for Scientific Research under an NWO VICI grant (2020--2027).
The research of MM is funded in part by the NWO Gravitation project N{\sc etworks}, grant number 024.002.003.

\vb

\noindent
{\sc Email addresses.} \url{j.r.baars@uva.nl},\  \url{r.j.a.laeven@uva.nl}, and\  \url{m.r.h.mandjes@math.leidenuniv.nl}.

\vb

\noindent
\textit{Date}: \today.
                           
\vspace{-1cm}
\end{abstract}

\maketitle
 
 
\section{Introduction}\label{introduction}
Since their introduction in 1971 \cite{Hawkes, Hawkes2}, Hawkes processes have gained significant attention in the academic literature. 
One notable application is in finance \cite{ACL15,ALP14,finance2, finance}, where they have been used to capture the clustering behavior of financial returns and transactions, such as stock trades or order arrivals in electronic markets. 
Hawkes processes have also been applied in social network analysis \cite{social media2}, to represent the contagious nature of information diffusion or the spread of online content in social media platforms. 
Additionally, in the field of seismology \cite{seismology, earthquakestressrelease, Ikefuji} they are used to model earthquake aftershock sequences. 
Other applications include the analysis of disease outbreaks \cite{covid}, the prediction of online user activity \cite{social media}, crime modeling \cite{criminology}, and the assessment of neuronal spike trains \cite{neuroscience}. 
The versatility of Hawkes processes makes them a valuable tool in various domains, providing insights into the underlying mechanisms driving the observed events.

The Hawkes process is a self-exciting c\`adl\`ag  point process $(N(t))_{t\geq0}$, which can be defined through its conditional intensity process $(\Lambda(t))_{t\geq0}$ \cite{DaleyVereJones}.
In the simplest linear, unmarked, univariate case, the (left-continuous, predictable) conditional intensity process is given by 
\begin{equation}\label{intensityHawkes}
\Lambda(t)=\lambda_0+\sum_{t_i<t}h(t-t_i)=\lambda_0+\int_{(-\infty,t)}h(t-s)\ \mathrm dN(s),
\end{equation}
where $\lambda_0>0$ is the \textit{baseline intensity}, or \textit{immigration intensity}, $(t_i)_{i\in\mathbb N}$ is an increasing sequence of \textit{arrival times}, $h:[0,\infty)\to[0,\infty)$ is the \textit{excitation kernel}, which is assumed to be integrable, and $N((-\infty,0))$ is some initial condition, typically either a random initial condition resulting in a stationary version of the process, or an empty history.

Besides being a suitable process to model real-world phenomena, the Hawkes process owes much of its popularity to its high tractability. 
In particular, recursive procedures have been developed to determine corresponding moments  \cite{elementary, matrix method, DZ11, Infinite server queues}; 
a procedure has been devised by which, in the context of Hawkes-fed population processes, 
transforms can be approximated by iterates of a certain operator \cite{multivariateKLM}; 
heavy-tailed and heavy-traffic asymptotics have been identified \cite{multivariateKLM, Infinite server queues}; 
techniques for nonparametric estimation of the model parameters, with provable performance guarantees, have been set up \cite{Kirchner}; 
a broad range of scaling and large deviation limits have been studied \cite{Bacry, Horst, Rosenbaum1, Rosenbaum2,ldpKLM,Zhu}; 
existence, uniqueness and stability results have been established that apply under great generality \cite{Stabilitypaper, Massoulie, AgeDependent}; 
and recently results on the distribution of the Hawkes process' underlying cluster duration have become available \cite{Conditional uniformity}. 
Evidently, this list is by no means exhaustive, but it provides an illustration of the process' amenability for analysis, focusing on contributions of direct relevance to this paper. 

Since its inception, various generalized versions of the basic variant of the Hawkes process have been examined, all of them being point or population processes in which the occurrence of events affect the conditional intensity process. 
For example, Massoulié \cite{Massoulie} considers a highly flexible family of models involving a (possibly) nonlinear intensity function $\lambda$: 
\begin{equation}
\Lambda(t)=\lambda\left(\int_{(-\infty,t)}h(t-s,B_s)\ \mathrm dN(s)\right),
\end{equation}  
which allows for the dependency on space-dependent \textit{random marks}, $(B_{s_{i}})_{i\in \mathbb{N}}$, taking values in some general measurable space.
Hawkes-driven birth-death population processes have been studied in \cite{dawpenderqueues, multivariateKLM, Infinite server queues}.
Another variant is the \textit{ephemerally self-exciting point process}, as introduced in \cite{Ephemeral}, in which the excitation caused by the $i$-th arrival vanishes after some stochastic time $J_i$. 
By considering this system as a birth-death process with lifetimes $(J_i)_{i\in\mathbb N}$, one could say that `a particle excites as long as it is in the system'. 
Another variant of the classical Hawkes process is analyzed in \cite{AgeDependent}, in which the excitation is dependent on the time since the last arrival, a phenomenon termed \textit{age-dependency}. 
A process that describes behavior opposite to the Hawkes process, is the {\it self-correcting process} \cite{self-correcting0, self-correcting1, self-correcting2}, in which any arrival \textit{decreases} the conditional intensity, making more arrivals in the near future {\it less} likely.\footnote{A general observation, based on the cases dealt with in the literature, is that tractability tends to be preserved for models that admit a cluster process representation. 
This is the case for multivariate linear marked Hawkes point and birth-death processes (covering specific Hawkes-fed population processes), and for the ephemerally self-exciting process. 
For these classes of processes one has succeeded in establishing analogs to results known for the classical Hawkes process. 
On the other hand, for processes having nonlinear intensity functions, age-dependent processes, Hawkes-fed single-server queues and self-correcting processes, there is no cluster representation, making such models considerably harder to analyze than the classical Hawkes process.}
Recently, Hawkes processes allowing for both self-excitation and self-inhibition were studied; see \cite{CCC22} and the references therein.

In this paper, we introduce a variant of the Hawkes process new to the literature, to the best of our knowledge. 
This variant is motivated as follows. 
Consider first a standard Hawkes-fed birth-death population process, or `infinite-server queue with Hawkes input', denoted by $(Q(t))_{t\geq0}$; see e.g.,\ \cite{dawpenderqueues,Infinite server queues}. 
Then, particles arrive at rate $\Lambda(t)$, and at the $i$-th {\it arrival} at time $t_i$, the conditional intensity process $\Lambda(t)$ jumps upwards by $B_ih(t-t_i)$, where $(B_i)_{i\in\mathbb N}$ are i.i.d.\ marks. 
The particle stays in the system for a duration $J_i$, where $(J_i)_{i\in\mathbb N}$ are i.i.d.\ lifetimes, or `service times' in queueing terminology; 
after departure, the excitation effect is still present. 
By contrast, we define a process in which the conditional intensity does not jump at arrivals, but at {\it departures} from the system. 
More specifically, the intensity process $\Lambda(t)$ does not change at an arrival, but jumps upwards by $B_ih(t-t_i')$ at the $i$-th departure at time $t_i'$. 
In this situation, an arrival still increases the conditional intensity $\Lambda(t)$, but only after a \emph{delay} equal to its lifetime (or service time, in queueing terms).
For this reason, one may call the corresponding process $(Q(t))_{t\geq0}$ 
a \textit{delayed Hawkes birth-death process}, or 
a \textit{delayed Hawkes infinite-server queue}; or, more briefly, a DH/G/$\infty$ queue, using Kendall's notation. 
We refer to the counting process $(N(t))_{t\geq0}$ as a \textit{delayed Hawkes process} or briefly as \textit{delayed Hawkes}.\footnote{The terms `infinite-server queue' and `birth-death population process' can be used interchangeably. 
In fact, one could argue that an infinite-server queue is not really a queue, since customers are always served directly, do not observe each other, and never wait. 
Using population processes terminology, one could refer to delayed Hawkes infinite-server queues as \textit{birth-death processes exhibiting posthumous excitation}.}

A typical realization of the delayed Hawkes birth-death process can be found in Figure~\ref{figDHreal}.
Intuitively, one would expect this process to share some common features with the classical Hawkes process, but with a `lower level of clustering' of events: one has to wait some time (distributed as the random variable $J$) for the excitation to start, so that arrivals induced by excitation are further away from the initial arrival than under the classical Hawkes process.

\begin{figure}[!htbp]
\centering
   \includegraphics[width=1\linewidth]{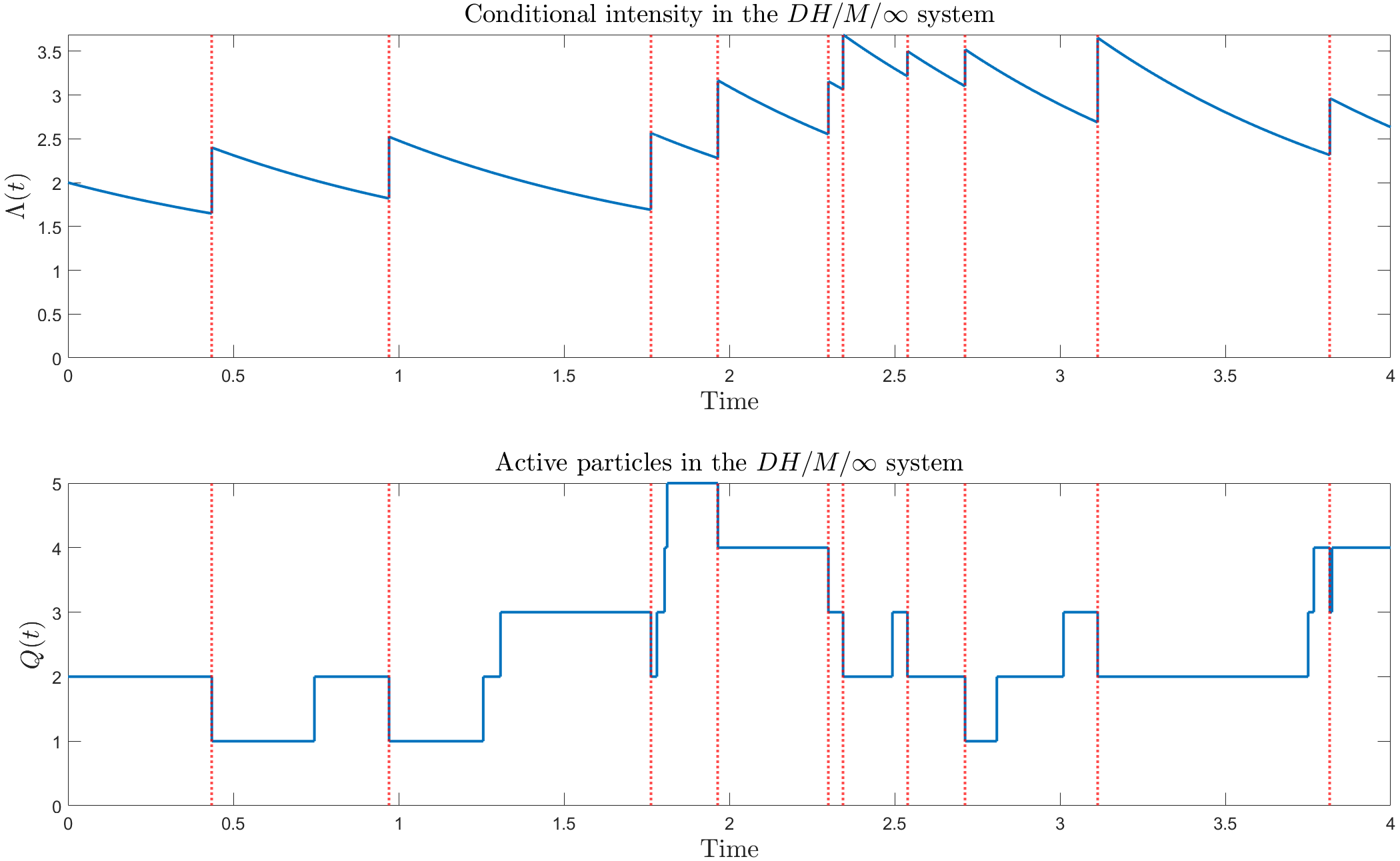}
   \caption{A realization of the Markovian DH/M/$\infty$ queue, with $\lambda_0=1$, $h(t)=e^{-t}$, $B\sim\mathrm{Beta}(3.5,1.5)$ and $J\sim\mathrm{Exp}(1)$. 
   We start at $Q(0)=2=\Lambda(0)$. 
   The vertical dotted lines correspond to departures, causing intensity increases.} 
   \label{figDHreal}
\end{figure}

By setting the lifetimes $J$ equal to zero and by keeping track of $N(\cdot)$, we recover the classical Hawkes process, entailing that the delayed Hawkes process constitutes a generalization of the classical Hawkes process.
In the following examples, the delayed Hawkes birth-death process may provide a realistic and appealing probabilistic model.
\begin{itemize}
\item Word-of-mouth referrals: in a queueing context, customers who are satisfied about the service may excite other potential customers. 
Therefore, the arrival process may behave like a self-exciting process; however, a customer typically does not start exciting others during the service, but only starts doing so upon leaving/finishing the system/service.
\item In epidemiology, the spread of infectious diseases often exhibits \emph{delayed} self-exciting behavior. 
Indeed, when an individual becomes infected, there is typically an incubation period before the individual starts showing symptoms or becomes contagious. 
As more individuals become infected,   start exhibiting symptoms and become contagious, the transmission rate increases, resulting in a (delayed) increase in the number of new cases.
\item A financial order typically triggers more orders, but it may take time before an order is executed and therefore before it starts exciting. 
Even when the execution time is (very) small, as in liquid electronic markets, this delay changes the dynamics. 
Similar patterns arise in neuroscience.
\item On social media platforms, the spread of content can exhibit self-exciting behavior with delay. 
When a popular post or topic emerges, it can trigger a cascade of user interactions. 
As it takes time for users to engage with the content and for the effects to ripple through their social networks, the propagation of these interactions can display a delayed response.
\end{itemize}
If we would like to model real-world phenomena, such as those described in the examples above, using the delayed Hawkes process, we need the process to be tractable, in order to understand its probabilistic structure. 
As it turns out, this novel process is remarkably tractable: its linear version admits a cluster process representation, and many results that are known for the classical Hawkes process have suitably modified counterparts for the delayed Hawkes process.

This work contributes to the literature in several ways. 
First, we introduce the delayed Hawkes process. 
In fact, more generally, we introduce a family of multivariate sojourn-time dependent point processes, containing the classical Hawkes, delayed Hawkes, and the ephemerally self-exciting point process \cite{Ephemeral} as special cases. 
This general family of models is formulated via a stochastic differential equation for the conditional intensity process, which we exploit to prove existence, uniqueness and stability results, leveraging methodology from \cite{Stabilitypaper, Massoulie}. 

Second, we contribute to a rich literature on scaling limits for Hawkes processes, see e.g., \cite{Bacry, Horst, Rosenbaum1, Rosenbaum2}, by deriving a scaling limit that exhibits the effect of the delay for the delayed Hawkes process. 
Specifically, we show that our family of models obeys the same functional central limit theorem as the classical Hawkes process;
however, in a scaling regime in which sojourn times are stretched out by a factor $\sqrt T$, after which time gets contracted by a factor $T$, and $T$ is sent to $\infty$, the delayed Hawkes process behaves markedly differently from its classical counterpart. 

Third, for the linear version of our family of models, we provide a cluster process representation, allowing us to derive fixed-point equations that enable transform characterizations. 
In addition, we employ these fixed-point equations to establish heavy-tailed asymptotics. 
We use the cluster representation of the Hawkes and delayed Hawkes processes to prove stochastic dominance results, which are typically proved by comparing sample paths. 
In essence, we couple sample paths only within generations, obtaining a complex genealogical coupling for both processes. 
From a methodological standpoint, the ideas underlying this approach have the potential to be fruitful in other contexts as well.

Finally, we generalize results of \cite{Infinite server queues} for calculating moments of the Hawkes process in the univariate, Markovian setting, to a higher-dimensional, delayed Hawkes setting, also allowing for network effects.
Interestingly, this analysis now involves a Clement-Kac-Sylvester matrix.

The remainder of this article is structured as follows. 
In Section~\ref{preliminaries}, we introduce a general family of multivariate point process models encompassing classical Hawkes, delayed Hawkes and ephemeral Hawkes as special cases.
In Section~\ref{existence}, existence, uniqueness and stability results are established for this general family of models.
Section~\ref{scalinglimit} studies scaling limits; in particular, we derive a scaling limit for delayed Hawkes highlighting the effect of the delay.
In Section~\ref{chapternonmarkov}, 
we use cluster-representation based methods to describe fixed points in the transform domain;
and 
exploit those fixed-point equations to derive heavy-tailed asymptotics.
In Section~\ref{comparisons}, we compare Hawkes to delayed Hawkes systems using stochastic ordering. 
In Section~\ref{markov}, we study Markovian models, for which we describe recursive methods to calculate moments analytically.
We provide a discussion and concluding remarks in Section~\ref{discussion}.
Various (lengthy) proofs and some additional results are relegated to the Appendix.
In online Supplementary Material \cite{Supplement}, we provide the proof of Theorem~\ref{summary KLM}.

\section{Model definitions} \label{preliminaries}
In this section, we introduce and provide definitions for a  family of multivariate point process models having sojourn-time dependent excitation, using both their conditional intensity processes and, in the linear case, their cluster process representation.
Furthermore, we define a network of delayed Hawkes birth-death processes through conditional intensities. 
 
We start by defining a family of models exhibiting sojourn-time dependent excitation, encompassing the classical Hawkes, the delayed Hawkes, and the ephemerally self-exciting \cite{Ephemeral} process.
We first describe this family of models through a conditional intensity representation, allowing for nonlinearity.
We then restrict attention to the linear case for which we also provide a cluster representation-based definition. 
The two definitions are equivalent for processes starting on an empty history, whenever the cluster representation exists (i.e., in the linear case).
The conditional intensity-based definition allows for nonlinear effects, but we only use this in Section~\ref{existence} when proving existence, uniqueness and stability results; in the rest of this article we focus on the linear case. 
 
We denote the $d$-dimensional joint point (or counting), birth-death and conditional intensity process of the intended model (defined below) by the triple $(\boldsymbol  N(t),\boldsymbol  Q(t),\boldsymbol\Lambda(t))_{t\geq0}$, with $\boldsymbol  Z(t)=[Z_1(t)\,\cdots\, Z_d(t)]^\top$, for $ \boldsymbol Z\in\{\boldsymbol N,\boldsymbol Q,\boldsymbol\Lambda\}$.
We write $t_1<t_2<\cdots$ for the a.s.\ increasing sequence of jump (or event) times of $\boldsymbol N(\cdot)$, and we denote events by triples $(t_r,j_r,J_r)$, where $J_r\sim J_j$ if $j_r=j$. 
We assume that an arrival in coordinate $j$ at time $t_r$ induces a \emph{random} jump in the intensity in the $i$-th coordinate of size $h_{ij,J,\omega}(\cdot-t_r)$; the randomness in $h_{ij,J}$ is modeled by the $\omega$-dependence.
  
\begin{definition}[Conditional intensity for $d$-dimensional point processes with sojourn-time dependent excitation]\label{Def1condint}
Let $d\in\mathbb N$ denote the dimension. 
For $j\in[d]$, let $J_j$ be the positive sojourn time random variable of coordinate $j$. 
For each $i,j\in[d]$, 
let $\omega\mapsto h_{ij,J_j,\omega}$ be a random $J_j$-dependent piecewise continuous function with support contained in $[0,\infty)$, for almost all $(J_j,\omega)$. 
Furthermore, suppose that for each $i,j\in[d]$, $\mathbb E_{J_j}\mathbb E_{\omega|J_j}\|h_{ij,J_j,\omega}\|_{L^\infty}<\infty$.
Assume that the realizations of the random functions are conditionally (on $J$) cross-sectionally and serially independent.
Suppose that the lifetimes are drawn at the time of arrival. Let $\boldsymbol H_{\boldsymbol J}$ be the (random, $\boldsymbol J$-dependent) matrix consisting of elements $(\boldsymbol H_{\boldsymbol J})_{ij}=h_{ij,J}$, where the $j$-th column is dependent on the same realization of $J_j$. 
Define the $d$-dimensional c\`adl\`ag point process $\boldsymbol N=(N_i(t))_{i\in[d],t\in\mathbb R}$  with sojourn-time dependent excitation through
\begin{align*}
\mathbb P(N_i(t+\Delta t)-N_i(t)=0\,|\,\mathcal H_t)&=1-\Lambda_i(t)\Delta t+o(\Delta t),\\\mathbb P(N_i(t+\Delta t)-N_i(t)=1\,|\,\mathcal H_t)&=\Lambda_i(t)\Delta t+o(\Delta t), \\\mathbb P(N_i(t+\Delta t)-N_i(t)\geq2\,|\,\mathcal H_t)&=o(\Delta t),
\end{align*}
as $\Delta t\downarrow 0$, where $(\mathcal H_t)_{t\in\mathbb R}=\sigma(\boldsymbol N(s),(\boldsymbol H_{\boldsymbol J}(\cdot))(s),\boldsymbol J(s):s\leq t)_{t\in\mathbb{R}}$ is the natural filtration generated by $\boldsymbol N$ along with random lifetimes $\boldsymbol J(s)$ and excitation kernels $(\boldsymbol H_{\boldsymbol J}(\cdot))(s)$ corresponding to an arrival at time $s$.
In the linear case, we set \begin{equation}
\boldsymbol\Lambda(t)=\boldsymbol\lambda_0+\int_{-\infty}^t\boldsymbol H_{\boldsymbol J}(t-s)\ \mathrm d\boldsymbol N(s),\label{intensity multi}
\end{equation} 
where $\boldsymbol\lambda_0\geq 0$, with at least one of the base rates being strictly positive, and the integral in \eqref{intensity multi} is understood to exclude $t$.
In the nonlinear case, we take measurable $L_i$-Lipschitz functions $\phi_i:\mathbb R\to\mathbb R_+$, for each $i\in[d]$, and define the conditional intensity of the $i$-th coordinate via 
\begin{equation}
\Lambda_i(t)=\phi_i\left(\sum_{j=1}^d\int_{-\infty}^th_{ij,J}(t-s)\ \mathrm dN_j(s)\right)\label{intensity nonlinear1}.
\end{equation}
\end{definition}

From the point process, including the realizations of sojourn times, the birth-death process $(\boldsymbol Q(t))_{t\geq0}$ can easily be constructed. 
The conditional intensity process $\boldsymbol\Lambda(\cdot)$ is taken left-continuous and is predictable; cf.\ \cite[Example 7.2(b) and Ch.~14]{DaleyVereJones}. 
We interchangeably start the point process on a history on $(-\infty,0)$, which typically refers to the stationary version of the point process, or on an empty history, in which case an integral $\int_{-\infty}^t\cdot\ \mathrm d\boldsymbol{N}(s)$ reduces to $\int_{0}^t\cdot\ \mathrm d\boldsymbol{N}(s)$.

Definition~\ref{Def1condint} encompasses the multivariate marked \emph{classical} Hawkes process, the multivariate marked hybrid \emph{ephemerally} self-exciting process (cf.\ \cite{Ephemeral}), and the multivariate marked \emph{delayed} Hawkes process, which are defined by setting
\begin{align}
h_{ij,J,\omega}(\cdot)=
\begin{cases}
B_{ij,\omega}h_{ij}(\cdot),&\qquad\mathrm{(classical)}\\
B_{ij,\omega}h_{ij}(\cdot)\mathbf 1\{\cdot<J\},&\qquad\mathrm{(ephemeral)}\\
B_{ij,\omega}h_{ij}(\cdot-J)\mathbf 1\{\cdot>J\},&\qquad\mathrm{(delayed)}
\end{cases}
\label{eq:general}
\end{align}
respectively.

We highlight the richness of the family of processes introduced in Definition~\ref{Def1condint}. 
Notably, this family includes processes where the degree of self-excitation depends, positively or negatively, upon the lifetimes of particles. 
For instance, one can define $$h_{ij,J,\omega}(\cdot) = B_{ij,J,\omega} h_{ij}(\cdot).$$ 
One might \emph{a priori} expect that such general models would be intractable; however, e.g., Theorem~\ref{summary KLM} in Section~\ref{chapternonmarkov} demonstrates that calculations for these models can, in fact, be carried out effectively.

In Section~\ref{existence}, existence, uniqueness and stability for the nonlinear family of sojourn-time dependent point processes with intensities \eqref{intensity nonlinear1} is established. 
For the rest of this article, we focus on the linear case \eqref{intensity multi}. 
The linear case admits a cluster representation, as follows.
\renewcommand{\theenumi}{\roman{enumi}}
\begin{definition}[Cluster representation for $d$-dimensional point processes with sojourn-time dependent excitation]\label{DHcluster2} 
Let $d\in\mathbb N$ denote the dimension. 
For $j\in[d]$, let $J_j$ be the positive sojourn time random variable of coordinate $j$. 
For each $i,j\in[d]$, 
let $h_{ij,J_j,\omega}$ be a random $J_j$-dependent piecewise continuous function with support contained in $[0,\infty)$, for almost all $(J_j,\omega)$.  
Let $K_{ij,J_j,\omega}$ be an inhomogeneous Poisson process of intensity $h_{ij,J_j,\omega}$. Furthermore, suppose that for each $i,j\in[d]$, $\mathbb E_{J_j}\mathbb E_{\omega|J_j}\|h_{ij,J_j,\omega}\|_{L^\infty}<\infty$, i.e., that $h_{ij,J,\omega}$ is a.s.\ bounded. 

Now let $T\in[0,\infty]$, and define a point process $\boldsymbol N(\cdot)$ through a sequence of events generated according to the following procedure:
\begin{enumerate}
\item For $j\in[d]$, let $I_j(\cdot)$ be a homogeneous Poisson process of rate $\lambda_{j,0}$, generating immigration events \[\{(t_r^{(0)},j,J_r^{(0)})\}_{r=1}^{R_j^{(0)}(t)},\] and where $J_r^{(0)}\stackrel{\mathrm{iid}}\sim J_j$ and where $R_j^{(0)}(t)$ is the number of immigration events in component $j$ up to time $t\in[0,T]$.
\item For each immigration event $(t_r^{(0)},j,J_r^{(0)})$, in each target component $m\in[d]$, generate first-generation events \[\{(t_r^{(1)},m,J_r^{(1)})\}_{r=1}^{R_m^{(1)}(t)}\] according to $K_{mj,J_r^{(0)},\omega}(t-t_r^{(0)})$, where $J_r^{(1)}\stackrel{\mathrm{iid}}\sim J_m$.
\item Upon iterating the above rule, given an $r$-th event of the $(n-1)$-st generation in source component $m\in[d]$, for each target component $l\in[d]$, descendant $(t_r^{(n-1)},m,J_r^{(n-1)})$ generates $n$-th generation events \[\{(t_r^{(n)},l,J_r^{(n)})\}_{r=1}^{R_l^{(n)}(t)},\] according to $K_{lm,J_r^{(n-1)},\omega}(t-t_r^{(n-1)})$, and where $J_r^{(n)}\stackrel{\mathrm{iid}}\sim J_l$.
\end{enumerate} 
Here, the Poisson processes are conditionally independent within and between each iteration, and the excitation functions are drawn conditionally independently. 
Then, 
$$\boldsymbol N(t)=\bigcup_{n\geq0}\left(\{(t_r^{(n)},1,J_r^{(n)})\}_{r=1}^{R_1^{(n)}(t)}\times\cdots\times\{(t_r^{(n)},d,J_r^{(n)})\}_{r=1}^{R_d^{(n)}(t)}\right)$$ is the resulting multivariate sojourn-time dependent Hawkes process.
\end{definition}
In Definition~\ref{DHcluster2}, having drawn lifetimes and excitation functions, one can construct the corresponding birth-death and conditional intensity processes, $\boldsymbol Q(\cdot),\boldsymbol\Lambda(\cdot)$, in a straightforward manner.


The cluster representation from Definition~\ref{DHcluster2} exhibits the following useful properties. 
First, modulo the time shift corresponding to the arrival times, clusters generated by immigrants in the same coordinate are i.i.d. 
Second, cluster processes are generated independently across source components. 
Finally, within each source component, every event produces offspring using an identical iterative procedure, as each child represents a cluster, thus demonstrating \textit{self-similarity}.

For later use, these cluster properties can be operationalized using notation borrowed from \cite{multivariateKLM}. 
For an arrival in coordinate $j$, denote the $d$-dimensional counting, birth-death and intensity cluster process it generates by $\boldsymbol S_j^{\boldsymbol N}(\cdot), \boldsymbol S_j^{\boldsymbol Q}(\cdot),\boldsymbol S_j^{\boldsymbol \Lambda}(\cdot)$, respectively. 
Those have $i$-th coordinate $S_{i\leftarrow j}^{\boldsymbol \star}(\cdot)$ for $\boldsymbol \star\in\{\boldsymbol N,\boldsymbol Q,\boldsymbol \lambda\}$.
Here, $S_{i\leftarrow j}^{\boldsymbol N}(u)$ records the number of events in component $i$ up to time $u$ with as oldest ancestor the arrival generating $\boldsymbol S_j^{\boldsymbol N}(\cdot)$, including the arrival itself when $i=j$. 
Similarly, $S_{i\leftarrow j}^{\boldsymbol Q}(u)$ records the number of nonexpired events in component $i$ up to time $u$ with as oldest ancestor the arrival generating $\boldsymbol S_j^{\boldsymbol Q}(\cdot)$, including the ancestor itself if $i=j$ and if the ancestor has not yet left the system. 
Finally, $S_{i\leftarrow j}^{\boldsymbol \Lambda}(u)$ records aggregated  change in the intensity of component $i$ caused by jumps with excitation functions $h_{im,J,\omega}$, following arrivals in component $m$ with sojourn time $J$ within the cluster $\boldsymbol S_j^{\boldsymbol \Lambda}(\cdot)$ generated by an arrival in component $j$. 

Next, we define a $d$-dimensional network $(\boldsymbol  N(t),\boldsymbol  Q(t),\boldsymbol\Lambda(t))_{t\geq0}$ of (linear) delayed Hawkes birth-death processes. 

\begin{definition}[Network of delayed Hawkes birth-death processes]\label{multivariatedef}
Let $d\in\mathbb N$, and let $\mu_j,\mu_{ij}\geq0$ for all $i,j\in[d]$, such that for each $j\in[d]$, either $\mu_j>0$, or there is a sequence $(i_1,\ldots,i_k)\subset[d]$ such that $\mu_{i_k}\mu_{i_{k}i_{k-1}}\mu_{i_{k-1}i_{k-2}}\cdots\mu_{i_2i_1}\mu_{i_1j}>0$. For each $i,j\in[d]$, let $\lambda_{i,0}\geq0$,  let $(B_{ij}(s))_{s\in\mathbb R}$ be a collection of cross-sectionally and serially independent distributed random marks, distributed as the generic random variable $B_{ij}$, which is assumed to be positive a.s., and let $h_{ij}\in L^\infty$ be a.s.\ positive excitation functions.
Suppose that $\boldsymbol N(0)=\boldsymbol  Q(0)=\boldsymbol0$ and $\boldsymbol\Lambda(0)=\boldsymbol\lambda_0:=[\lambda_{1,0}\cdots \lambda_{d,0}]^\top$.
A \textit{network of delayed Hawkes birth-death processes} involves a $d$-dimensional point process $\boldsymbol  N(\cdot)$, taking values in $\mathbb N_0^d$, whose components $N_i(\cdot)$ satisfy, as $\Delta t\downarrow 0$, \begin{align*}\mathbb P(N_i(t+\Delta t)-N_i(t)=0\,|\,\mathcal H_t)&=1-\Lambda_i(t)\Delta t+o(\Delta t),\\\mathbb P(N_i(t+\Delta t)-N_i(t)=1\,|\,\mathcal H_t)&=\Lambda_i(t)\Delta t+o(\Delta t),\\\mathbb P(N_i(t+\Delta t)-N_i(t)\geq2\,|\,\mathcal H_t)&=o(\Delta t).
\end{align*} 
Suppose that the \textit{network of birth-death processes} $\boldsymbol  Q(\cdot)$ satisfies the following dynamics. 
(We write $\boldsymbol e_i$ for the $i$-th standard unit vector in $\mathbb R^d$.)
\begin{itemize}
\item \textit{Arrivals}, which are jumps upwards by $\boldsymbol e_i$, match jumps in $N_i(\cdot)$;
\item \textit{Rerouting} from coordinate $j$ to $i$, that is, a jump by $\boldsymbol e_i-\boldsymbol e_j$, occurs with probability $\mu_{ij}Q_j(t)\Delta t+o(\Delta t)$ in $(t,t+\Delta t)$;
\item \textit{Departures}, which are jumps downwards by $\boldsymbol e_j$, occur with probability $\mu_jQ_j(t)\Delta t+o(\Delta t)$ in $(t,t+\Delta t)$.
\end{itemize}
Now let $\boldsymbol D(\cdot)$  be the \textit{departure process}, taking jumps upwards by $\boldsymbol e_j$ precisely when there is a departure in coordinate $j$, i.e., when $\boldsymbol Q(\cdot)$ jumps downwards by $\boldsymbol e_j$.
The intensity $\Lambda_i(\cdot)$ of component $i$ is given by \begin{equation}\Lambda_i(t)=\lambda_{i,0}+\sum_{j=1}^d\int_{(0,t)}B_{ij}(s)h_{ij}(t-s)\ \mathrm dD_j(s).\label{networkintensity}\end{equation} 
The $\mathcal H_t$-progressively measurable process $\boldsymbol \Lambda(\cdot)$ is called the \textit{conditional intensity process}. 
(In \eqref{networkintensity}, we may integrate over $(-\infty,t)$ in order to study the process in stationarity.)
\end{definition}

We note that the process from Definition~\ref{multivariatedef} is Markovian if and only if we have $h_{ij}(t)=e^{-r_it}$ for all $i,j\in[d]$, where the $r_i$'s are called \textnormal{exponential rates}. 
We also note that one can easily generalize Definition~\ref{multivariatedef} to nonexponential sojourn times.

A particle in coordinate $j$ moves away at rate $\tilde\mu_j:=\mu_j+\sum_{i=1}^d\mu_{ij}$, after which it leaves the system with probability $\mu_j/\tilde\mu_j$, and is rerouted to coordinate $i$ with probability $\mu_{ij}/\tilde\mu_j$. 
Note that we do not assume that we have a feedforward network: we allow for the possibility of loops. 
Although in natural applications one would typically set $\mu_{jj}=0$, we do not make that assumption either.
%
A particle creates excitation as soon as it leaves the system. 
Because of the possibility of rerouting, this is not necessarily in the coordinate where the particle arrived. 
It is possible to study a model where rerouting creates excitation as well: this yields similar results as those found in Section~\ref{networksmarkov}.

At a departure in coordinate $j$, there is a jump $B_{ij}$ in each coordinate $i$, so that we have \textit{mutual excitation}. We let $\boldsymbol  B_j=[B_{1j}\cdots B_{dj}]^\top$ denote the vector of marks resulting from a departure in coordinate $j$. 

\section{Existence, uniqueness and stability} \label{existence}
In this section, we prove that there exists a unique stationary distribution for the process $\boldsymbol N(\cdot)$ from Definition~\ref{Def1condint} having nonlinear sojourn-time dependent excitation, and we state conditions under which a transient process satisfying the given dynamics is shown to converge to this stationary distribution. 
In contrast to classical Hawkes, at each arrival a random excitation function is drawn, whose distribution depends on the sojourn time realization. 
It suffices to consider a model with i.i.d.\ random excitation functions $h_{ij}$ having a distribution only depending on $(i,j)$; the $J$-dependent randomness of the form $\omega|J$ occurs as a special case of this general randomness. 

Let $(N_r,j_r,h_{1r},\ldots,h_{dr})_{r\in\mathbb Z}$ be the events of a random-function marked point process $\boldsymbol N(\cdot)$, where $N_r$ denotes the $r$-th event after time $0$ for $r\geq1$, and the $-(r+1)$ event before time $0$ for $r\leq0$; where $j_r$ denotes the coordinate in which this event occurred; and where $h_{ir}\sim h_{ij_r}$ denotes the excitation function for coordinate $i$ associated to the $r$-th arrival.
Let $(\Omega_{ij},\mathcal F_{ij},\mathbb Q_{ij})$ be the probability space on which $h_{ij}$ is defined.
Letting $\mathcal H_t^{\boldsymbol N}$ be the history of $\boldsymbol N(\cdot)$ up to time $t$, we assume that the model is driven by an $\mathcal H_t^{\boldsymbol N}$-progressively measurable intensity with $i$-th coordinate
 \begin{equation}\Lambda_i(t)=\phi_i\left(\sum_{j=1}^d\int_{(-\infty,t)\times\Omega_{ij}}h_{ij}(t-s,\omega_{ij})N_j(\mathrm ds\times\mathrm d\omega_{ij})\right)\label{intensity nonlinear},\end{equation}
with the understanding that the random functions $\omega_{ij}\mapsto h_{ij}(\cdot,\omega_{ij})$ are drawn independently with common distribution $h_{ij}$, for all $i,j\in[d]$. We assume that $\phi_i:\mathbb R\to\mathbb R_+$ and $h_{ij}:\mathbb R_+\to\mathbb R$, for all $i,j\in[d]$. In the linear case, which is the main focus of this paper, $\phi_i(x)=\lambda_{i,0}+x$ and $h_{ij}:\mathbb R_+\to\mathbb R_+$. In the univariate case, (\ref{intensity nonlinear}) reduces to
 \begin{equation}\Lambda(t)=\phi\left(\int_{(-\infty,t)\times\Omega}h(t-s,\omega)N(\mathrm ds\times\mathrm d\omega)\right)\label{intensity nonlinear uni},\end{equation}
 where $(\Omega,\mathcal F,\mathbb Q)$ is the probability space on which the random functions $\omega\mapsto h(\cdot,\omega)$ are defined.
 
We construct an adapted point process: $N$ on $\mathbb R\times\Omega$ with intensity $\Lambda(t)\mathbb Q(\mathrm d\omega)$ in the univariate case, and $\boldsymbol N$ on $\prod_{i=1}^d\left(\mathbb R\times\prod_{j=1}^d\Omega_{ji}\right)$ with intensity $\Lambda_i(t)\prod_{j=1}^d\mathbb Q_{ji}(\mathrm d\omega_{ji})$ in coordinate $i$ in the multivariate case. 
In Appendix~A, we present a proof for existence, uniqueness and stability of the univariate process having dynamics \eqref{intensity nonlinear uni}, 
leveraging the classical Picard proof for the existence of solutions to a differential equation,
following the approach of \cite{DaleyVereJones}, \S14.3 and \cite{Stabilitypaper}, Theorem~1 and using ideas from \cite{Massoulie}.
From this, the multivariate results can be proved along the lines of \cite{Stabilitypaper}, Theorem~7, taking the randomness of the excitation functions into account in the same fashion as we do in the univariate case. 
 
The conditional intensity specification \eqref{intensity nonlinear uni} deals with i.i.d.\ random excitation functions, which can be seen to exist by invoking the Kolmogorov extension theorem. 
However, for a single random function, this construction only enables us to say something about the behavior of the function on a countable subset of $\mathbb R$, but in general this does not allow us to conclude anything about sample-path properties, such as measurability. 
To tackle this problem, we make additional assumptions on the generic random function $h(\cdot)$.

\begin{definition}\label{separabledef}
A random function $h$ is called \textit{separable} with respect to a class $\mathcal K$ of subsets of $\mathbb R$ if there exists a countable subset $C\subset\mathbb R$ such that for each $K\in\mathcal K$ and each open interval $I\subset\mathbb R$ it holds that 
\begin{equation*}
\bigcap_{t\in I\cap C}\{h(t)\in K\}=\bigcap_{t\in I}\{h(t)\in K\},\quad\text{a.s.}
\end{equation*}
\end{definition}

\noindent We typically assume that the random function $h$ is a.s.\ piecewise continuous. 
In that case, $h$ is separable with respect to the class of open subsets of $\mathbb R$, taking $C$ to be any countable dense subset of $\mathbb R$, e.g., the set of rational numbers. 
By \cite{Neveu}, \S III.4, measurability of $(\omega,t)\mapsto h_\omega(t)$ can then be ensured. 
The feasibility of such a construction essentially comes down to the separability of the range space of the excitation functions.

In the following, we construct a univariate process having dynamics \eqref{intensity nonlinear uni} upon a basis consisting of a bivariate Poisson process of unit rate marked by random functions $h(\cdot,\omega)$, for which we use Lemma~\ref{Poisson construction} below. 
This is a well-known result underlying many simulation algorithms of point processes driven by conditional intensities, see e.g., \cite{OgataLewis}. 
Compare \cite{Massoulie}, Lemma~1.
To state this lemma, we define the \textit{left-shift operator} $S_t$, $t\in\mathbb R$. 
For a univariate stochastic process $X$, we set $S_tX(A)=X(A+t)$, for all $A\in\mathcal B(\mathbb R)$, with $\mathcal B$ the $\sigma$-algebra of Borel sets. 
Furthermore, we set $$X_+=\{X(A):A\in\mathcal B([0,\infty))\},\quad X_-=\{X(A):A\in\mathcal B((-\infty,0])\}.$$ 
With this notation, $S_tX_\pm$ can be interpreted as the future/history at time $t$. 
For a multivariate stochastic process, we assume that this shift is done with respect to the first variable, which is to be interpreted as time. 
In particular, the \textit{history} at time $t$ of a $(d+1)$-dimensional process $Y$ is given by $S_tY_-:=\{Y(A):A\in\mathcal B((-\infty,t]\times\mathbb R^{d})\}$.

\begin{lemma}\label{Poisson construction}
Let $M$ be a marked Poisson process on $\mathbb R\times\mathbb R_+\times\Omega$ with intensity $\mathrm dt\times\mathrm ds\times\mathbb Q(\mathrm dz)$, where the marks are defined on $(\Omega,\mathcal F,\mathbb Q)$. 
Let $\mathcal H_t^M$ be a sigma-algebra containing the history of $M$ at time $t$, such that $\mathcal H_s^M$ is independent of $S_tM_+$ for $s<t$. 
For some $\mathcal H_t^M$-predictable process $\Lambda(\cdot)$, define \begin{equation}N(A\times B)=\int_{A\times\mathbb R_+\times B}\mathbf1_{[0,\Lambda(t)]}(s)M(\mathrm dt\times\mathrm ds\times\mathrm dz),\quad A\times B\in\mathcal B(\mathbb R)\otimes\mathcal F.\end{equation} Then $N$ admits $\Lambda(t)\mathbb Q(\mathrm dz)$ as an $\mathcal H_t^M$-intensity.
\end{lemma}

The (lengthy) proof of the next result is postponed until Appendix~\hyperref[app A]{A}.

\begin{theorem}[Existence, uniqueness and stability]\label{stabilityth uni}
Assume that $\phi_i:\mathbb R\to\mathbb R_+$ is $L_i$-Lipschitz for all $i\in[d]$. Suppose that for all $i,j\in[d]$, $h_{ij}(\cdot,\omega)$ is a random function defined on $(\Omega_{ij},\mathcal F_{ij},\mathbb Q_{ij})$, which is separable with respect to the class of open sets, such that $h_{ij}(t)\in L^1(\mathbb Q_{ij})$ for almost all $t\in\mathbb R$, and such that the $d\times d$ matrix $\|\boldsymbol H\|:=(L_i\|\mathbb E|h_{ij}|\|_{L^1})_{i,j\in[d]}$ has spectral radius less than $1$. Then there exists a stationary distribution for a process $\boldsymbol N(\cdot)$ satisfying the dynamics \eqref{intensity nonlinear}. 

In addition, assume that $\|\mathbb E|h_{ij}|\|_\infty<\infty$ for all $i,j\in[d]$. Then this stationary distribution is unique. 
Let 
\begin{equation}
i_c(t)=\sum_{i,j\in[d]}\mathbb E_{h_{ij}}\left[\int_{t-c}^t\int_{(-\infty,0)\times\Omega_{ij}}|h_{ij}(s-\tau,\omega_{ij})|\ N_j(\mathrm d\tau\times\mathrm d\omega_{ij})\ \mathrm ds\right]\label{initial condition2}.
\end{equation} 
Let $M^d$ be a multivariate version of the marked Poisson process from Lemma~\ref{Poisson construction}, i.e., a marked Poisson process on $\mathbb R\times\mathbb R^d_+\times\prod_{i=1}^d\prod_{j=1}^d\Omega_{ji}$ with intensity $\mathrm dt\times\prod_{i=1}^d\left(\mathrm ds_i\times\prod_{j=1}^d\mathbb Q_{ji}(\mathrm d\omega_{ji})\right)$. 
Suppose that $\boldsymbol{N}$ is defined w.r.t.\ $M^d$. 
If (i) for all $c>0$, $\sup_{t\geq0}i_c(t)<\infty$ and $\lim_{t\to\infty}i_c(t)=0$, a.s., or (ii) for all $c>0$,  $\sup_{t\geq0}\mathbb E_{M^d}i_c(t)<\infty$ and $\lim_{t\to\infty}\mathbb E_{M^d}i_c(t)=0$, a.s., then for any $\tilde{\boldsymbol N}$ satisfying \eqref{initial condition2} with dynamics (\ref{intensity nonlinear}) on $\mathbb R_+$, we have $S_t\tilde {\boldsymbol N}_+\stackrel{\mathcal D}\to \boldsymbol N_+$, as $t\to\infty$; i.e., we have stability in distribution.
\end{theorem}

\begin{remark}
Both initial conditions (i) and (ii) say, in different ways, that the influence of the history at time $0$, i.e., the behavior on $(-\infty,0]$, on the future at time $t$, i.e., the behavior on $[t,\infty)$, vanishes, as $t\to\infty$. 
\end{remark}

\begin{remark}
Whereas existence, uniqueness and stability for the three specific processes given in Eqn.\ \eqref{eq:general} is, in principle, already implied by \cite{Massoulie}, our Theorem~\ref{stabilityth uni} above is more explicit. 
To apply \cite{Massoulie}, consider, for example, the univariate delayed Hawkes process. We can define a point process on $\mathbb R\times\mathbb R_+\times\mathbb R_+$, where the coordinates represent time, marks and sojourn times, respectively. 
Then the conditional intensity can be written as $$\Lambda(t,\mathrm db,\mathrm dw)=\psi(S_tN_-)B(\mathrm d b)J(\mathrm dw),$$ where $B$ denotes the mark distribution, $J$ denotes the sojourn time distribution, and $$\psi(S_tN_-)=\phi\left(\int_{(-\infty,t)\times\mathbb R_+\times\mathbb R_+}b h(t-s-w)\mathbf1_{[w,\infty)}(t-s)N(\mathrm ds\times\mathrm db\times\mathrm dw)\right).$$ 
Theorem~\ref{stabilityth uni} may then be compared to \cite[Theorems~2, 4]{Massoulie}: it gives more concrete conditions on $h(\cdot,\omega)$, and allows for a direct proof.
\end{remark}



\section{Scaling limit with stretched sojourn times} \label{scalinglimit}

Having formally introduced the general family of models having sojourn-time dependent excitation, we ask ourselves to what extent members of this family differ, statistically and probabilistically, from the classical Hawkes process. 
In fact, it turns out to be possible to distinguish between a Hawkes and a delayed Hawkes process from observed sample paths using statistical techniques from \cite{GoFpaper}, as we outline in Appendix \hyperref[app B]{B}.

In this section, we approach the problem of distinguishing between a Hawkes process and a delayed Hawkes process probabilistically, by analyzing asymptotic behavior through \textit{scaling limits}.
That is, we look for convergence at process level of some scaled version of the process. 
This convergence is weakly in $D[0,1]$, the space of càdlàg functions on the unit interval, equipped with the Skorokhod $J_1$-topology. 
We consider the linear, univariate case.
In a typical scaling regime, one considers the compensated counting process (see \cite{DaleyVereJones}, \S7.2); one contracts time by a factor $T$; after which one divides by $\sqrt T$.  
This is the quantity studied in a \textit{functional central limit theorem} (FCLT). 

For the unmarked Hawkes process, a scaling limit of this type can be found in  \cite{Bacry}. 
For a univariate model with immigration intensity $\lambda_0$ and excitation function $h$, their results imply that 
\begin{equation}\frac{N(T\cdot)-\mu T\cdot}{\sqrt T}\to\sigma B(\cdot),\label{SLBacry1}
\end{equation} as $T\to\infty$, weakly on $D[0,1]$ equipped with the Skorokhod $J_1$-topology, where $B$ is a standard Brownian motion, and where 
\begin{equation}\mu=\frac{\lambda_0}{1-\|h\|_{L^1}},\quad\sigma^2=\frac{\lambda_0}{(1-\|h\|_{L^1})^3}.\label{SLBacry2}
\end{equation}

On the other hand, for a model having sojourn-time dependent excitation, we can apply an existing FCLT for marked Hawkes random measures, as given in \cite{Horst}, Theorem~3.12. 
Since the scaling limit considers the counting process instead of the population process, we can, as in Section~\ref{existence}, replace the sojourn-time dependency of the random excitation function by general randomness. 
Letting $\mathbb U$ be a Lusin space modeling the randomness of the excitation functions, we use marks $\omega\in\mathbb U$ and excitation functions $h(t,\omega)$. 
It follows from \cite{Horst}, Theorem~3.12, that any two processes with random excitation functions \textit{having the same expected $L^1$-norm} admit the same scaling limit of the FCLT type (i.e., take a compensated process; contract time by a factor $T$; divide by $\sqrt T$). 

In particular, we can compare a Hawkes process to a delayed Hawkes process having the same parameters, corresponding to bivariate marks $\boldsymbol\xi\in\mathbb R_+^2$ whose coordinates represent `actual' mark and sojourn time, respectively, and excitation functions 
$$\phi_{\text{Hawkes}}(t,\boldsymbol \xi)=\xi_1h(t),\quad\mathrm{and}\quad\phi_{\text{delayed}}(t,\boldsymbol \xi)=\xi_1h(t-\xi_2)\mathbf1\{t\geq\xi_2\},$$ 
to infer that they admit the same scaling limit, being the sum of a Gaussian white noise (contributed by the marks) and a correlated Brownian motion, having the same parameters for both models. 
Heuristically, if we contract time, deviations from the mean from the random excitation functions cancel each other out. 
For the delayed Hawkes process, if sojourn times stay the same, but if we contract time by a factor $T$, the delays are of length $J/T$, hence vanish, as $T\to\infty$.


A natural, subsequent question is whether the difference between two processes belonging to the family of processes having sojourn-time dependent excitation can be made visible in some scaling limit. 
To this end, we consider a univariate unmarked delayed Hawkes process with i.i.d.\ sojourn times $(J_i)_i\in\mathbb N$, which we compare to its nondelayed counterpart.
The idea is to consider the compensated process on an interval $[0,T]$ with sojourn times stretched out from $J_i$ to $T^{\alpha}J_i$, for some $\alpha\in[0,1)$, after which we contract time by a factor of $T$, mapping $[0,T]$ onto $[0,1]$. After rescaling our counting process by $T^{-1/2}$ and letting $T\to\infty$, we obtain a nondegenerate limit. 
By taking a low degree of sojourn-time stretching, $0\leq\alpha<\frac12$, we obtain the same scaling limit as given by (\ref{SLBacry1})--(\ref{SLBacry2}), while if we set $\alpha=\frac12$, the effect of the delays becomes visible. 
The case $\alpha>\frac12$, discussed in Remark~\ref{otheralpha} below, is less transparent.

To obtain insight into this scaling limit, we modify the arguments from \cite{Bacry}. 
Let $(N_\alpha^T(v))_{v\in[0,1]}$ be equal to $(N(Tv))_{v\in[0,1]}$, for $N(\cdot)$ the process having sojourn times $(T^\alpha J_i)_{i\in\mathbb N}$. 
Those sojourn times correspond to the increasing sequence of arrival times $(\tau_i)_{i\in\mathbb N}$, where it is assumed that $J_i$ is drawn at time $\tau_i$. 
This process $N_\alpha^T(\cdot)$ has an arrival intensity $\Lambda_\alpha^T(\cdot)$ given by 
\begin{equation}
\Lambda_\alpha^T(v)=\left(\lambda_0+\sum_{\tau_i<t}h(Tv-\tau_i-T^\alpha J_i)\mathbf1\{Tv\geq \tau_i+T^\alpha J_i\}\right)\cdot T.
\end{equation}

To derive our scaling limit, we impose the following three assumptions. For $\alpha\leq\frac12$, \begin{align*}\|h\|_{L^1}&<1\tag{A1},\\\int_0^\infty t^{\frac1{2(1-\alpha)}}h(t)\ \mathrm dt&<\infty,\tag{A2}\\\mathbb E[J]&<\infty.\tag{A3}
\end{align*} 
We assume (A1)--(A2) throughout this section, while we only need (A3) for $\alpha=\frac12$.

In the following, we use the function $\bar h^T(\cdot)$, which can be seen as an average of $h$ over the past, weighed according to the stretched sojourn times: \begin{equation}\bar h^T(t):=\int_0^{T^{-\alpha}t}h(t-T^\alpha w) \ \mathrm d\bar{\mathscr J}(w).\end{equation}   
We also define \begin{equation}\mathscr H^T:=\sum_{k\geq1}(\bar h^T)^{*k},\end{equation} where $*k$ denotes $k$-fold convolution. 
In the sequel, we suppress the $\alpha$-dependence in the notations $\bar h^T$ and $\mathscr H^T$ to make our notation more compact; the value of $\alpha$ will be clear from the context.
Note that for any $T>0$, $\alpha\in[0,1)$,
\begin{align}
\nonumber\|\bar h^T\|_{L^1}&=\int_0^\infty\int_0^{T^{-\alpha}t}h(t-T^\alpha w) \ \mathrm d\bar{\mathscr J}(w)\ \mathrm dt=\int_0^\infty\int_{T^\alpha w}^\infty h(t-T^\alpha w) \ \mathrm dt\ \mathrm d\bar{\mathscr J}(w)\\&=\int_0^\infty\int_0^\infty h(t)\ \mathrm dt\ \mathrm d\bar{\mathscr J}(w)=\|h\|_{L^1},
\label{h switching order int}
\end{align} 
and therefore, using that $\|(\bar h^T)^{*k}\|_{L^1}=\|\bar h^T\|_{L^1}^k$, as can easily be proved by induction, 
$$\|\mathscr H^T\|_{L^1}=\sum_{k\geq1}\|(\bar h^T)^{*k}\|_{L^1}=\sum_{k\geq1}\|h\|^k_{L^1}=\frac{\|h\|_{L^1}}{1-\|h\|_{L^1}}<\infty.$$

For the next lemmas, we define a process $\tilde N^T_\alpha$, which is `simply' a delayed Hawkes process with sojourn times stretched out by a factor $T^\alpha$. 
We emphasize that we do \textit{not} contract time, yet.  
The next three lemmas can be seen as suitable counterparts of \cite{Bacry}, Lemmas~2, 4 and 5, respectively.
\begin{lemma}\label{BacryL2}
Let $\alpha\in[0,1)$, $T,t\geq0$. 
For each a.s.\ finite stopping time $S$, we have \begin{align}
\mathbb E[\tilde N_\alpha^T(S)]&=\lambda_0\mathbb E[S]+\mathbb E\left[\int_0^S\bar h^T(S-t)\tilde N_\alpha^T(t)\ \mathrm dt\right],\\\mathbb E[\tilde N_\alpha^T(S)]&\leq\mu\,\mathbb E[S].
\end{align}
\end{lemma}
\begin{proof}
The proof is a modification of the proof of \cite{Bacry}, Lemma~2. 
Their first display would read $$\mathbb E[\tilde N_{\alpha}^T(S_p)]=\lambda_0\mathbb E[S_p]+\mathbb E\left[\int_0^{S_p}\int_0^{T^{-\alpha }t}\int_0^{t-T^\alpha w}h(t-s-T^\alpha w)\ \mathrm d\tilde N_\alpha^T(s)\ \mathrm d\bar{\mathscr J}(w)\ \mathrm dt\right],$$ 
after which it comes down to performing calculations similar to the ones performed in Eqn.~\eqref{h switching order int}. 
\end{proof}
 
Now consider the martingale $\tilde M_\alpha^T(t)=\tilde N_\alpha^T(t)-\int_0^t\tilde\Lambda_\alpha^T(s)\ \mathrm ds$, where $\tilde\Lambda_\alpha^T$ denotes the arrival intensity of $\tilde N_\alpha^T$. 
The next lemma can be derived from Lemma~\ref{BacryL2} in the same way as \cite{Bacry}, Lemma~4 is derived from \cite{Bacry}, Lemma~2; we should replace their $\varphi$ by our $\bar h^T$ and their $\psi$ by our $\mathscr H^T$. 
 
\begin{lemma}\label{BacryL4}
Let $\alpha\in[0,1)$, $T,t\geq0$. Then it holds that 
\begin{align}
\mathbb E[\tilde N_\alpha^T(t)]&=\lambda_0t+\lambda_0\int_0^t\mathscr H^T(t-s)s\ \mathrm ds\label{Bacry13},\\
\tilde N_\alpha^T(t)-\mathbb  E[\tilde N_\alpha^T(t)]&=\tilde M_\alpha^T(t)+\int_0^t\mathscr H^T(t-s)\tilde M_\alpha^T(s)\ \mathrm ds.\label{Bacry14}
\end{align}
\end{lemma}

Define
\begin{align}\bar \sigma:= \frac{\lambda_0\mathbb E[J]\|h\|_{L^1}}{(1-\|h\|_{L^1})^2}.\end{align}

\begin{lemma}\label{BacryL5}
Let $\alpha\in[0,1)$, let $p\in[0,1]$ and assume that $\int_0^\infty t^ph(t)\ \mathrm dt<\infty$. Let $\epsilon\in(0,1)$. Then:
\begin{itemize}
\item If $p<1$, then $\displaystyle T^{(1-\alpha)p}\left(T^{-1}\mathbb E[N_\alpha^T(v)]-\mu v\right)\to0$,\: as $T\to\infty$, uniformly in $v\in[0,1]$.
\item If $p=1$, then $\displaystyle T^{1-\alpha}\left(T^{-1}\mathbb E[N_\alpha^T(v)]-\mu v\right)\to-\bar\sigma $,\: as $T\to\infty$, uniformly in $v\in[\epsilon,1]$.
\end{itemize}
\end{lemma}
\begin{proof}
First, we calculate 
\begin{align}
\int_0^\infty t^p\bar h^T(t)\ \mathrm dt&=\int_0^\infty\int_0^{tT^{-\alpha}} h(t-T^\alpha w)\ \mathrm d\bar{\mathscr J}(w)\ \mathrm dt=\int_0^\infty\int_{wT^\alpha}^\infty t^ph(t-T^\alpha w)\ \mathrm dt\ \mathrm d\bar{\mathscr J}(w)\nonumber\\
&=\int_0^\infty\int_0^\infty(t+T^\alpha w)^ph(t)\ \mathrm dt\ \mathrm d\bar{\mathscr J}(w)\leq\int_0^\infty\int_0^\infty(t^p+T^{\alpha p} w^p)h(t)\ \mathrm dt\ \mathrm d\bar{\mathscr J}(w)\nonumber\\
&\leq \int_0^\infty t^ph(t)\ \mathrm dt+T^{\alpha p}(1+\mathbb E[J])\|h\|_{L^1},\label{bound h^t}
\end{align} 
where for $p=1$ we find the equality 
\begin{equation}
\int_0^\infty t\bar h^T(t)\ \mathrm dt=\int_0^\infty th(t)\ \mathrm dt+T^\alpha\mathbb E[J]\|h\|_{L^1}.\label{equality H^t}
\end{equation}

Next, as in the proof of \cite{Bacry}, Lemma~5, we find that 
\begin{equation}
\int_0^\infty t^p\mathscr H^T(t)\ \mathrm dt\leq\frac{\int_0^\infty t^p\bar h^T(t)\ \mathrm dt}{(1-\|h\|_{L^1})^2},
\label{bound H^t}
\end{equation}
again with equality if $p=1$.

Now, consider a fixed $T>0$, and scale the process $\tilde N_\alpha^T$ from Lemma~\ref{BacryL4} to $N_\alpha^T$ by contracting time by a factor $T$. 
Using \eqref{Bacry13}, it now follows that 
\begin{align}
\frac{\lambda_0}{1-\|h\|_{L^1}}v-T^{-1}\mathbb E[N^T_\alpha(v)]=\lambda_0v\int_{Tv}^\infty\mathscr H^T(s)\ \mathrm ds+\lambda_0T^{-1}\int_0^{Tv}s\mathscr H^T(s)\ \mathrm ds. \label{scalinglimit difference}
\end{align}
We bound $T^{(1-\alpha)p}$ times the first term (ignoring $\lambda_0$) from \eqref{scalinglimit difference} by
\begin{align}
vT^{(1-\alpha)p}\int_{Tv}^\infty\mathscr H^T(s)\ \mathrm ds=v^{1-p}T^{-\alpha p}\int_{Tv}^\infty(Tv)^p\mathscr H^T(s)\ \mathrm ds\leq v^{1-p}T^{-\alpha p}\int_{Tv}^\infty s^p\mathscr H^T(s)\ \mathrm ds,
\end{align} 
which converges to $0$ as $T\to\infty$, by invoking the bounds found in \eqref{bound h^t} and in \eqref{bound H^t}. 
The convergence is uniform in $v\in[0,1]$ in case $p<1$, while the convergence is uniform on $[\epsilon,1]$ (for any $\epsilon\in(0,1)$) in case $p=1$. 

Next, we consider the second term from \eqref{scalinglimit difference}. 
Suppose first that $p<1$. 
Since $T^{(1-\alpha)p}$ times the second term can be bounded by 
\begin{equation}
\lambda_0T^{-(1-(1-\alpha)p)}\int_0^Ts\mathscr H^T(s)\ \mathrm ds,\label{second term bound}
\end{equation}
to prove convergence to $0$, uniformly in $v\in[0,1]$, 
it suffices to prove that \eqref{second term bound} converges to 0. 
This can be proved in the same way as in \cite{Bacry}, Lemma~5, applying integration by parts to $G(t)=\int_0^ts^{(1-\alpha)p}\mathscr H^T(s)\ \mathrm ds$.

Suppose now that $p=1$. 
Using \eqref{equality H^t} and \eqref{bound H^t}, it follows that 
\begin{equation}
T^{1-\alpha}T^{-1}\int_0^{Tv}s\mathscr H^T(s)\ \mathrm ds\to \frac{\mathbb E[J]\|h\|_{L^1}}{(1-\|h\|_{L^1})^2},
\end{equation} 
as $T\to\infty$, uniformly in $v\in[\epsilon,1]$. 
For the last limit, we really need $\alpha<1$; otherwise the stretching factors are of the same order as the limit of integration $Tv$. 
The result follows.
\end{proof}

We are now equipped to establish an FLLN for $\left(N_\alpha^T(\cdot)\right)_{T\geq0}$. 
We only state a version for $L^2(\mathbb P)$-convergence, since that is all we require to prove our FCLT. 
After the FLLN, we 
present our FCLT for $\alpha$-stretched sojourn times.

\begin{theorem}[FLLN]\label{FLLN}
Let $\alpha\in[0,1)$. It holds that $\tilde N_\alpha^T(t)\in L^2(\mathbb P)$, for all $T,t\geq0$, and we have \begin{equation}
\sup_{v\in[0,1]}\left|T^{-1}N_\alpha^T(v)-\mu\right|\to0\text{ in }L^2(\mathbb P)\text{ as }T\to\infty.
\end{equation}
\end{theorem}
\begin{proof}
The proof follows from similar arguments as the proof of \cite{Bacry}, Theorem~1, using Lemmas~\ref{BacryL4} and \ref{BacryL5} established above instead of \cite{Bacry}, Lemma~4 and \cite{Bacry}, Lemma~5, respectively, and using an analog of \cite{Bacry}, Lemma~6, which is easily seen to hold in our case as well. 
\end{proof}
\begin{theorem}[FCLT]\label{FCLT}
Let $\epsilon\in(0,1)$, and let $B$ be a standard Brownian motion. For $\alpha=\frac12$, we have
\begin{equation}
\left(\frac{N^T_{1/2}(v)-\mu\,Tv}{\sqrt T}\right)_{v\in[\epsilon,1]}\longrightarrow\left(-\bar\sigma +\sigma\,B(v)\right)_{v\in[\epsilon,1]},\label{FCLTresult}
\end{equation}
as $T\to\infty$, weakly on $D[\epsilon,1]$ equipped with the Skorokhod $J_1$-topology. 
On the other hand, for $\alpha\in[0,\frac12)$, it holds that 
\begin{equation}
\left(\frac{N^T_{\alpha}(v)-\mu\,Tv}{\sqrt T}\right)_{v\in[0,1]}\longrightarrow\left(\sigma\,B(v)\right)_{v\in[0,1]},\label{FCLTresultlowalpha}
\end{equation}
as $T\to\infty$, weakly on $D[0,1]$ equipped with the Skorokhod $J_1$-topology.
\end{theorem}
\begin{proof}
The proof is analogous to the one of \cite{Bacry}, Theorem~2, using an analog of \cite{Bacry}, Lemma~7, using Lemma~\ref{BacryL4} above instead of \cite{Bacry}, Lemma~4, and using Lemma~\ref{BacryL5} above with $p=\frac{1}{2}(1-\alpha)^{-1}$ instead of \cite{Bacry}, Lemma~5.
\end{proof}

For $\alpha=\frac12$, Theorem~\ref{FCLT} yields convergence on intervals of the form $[\epsilon,1]$, where $\epsilon>0$ can be taken arbitrarily small. 
This is in contrast to the case $\alpha\in[0,\frac12)$ and to \cite{Bacry}, Theorem 2, where we obtain convergence on the whole unit interval. 
For each $\alpha\in[0,\frac12]$, both the centralising constant $\mu$ and the Brownian term are the same. 
A notable difference is that in Theorem~\ref{FCLT} with $\alpha=\frac12$ there is a  `correction term'  $-\bar\sigma$ in the limit. 

We can explain this result \emph{heuristically}.
In the limiting result \eqref{FCLTresult}, we start observing the process at time $T\epsilon$, for fixed $\epsilon>0$. 
For large $T$, this means that the process approaches stationarity on $[0,T\epsilon)$. 
By Corollary~\ref{steadystateequaldist} below --- which covers the Markovian case --- and the heuristic explanation given thereafter, there is good reason to believe that Hawkes and delayed Hawkes processes have the same stationary distributions. 
Therefore, we expect to find similar limits. 
However, in \eqref{FCLTresult} delays were also stretched out by a factor of $T^{1/2}$, meaning that excitation takes more time to come into full effect, which causes $\mu\,T\cdot$ to overestimate the mean of $N^T_{1/2}(\cdot)$ on $[0,\epsilon)$. 
This is compensated for by the negative term $-\bar\sigma$ appearing in the limit. 

\begin{remark}\label{otheralpha}
Under (A1), for any $\alpha\in[0,1)$, it is possible to find a FCLT as in Theorem~\ref{FCLT}, stating that, as $T\to\infty$,
\begin{equation}
\left(\frac{N^T_{\alpha}(v)-\mathbb E[N^T_{\alpha}(v)]}{\sqrt T}\right)_{v\in[0,1]}\longrightarrow\left(\sigma\,B(v)\right)_{v\in[0,1]},
\end{equation}
weakly on $D[0,1]$ equipped with the Skorokhod $J_1$-topology. When $\alpha\in(\frac12,1)$, in contrast to the case  $\alpha\in[0,\frac12]$, we cannot use Lemma \ref{BacryL5} to replace $\mathbb E[N^T_{\alpha}(v)]$ in this expression.

For $\alpha=1$, if we take sojourn times having support on $[1,\infty)$, the excitation would not be visible, since in the scaling limit we observe the process on (a subset of) $[0,1]$. 
In this case, the unscaled process on $[0,T]$ would just be a homogeneous Poisson process of rate $\lambda_0$, for which an FCLT holds; e.g., use (\ref{SLBacry1})--(\ref{SLBacry2}) with $h\equiv0$. 
When $\alpha>1$, we would see the same behavior. 
The situation where $\alpha=1$ \emph{and} where the sojourn time attains values in $(0,1)$ with positive probability is more delicate. 
\end{remark}

\section{Transform analysis and heavy-tailed asymptotics}\label{chapternonmarkov}
In this section, we perform transform analysis for point processes having sojourn-time dependent excitation.
First, in Section~\ref{non-markov}, we use cluster-representation based methods to describe fixed points in the transform domain, after which, in Section~\ref{heavy-tails}, those fixed-point equations are used to derive heavy-tailed asymptotics. 
A supplement to this section can be found in Appendix~\hyperref[app A3]{D}, where we study cluster size distributions for gamma-distributed marks.

\subsection{Transform characterizations with sojourn-time dependent excitation}\label{non-markov}
In \cite{multivariateKLM}, multivariate non-Markovian Hawkes-fed birth-death processes were studied using
cluster-representation based  methods. 
In Definition~\ref{DHcluster2}, we gave a cluster representation for the $d$-dimensional birth-death process with sojourn-time dependent excitation, analogous to the one for the multivariate Hawkes-fed birth-death process. 
As it turns out, the cluster representation is the pivotal ingredient for the results from \cite{multivariateKLM}, \S3--4: Definition~\ref{DHcluster2} enables us to obtain analogous results for our general family of models having sojourn-time dependent excitation. 
The modifications needed in the respective proofs are relatively straightforward, and mostly come down to suitably replacing randomness of the form $B_{ij,\omega}h_{ij}$ by sojourn-time dependent randomness of the form $h_{ij,J,\omega}$. 
Therefore, to save space, we provide the proof of the next result in online Supplementary Material~\cite{Supplement}.

\begin{theorem}\label{summary KLM}
Consider the joint birth-death and intensity process $(\boldsymbol Q(t),\boldsymbol\Lambda(t))$ from Definition~\ref{DHcluster2}. 
Under the regularity conditions given there, the joint Z- and Laplace transform of $(\boldsymbol Q(t),\boldsymbol\Lambda(t))$ can be expressed as 
\begin{equation}\mathbb E\left[\boldsymbol z^{\boldsymbol Q(t)}e^{-\boldsymbol s^\top\boldsymbol \Lambda(t)}\right]=\prod_{j=1}^d\exp\left(-\lambda_{j,0}\left(t+s_j-\int_0^t\mathbb E\left[\boldsymbol z^{\boldsymbol S_j^{\boldsymbol Q}(u)}e^{-\boldsymbol s^\top\boldsymbol S^{\boldsymbol \Lambda}_j(u)}\right]\ \mathrm du\right)\right),\label{jointtansform to clustertransform}
\end{equation} 
where the cluster processes $\boldsymbol S^{\boldsymbol Q}_j(\cdot)$, $\boldsymbol S^{\boldsymbol \Lambda}_j(\cdot)$ are defined in Section~\ref{preliminaries}.

Combine the cluster processes for individual coordinates into a matrix $\boldsymbol S^{\boldsymbol\star}(\cdot)$ with $j$-th column $\boldsymbol S^{\boldsymbol \star}_j(\cdot)$, for $\boldsymbol\star\in\{\boldsymbol Q,\boldsymbol\Lambda\}$. 
Then the joint vector-valued transform $\boldsymbol{\mathcal J}_{\boldsymbol S^{\boldsymbol Q},\boldsymbol S^{\boldsymbol\Lambda}}(\cdot)$ of $\boldsymbol S^{\boldsymbol Q}(\cdot)$, $\boldsymbol S^{\boldsymbol \Lambda}(\cdot)$, which has as $j$-th component the joint transform 
\begin{align}
u\mapsto\mathbb E\left[\boldsymbol z^{\boldsymbol S_j^{\boldsymbol Q}(u)}e^{-\boldsymbol s^\top\boldsymbol S^{\boldsymbol \Lambda}_j(u)}\right],    
\end{align} 
is the unique point of $\phi$, which maps the space $\mathbb J^d$ of vector-valued $d$-dimensional joint Z- and Laplace transforms $\boldsymbol{\mathcal J}(\cdot)$ to itself, and is defined by \begin{equation}\boldsymbol{\mathcal J}(\cdot)=\begin{bmatrix}\mathcal J_1(\cdot)\\\vdots\\\mathcal J_d(\cdot)\end{bmatrix}\mapsto\begin{bmatrix}\phi_1(\mathcal J_1,\ldots,\mathcal J_d)(\cdot)\\\vdots\\\phi_d(\mathcal J_1,\ldots,\mathcal J_d)(\cdot)\end{bmatrix}=\begin{bmatrix}\phi_1(\boldsymbol{\mathcal J})(\cdot)\\\vdots\\\phi_d(\boldsymbol{\mathcal J})(\cdot)\end{bmatrix}=\phi(\boldsymbol{\mathcal J})(\cdot),\end{equation} where for $j\in[d]$ \begin{align}\phi_j(\boldsymbol{\mathcal J})(u)&\equiv\phi_j(\boldsymbol{\mathcal J})(u,\boldsymbol s,\boldsymbol z)\label{defphij}\\&=\mathbb E_{J,\omega}\left[z_j^{\mathbf1\{J>u\}}\prod_{i=1}^de^{-s_ih_{ij,J,\omega}(u)}\prod_{m=1}^d\exp\left(-\int_0^uh_{mj,J,\omega}(v)\left(1-\mathcal J_m(u-v,\boldsymbol s,\boldsymbol z)\right)\ \mathrm dv\right)\right].\nonumber\end{align}
Furthermore, for any $\boldsymbol{\mathcal J}^{(0)}(\cdot)\in\mathbb J^d$, the sequence $(\boldsymbol{\mathcal J}^{(n)}(u))_{n\in\mathbb N_0}$ of iterates of $\boldsymbol{\mathcal J}^{(0)}(\cdot)$ under $\phi$, defined inductively by $\boldsymbol{\mathcal J}^{(n)}(\cdot):=\phi(\boldsymbol{\mathcal J}^{(n-1)})(\cdot)$, converges pointwise on intervals $[0,t]$ to the fixed point
$\boldsymbol{\mathcal J}_{\boldsymbol S^{\boldsymbol Q},\boldsymbol S^{\boldsymbol \lambda}}(u)$. That is, as $n\to\infty$, for any $u\in[0,t]$, \begin{equation}\boldsymbol{\mathcal J}^{(n)}(u)\equiv\boldsymbol{\mathcal J}^{(n)}(u,\boldsymbol s,\boldsymbol z)\to\boldsymbol{\mathcal J}_{\boldsymbol S^{\boldsymbol Q},\boldsymbol S^{\boldsymbol \lambda}}(u,\boldsymbol s,\boldsymbol z)\equiv\boldsymbol{\mathcal J}_{\boldsymbol S^{\boldsymbol Q},\boldsymbol S^{\boldsymbol \lambda}}(u).\end{equation}
\end{theorem}


\begin{remark}
It is possible to generalize Theorem~\ref{summary KLM} to a feedforward network in which a particle in coordinate $j\in[d]$ is sent to coordinate $j+1$ after service. 
Here, it is understood that when $j=d$, the particle leaves the system after service. 
Letting $J_j,\ldots,J_d$ be the sojourn times of the components visited by the particle that arrived in component $j$, and assuming that the excitation function $h_{ij,J,\omega}$ is dependent on the \emph{total} time $J:=\sum_{l=j}^dJ_l$ spent in the system, we can obtain a result analogous to Theorem~\ref{summary KLM}. 
The operator appearing in the fixed-point equation now reads \begin{align*}\phi_j(\boldsymbol{\mathcal J})(u,\boldsymbol s,\boldsymbol z)=\mathop{\mathbb E}_{J_j,\ldots,J_d,\omega}\left[c(u)\prod_{m=1}^d\exp\left(-\int_0^uh_{mj,J,\omega}(v)\left(1-\mathcal J_m(u-v,\boldsymbol s,\boldsymbol z)\right)\ \mathrm dv\right)\right],\end{align*} where $$c(u)=\prod_{l=j}^dz_l^{\mathbf 1\left\{\sum_{m=j}^{l=1}J_m\leq u,\sum_{m=j}^lJ_m>u\right\}}\prod_{i=1}^de^{-s_ih_{ij,J,\omega}(u)}.$$
An analysis treating multiple parallel tandem systems, as conducted for shot-noise processes in \cite{shotnoise}, is hard in the non-Markovian (delayed) Hawkes case: in contrast to a network of shot-noise processes, the sample paths of parallel (delayed) Hawkes networks influence each other. In the Markovian case, however,  we are able to characterize the transform of \emph{any} irreducible $d$-dimensional network; see Section~\ref{networksmarkov}.
\end{remark}

\subsection{Heavy-tailed asymptotics} \label{heavy-tails}
In this subsection, we specify the non-Markovian model from Section~\ref{non-markov} to the one-dimensional delayed Hawkes case, so that the randomness in the excitation function is of the form $h_{J,\omega}(\cdot)=B_\omega h(\cdot-J)\mathbf1\{\cdot>J\}$. 
We show that if the marks $B_\omega$ are heavy-tailed --- in the sense of being regularly varying --- the birth-death process will be so as well. 
Our proof uses \eqref{jointtansform to clustertransform} and the fixed-point equation for the transform, \eqref{defphij}.

\begin{definition}
Let $\alpha>0$.
An a.s.\ positive random variable $X$ is called \textit{regularly varying} of index $-\alpha$ if\begin{equation}\mathbb P(X>x)=\ell(x)x^{-\alpha},\quad x\geq0,\end{equation} where $\ell$ is a slowly varying function at infinity, meaning that $\ell(\gamma x)\sim\ell(x)$ as $x\to\infty$, for all $\gamma>1$. We write $\mathscr R(-\alpha)$ for the class of regularly varying random variables of \textit{tail index} $\alpha$.
\end{definition}

We also use the stronger notion of asymptotically power-law tails.

\begin{definition}
An a.s.\ positive random variable $X$ is said to have an \textit{asymptotically power-law tail} (APT) if there exist $C>0$ and $\gamma>1$ such that \begin{equation}\mathbb P(X>x)x^\gamma\to C,\end{equation} as $x\to\infty$. 
In this case we write $X\in\mathrm{APT}(-\gamma)$ and we refer to $\gamma$ as the \textit{tail index}.
\end{definition}

The next result may be compared to \cite{Infinite server queues}, Theorem~6.2. 
Its (lengthy) proof is postponed until Appendix  \hyperref[app A5]{C}.

\begin{theorem}\label{heavytailsunith}
Consider the univariate delayed Hawkes birth-death process with general sojourn times. Assume the stability condition $\| h\|_{L^1}b_1<1$, where $b_1:=\mathbb E[B]$. 
Suppose that $B\in\mathscr R(-\alpha)$ with $\alpha\in(1,2)$. 
Then also $Q(t)\in\mathscr R(-\alpha)$.
\end{theorem}
\renewcommand{\theenumi}{\roman{enumi}}
\begin{remark}
Theorem~\ref{heavytailsunith} admits various extensions.
\begin{enumerate} 
\item We can take sojourn-time dependent marks, i.e., $h_{J,\omega}(\cdot)=B_{J,\omega}h(\cdot-J)\mathbf1\{\cdot>J\}$. 
Suppose that $B|J$ is either light-tailed, or regularly varying of index $-\alpha$ for some $\alpha>1$, $\bar{\mathscr J}$-a.s., in such a way that the infimum of the $\alpha$ for which $B_w\in\mathscr R(-\alpha)$, lies in $(1,2)$, and is attained with positive $\bar{\mathscr J}$-probability. 
Expanding $\beta_w:=\mathbb E[e^{-sB}\,|\,J=w]$  in (\ref{fixedpointeta}) using the Tauberian theorem for $w$ such that $B|J=w$ is regularly varying, and using a Taylor expansion for other $w$, we obtain an equivalent of (\ref{1-eta asymp}), after which we proceed as in the proof of Theorem~\ref{heavytailsunith}.
\item If $\alpha\in(k,k+1)$, $k\in\{2,3,\ldots\}$, the Tauberian theorem for a higher-order expansion yields a more involved, but conceptually analogous, proof for $Q(t)\in\mathscr R(-\alpha)$.
\item Theorem~\ref{heavytailsunith} admits a multivariate generalization, by following the arguments from \cite{multivariateKLM}, \S5. 
\item A proof analogous to the proof of Theorem~\ref{heavytailsunith} shows that if we have regularly varying marks, those marks propagate to the intensity $\Lambda(t)$ as well.
\end{enumerate}
\end{remark}


The following corollary describes heavy-traffic behavior in the heavy-tailed setting; its proof is in Appendix \hyperref[app A5]{C}.

\begin{corollary}\label{heavytailstraffic}
Assume that we are in the heavy-tailed setting of Theorem~\ref{heavytailsunith}, with $B\in\mathrm{APT}(-\alpha)$ for some $\alpha\in(1,2)$. 
Let $\rho=\| h\|_{L^1}b_1<1$, and write $(Q,\Lambda)$ for the stationary distribution of $(Q(\cdot),\Lambda(\cdot))$. 
Then it holds that $(1-\rho)Q$ converges in distribution to some nondegenerate, nondefective random variable $X$ with $\mathbb E\left[X^\alpha\right]=\infty$, as $\rho\uparrow1$.
\end{corollary}

\section{Comparisons using stochastic ordering} \label{comparisons}

In this section, we consider a multivariate Hawkes-fed birth-death process $(\boldsymbol  N(t),\boldsymbol  Q(t),\boldsymbol\Lambda(t))_{t\in\mathbb R_+}$ with intensity $\Lambda_i(\cdot)$ in component $i$ given by \begin{equation}\Lambda_i(t)=\lambda_{i,0}+\sum_{j=1}^d\int_{-\infty}^tB_{ij}(s)h_{ij}(t-s)\ \mathrm dN_j(s),\label{condintH}\end{equation} where, for each $i,j\in[d]$, $(B_{ij}(s))_{s\in\mathbb R}$ is a collection of cross-sectionally and serially independently distributed random variables distributed as the a.s.\ positive random variable $B_{ij}$. 
We compare this process to the corresponding multivariate delayed Hawkes birth-death process $(\tilde{\boldsymbol  N}(t),\tilde{\boldsymbol  Q}(t),\tilde{\boldsymbol\Lambda}(t))_{t\in\mathbb R_+}$ having the same parameters; 
its intensity $\tilde\Lambda_i(\cdot)$ in coordinate $i$ is given by \begin{equation}\tilde\Lambda_i(t)=\lambda_{i,0}+\sum_{j=1}^d\int_{-\infty}^tB_{ij}(s)h_{ij}(t-s)\ \mathrm d\tilde D_j(s),\label{condintDH}\end{equation} where $\tilde D_j(\cdot)$ denotes the departure process of the $j$-th coordinate. 

In the univariate case, both systems can be specified through conditional intensities of the form 
\begin{equation}\Lambda(t)=\lambda_0+\sum_{t_i< t}B_ih(t-t_i),\label{intensityDH/H}
\end{equation} 
the only difference being that 
in the former, classical case $(t_i)_{i\in\mathbb N}$ denote arrival times for the Hawkes-fed birth-death process, whereas in the latter, delayed case $(t_i)_{i\in\mathbb N}$ denote departure times for the delayed Hawkes birth-death process. 
To argue that the delayed Hawkes process is in a sense `dominated' by the Hawkes process, we consider a comparison using stochastic ordering. 
\begin{definition}
Let $X,Y$ be random variables. 
We say that $X$ is larger than $Y$ in the stochastic order, or, equivalently, that $X$ dominates $Y$, if $F_X(z)\leq F_Y(z)$ for all $z\in\mathbb R$. 
We write $X\geq_{\mathrm{st}}Y$.
\end{definition}
Another way of representing both birth-death processes is by considering their respective cluster process representations, as given via Definition~\ref{DHcluster2}. 
Then, by an obvious coupling, the baseline intensity $\boldsymbol\lambda_0$ generates the same stream of immigrants, and therefore the same multiplicity of clusters, for both birth-death processes. 
Coupling those clusters as well, it is clear that $\boldsymbol N(t)\geq_{\mathrm{st}}\tilde{\boldsymbol N}(t)$ for all $t\geq0$, since the clusters produce the same offspring for both processes, but $k$-th generation children are counted $k$ lifetimes later for $\tilde N(t)$ than for $N(t)$.
This observation immediately raises the question whether we can compare $\boldsymbol Q(t)$ to $\tilde{\boldsymbol Q}(t)$ and $\boldsymbol\Lambda(t)$ to $\tilde{\boldsymbol \Lambda}(t)$ in the stochastic order as well. 
This question is answered affirmatively by the following theorem.


\begin{theorem}\label{thstochasticdomination}
Let $(\boldsymbol  N(t),\boldsymbol  Q(t),\boldsymbol\Lambda(t))_{t\in\mathbb R_+}$  be a multivariate Hawkes-fed birth-death process, with conditional intensity given by \eqref{condintH}. 
Furthermore, let  $(\tilde{\boldsymbol  N}(t),\tilde{\boldsymbol  Q}(t),\tilde{\boldsymbol\Lambda}(t))_{t\in\mathbb R_+}$ denote the corresponding delayed Hawkes birth-death process having the same parameters, i.e., its conditional intensity satisfies \eqref{condintDH}. 
Also assume both systems have the same sojourn time distributions $J_i$, having CDF $\bar{\mathscr J}_i$.
For both systems, let the $(B_{ij}(s))_{s\geq0}$ be i.i.d., independent of other random variables driving the processes.  
In both cases, suppose that we start in an empty system with zero arrivals, and a conditional intensity equal to the baseline intensity $\boldsymbol\lambda_0$. 
Then we have for all $j\in[d]$ and for all $t\geq0$, $N_j(t)\geq_{\mathrm{st}}\tilde N_j(t)$, $Q_j(t)\geq_{\mathrm{st}}\tilde Q_j(t)$ and $\Lambda_j(t)\geq_{\mathrm{st}}\tilde \Lambda_j(t)$.
\end{theorem}

\begin{proof}
We first prove the result for univariate processes, after which we extend the arguments to multivariate processes.
The proof relies on the cluster representation as given in Definition~\ref{DHcluster2} with excitation functions specified in Eqn.~\ref{eq:general}.

\textit{Univariate case}. Consider the conditional intensity  processes $\Lambda(\cdot)$ and $\tilde\Lambda(\cdot)$. 
For Hawkes, set $$\Lambda_{k\mathrm G}(t))=\begin{cases}
    \lambda_0,&\text{ if }k=0,\\
    \displaystyle\sum_{t_i<t\text{ of generation }k-1}B_ih(t-t_i),&\text{ if }k\in\mathbb N.
\end{cases}$$ 
Define $\tilde \Lambda_{kG}$ similarly for delayed Hawkes.
In the following, we consider the \textit{a priori}  
arrival intensity processes of $k$-th generation offspring, $\mathbb E[\Lambda_{kG}|\mathcal H_0]$ and $\mathbb E[\tilde\Lambda_{kG}|\mathcal H_0]$. Note that $\Lambda_{0\mathrm G}(t)=\lambda_0=\tilde\Lambda_{0\mathrm G}(t)$. 

We say that a cluster \textit{starts} when the excitation starts; for Hawkes this is at the birth of a particle, for delayed Hawkes at expiration of a particle.
This means that for the delayed Hawkes process, at time $t$, \textit{starting} clusters arrive at rate $\int_0^t\lambda_0\ \mathrm d\bar{\mathscr J}(s)=\lambda_0\bar{\mathscr J}(t)\leq\lambda_0$.
Let $(\Omega,\mathcal F,\mathbb Q)=(\mathbb R_+^2,\mathcal F_B\otimes\mathcal F_J,\mathcal Q\otimes\bar{\mathscr J})$ be the probability space on which the marks $B$ and lifetimes $J$ are defined jointly.
For $k\geq0$, define the $k$-th cluster of a delayed Hawkes process $\tilde N$
recursively w.r.t.\ i.i.d.\ Poisson random measures (PRMs) $(M_k)_{k\in\mathbb N_0}$ on $\mathbb R\times\mathbb R_+\times\Omega$ with intensity $\mathrm dt\times\mathrm ds\times\left(\mathcal Q(dz)\otimes\mathrm d\bar{\mathscr J}(w)\right)$ by 
\begin{align}
     \tilde N_{kG}(A\times B)&=\int_{A\times\mathbb R_+\times B}\mathbf1_{[0,\tilde\Lambda_{kG}(t)]}(s) \ M_k(\mathrm dt\times\mathrm ds\times d(z,w)),\quad A\times B\in\mathcal B(\mathbb R)\otimes\mathcal F.\label{kth g DH}
\end{align}

Couple a fraction $\bar{\mathscr J}(t)$ of Hawkes clusters to delayed Hawkes clusters starting at the same time:
\begin{align}
    N_{0G}(A\times B)&=\int_{A\times B}\mathbf1_A(u+w) \ \tilde N_{0G}(\mathrm du\times\mathrm d(z\times w))\tag{$\mathrm{Poi}(\bar{\mathscr J}(t))$ \text{ stream}}
    \nonumber\\&+\int_{A\times\mathbb R_+\times B}\mathbf1_{[0,(1-\bar{\mathscr J}(t))\Lambda_{0G}(t)]}(s) \ M'_0(\mathrm dt\times\mathrm ds\times d(z,w))\nonumber\\&=:N_{0G,1}+N_{0G,2}, \quad A\times B\in\mathcal B(\mathbb R)\otimes\mathcal F,\label{0th g H}
\end{align}
where $M_k', k\in\mathbb N_0$ are independent PRMs with the same distribution as $M_0$. 
Since the second term of \eqref{0th g H} is nonnegative, the `above-baseline' intensity caused by immigrant arrivals of the Hawkes process stochastically dominates that of the delayed Hawkes process: $\Lambda_{1\mathrm G}(t)\geq_{\mathrm{st}}\tilde\Lambda_{1\mathrm G}(t)$.

We now consider the arrivals of subclusters: for the Hawkes process this happens at arrivals of first-generation offspring, while for the delayed Hawkes process this happens when first-generation offspring leaves the system. 
We note that the \textit{a priori} expected arrival rate for first-generation offspring increases over time, since for such an arrival we have to go through multiple stages: immigrant arrival, sojourn time $J$ (only for delayed Hawkes), and arrival triggered by excitation caused by an immigrant arrival; here, we use that we start from an empty system. 

The arrival intensity of starting second-generation clusters for the delayed Hawkes process equals the arrival rate of first-generation offspring convoluted with $\bar{\mathscr J}$. 
Since 
$\mathbb E[\tilde\Lambda_{1\mathrm G}(t)|\mathcal H_0]$ is increasing and since the convolution averages over the past, it follows that the expected arrival rate of starting subclusters for the delayed Hawkes process is dominated by the expected arrival rate of first-generation offspring (i.e., starting subclusters) for the Hawkes process resulting from the immigrants $N_{0G,1}$; 
denote the ratio between the two at time $t$ by $\vartheta(t)\in[0,1]$. 
In \eqref{0th g H}, we coupled a fraction of Hawkes clusters to delayed Hawkes clusters starting at the same time through 
$N_{0G,1}$. 
Denote the increase in intensity for the Hawkes process resulting from the immigrants $N_{0G,1}$ by $\Lambda_{0G,1}$. 
Refine the previous coupling by coupling a fraction of starting subclusters resulting from the particles $N_{0G,1}$ for the Hawkes process to delayed Hawkes subclusters starting at the same time, through 
\begin{align}
    N_{1G,0}(A\times B)&=\int_{A\times B}\mathbf1_A(u+w) \ \tilde N_{1G}(\mathrm du\times\mathrm d(z\times w))\tag{$\mathrm{Poi}(\vartheta(t))$ \text{ stream}}
    \nonumber\\&+\int_{A\times\mathbb R_+\times B}\mathbf1_{[0,(1-\vartheta(t))\Lambda_{1G,1}(t)]}(s)\ M'_1(\mathrm dt\times\mathrm ds\times d(z,w))\nonumber\\&=:N_{1G,1}+N_{1G,2}, \quad A\times B\in\mathcal B(\mathbb R)\otimes\mathcal F.\label{1th g H}
\end{align}
As $N_{1G,1}$ is coupled to the stream of starting second-generating clusters for delayed Hawkes, we conclude that $\Lambda_{2\mathrm G}(t)\geq_{\mathrm{st}}\tilde\Lambda_{2\mathrm G}(t)$. 

The argument of the previous paragraph can be repeated inductively for any $k\in\mathbb N$, obtaining $\Lambda_{k\mathrm G}(t)\geq_{\mathrm{st}}\tilde\Lambda_{k\mathrm G}(t)$ for all $k\in\mathbb N$. 
In any step, our coupling of $k$-th generation starting subclusters is a refinement of the previous coupling, and uses the \emph{genealogical order}. 
In the above construction, we coupled clusters, subclusters, subsubclusters, etc., and by the independency structure inherent in the cluster representation it follows that $\sum_{k=0}^n\Lambda_{k\mathrm G}(t)\geq_{\mathrm{st}}\sum_{k=0}^n\tilde\Lambda_{k\mathrm G}(t)$ for all $n\geq0$.

Note that $\sum_{k=0}^n\Lambda_{k\mathrm G}(t)\stackrel{\mathcal D}\to\Lambda(t)$ as $n\to\infty$, for all $t\geq0$, and similarly for $\tilde\Lambda(t)$. 
Hence, for every continuity point $x$ of $F_{\Lambda(t)}$, 
\begin{align}
  \lim_{n\to\infty} F_{\sum_{k=0}^n\Lambda_{k\mathrm G}(t)}(x) =  F_{\Lambda(t)}(x),\notag
\end{align}
and similarly for every continuity point $x$ of $F_{\tilde\Lambda(t)}$, \[F_{\sum_{k=0}^n\tilde\Lambda_{k\mathrm G}(t)}(x)\to F_{\tilde\Lambda(t)}(x).\] 
Since any distribution function has at most countably many discontinuities, and using the fact that $F_{\sum_{k=0}^n\Lambda_{k\mathrm G}(t)}(x)\leq F_{\sum_{k=0}^n\tilde\Lambda_{k\mathrm G}(t)}(x)$ for all $x\in\mathbb R$ and $n\geq0$ by the stochastic ordering we established above, it immediately follows that $F_{\Lambda(t)}(x)\leq F_{\tilde\Lambda(t)}(x)$ for all but at most countably many $x$. 
By right-continuity of distribution functions, if this inequality does not hold at $z$, it does not hold for a continuum of values $x\in[z,z+\epsilon]$. 
Hence, this inequality holds for all $x\in\mathbb R$, and we conclude that $\Lambda(t)\geq_{\mathrm{st}}\tilde \Lambda(t)$. 

From this, we can decompose the conditional intensity $\Lambda(\cdot)$ of a Hawkes process as the sum of the intensity $\tilde\Lambda(\cdot)$ of a delayed Hawkes process with the same parameters, and the nonnegative process $(\Lambda-\tilde\Lambda)(\cdot)$ consisting of the remaining intensity. 
By coupling arrivals and setting sojourn times equal, it follows that $Q(t)\geq_{\mathrm{st}}\tilde Q(t)$ and $N(t)\geq_{\mathrm{st}}\tilde N(t)$, as claimed.

\textit{Multivariate case}.
Suppose that an immigrant in coordinate $i_0$ produces offspring in coordinate $i_1$, which in turn produces offspring in coordinate $i_2$, and so on, until there is a child in coordinate $i_n$. 
Write $i_0i_1\ldots i_n$ for the path indicating this order of visited coordinates. 
By analogy to the univariate case, let $\Lambda_{i_0i_1\ldots i_n}$ and $\tilde\Lambda_{i_0i_1\ldots i_n}$ be the \textit{a priori} arrival intensities of $n$-th generation offspring in coordinate $i_n$ through the order $i_0i_1\ldots i_n$ for the Hawkes and the delayed Hawkes process, respectively. 
As in the univariate case, it can be argued that $\Lambda_{i_0i_1\ldots i_n}(t)\geq_{\mathrm{st}}\tilde\Lambda_{i_0i_1\ldots i_n}(t)$ for each such path $i_0i_1\ldots i_n$, where couplings can be chosen as refinements of the couplings for the path $i_0i_1\ldots i_{n-1}$. By using the conditional independency structure inherent in the cluster representation and by summing over all possible paths $i_0i_1\ldots i_n$, $n\in\mathbb N_0$, $i_k\in[d]$, it follows that for each $j\in[d]$, $\Lambda_j(t)\geq_{\mathrm{st}}\tilde \Lambda_j(t)$. By coupling arrivals and setting sojourn times equal, the other claims follow.
\end{proof}

We conclude this section by considering two univariate delayed Hawkes birth-death processes having different parameters that dominate each other, and indicate when one process dominates the other.
Denote those delayed Hawkes birth-death processes by $(N^{(j)}(\cdot),Q^{(j)}(\cdot),\Lambda^{(j)}(\cdot))$, $j=1,2$, in which we have arrivals generated by conditional intensities of the form \begin{equation}\Lambda^{(j)}(t)=\lambda_{0}^{(j)}+\sum_{t_{i}^{(j)}< t}B_{i}^{(j)}h^{(j)}(t-t_{i}^{(j)}),\label{intensityDH/H2}\end{equation}where $(t_i^{(j)})$ denote departure times from system $j$, and where $B_i^{(j)}\stackrel{\mathrm{iid}}\sim B^{(j)}$. 
For system $j$, we have i.i.d.\ departures distributed as $J^{(j)}$. 
If the baseline intensity, mark distribution or excitation function of system $1$ dominates that of system $2$, or if the sojourn time of system $2$ dominates that of system~$1$, we would expect system $1$ to stochastically dominate system $2$. 
Those conjectures are confirmed by the next theorem.

\begin{theorem}\label{stochdompar}
Suppose that system $1$ and system $2$ satisfy the following conditions:
\renewcommand{\theenumi}{\roman{enumi}}
\begin{enumerate}
\item $\lambda_{0}^{(1)}\geq\lambda_{0}^{(2)}$;
\item $B^{(1)}\geq_{\mathrm{st}}B^{(2)}$;
\item $h^{(1)}(t)\geq h^{(2)}(t)$ for almost all $t$;
\item $J^{(1)}\leq_{\mathrm{st}}J^{(2)}$.
\end{enumerate}
Then $N^{(1)}(t)\geq_{\mathrm{st}}N^{(2)}(t)$, $Q^{(1)}(t)\geq_{\mathrm{st}}Q^{(2)}(t)$ and $\Lambda^{(1)}(t)\geq_{\mathrm{st}}\Lambda^{(2)}(t)$ for all $t\geq0$.
\end{theorem}
\begin{proof}
It suffices to consider the case where just one of the conditions (i)-(iv) holds strictly.
For example, if (i) and (ii) hold strictly, we select an intermediate process $N^{(3)}$ with $\lambda_0^{(3)}=\lambda_0^{(2)}$ and $B^{(3)}=B^{(1)}$, and use our arguments to arrive at $N^{(1)}(t)\geq_{\mathrm{st}}N^{(3)}(t)\geq_{\mathrm{st}}N^{(2)}(t)$.

For (i)--(iii), the proof is straightforward: it uses the cluster representation, and relies on an easy coupling argument, by partly coupling the parameter of system $1$ to the corresponding one of system $2$,
with the remaining part generating a positive stream. 
For (iv), we argue as in the univariate case of the proof of Theorem~\ref{thstochasticdomination}. 
\end{proof}

\section{Networks of Markovian delayed Hawkes birth-death processes} \label{markov}
Next, we specify to (networks of) the Markovian delayed Hawkes process, which allows us to set up a more concrete characterization of the transform than the one found in Section~\ref{chapternonmarkov}, and to formulate a recursive procedure for calculating the joint moments of $(\boldsymbol Q(t),\boldsymbol\Lambda(t))$; see Section~\ref{networksmarkov}. 
In the univariate case, this leads to a system of ODEs involving a Clement-Kac-Sylvester matrix, which can be solved explicitly; see Section~\ref{markovtransient}. 
Furthermore, using the results of Section~\ref{comparisons}, we are able to describe the steady-state behavior of univariate delayed Hawkes birth-death processes in Section~\ref{subsectionsteadystate}.

\subsection{Networks of birth-death processes}\label{networksmarkov}
Networks of birth-death processes with shot-noise driven arrival rates have been studied in \cite{shotnoise}. 
Although networks of Hawkes processes have been introduced in \cite{Hawkesnetwork}, to the best of our knowledge, there is no account in the literature of the exact transient behavior of such processes. 
In this subsection, we analyze transient behavior for a network of Markovian delayed Hawkes birth-death processes. 
After obvious modifications, this analysis can be adapted to networks of (classical) Hawkes-fed birth-death processes. 
Furthermore, by setting $\mu_{ij}=0$ for all $i,j\in[d]$, see Definition~\ref{multivariatedef}, our analysis applies to the multivariate delayed Hawkes (point) process as well.

We first characterize the distribution of the Markovian network process from Definition~\ref{multivariatedef} by deriving a PDE for the joint Z- and Laplace transform of $(\boldsymbol  Q(\cdot),\boldsymbol\Lambda(\cdot))$, given by 
\begin{equation}\zeta(t,\boldsymbol  z,\boldsymbol  s)=\mathbb E\left[\boldsymbol  z^{\boldsymbol  Q(t)}e^{-\boldsymbol  s^\top\boldsymbol\Lambda(t)}\right]=\mathbb E\left[\prod_{j=1}^dz_j^{Q_j(t)}e^{-s_j\Lambda_j(t)}\right],\label{defjointtransform}
\end{equation}
where $\boldsymbol  z\in[-1,1]^d,\boldsymbol  s\in\mathbb R_+^d$. For $\boldsymbol  x\in\mathbb R^d$, $\boldsymbol  n\in\mathbb N_0^d$, write $\boldsymbol  x^{\boldsymbol  n}=\prod_{j=1}^dx_j^{n_j}$. 
By analogy to the univariate, nondelayed case, see \cite{Infinite server queues}, we derive a PDE for $\zeta$, to which we apply the method of characteristics to reduce it to a system of ODEs. 
Furthermore, this PDE can be used to derive a system of ODEs for the joint moments. 
The proofs of the following two results can be found in Appendix~\hyperref[app A4]{E}.

\begin{theorem}\label{theorem7characterization}
For all $i,j\in[d]$, assume that $h_{ij}(t)=e^{-r_it}$, where $r_i>0$, and assume that $B_{ij}>0$ a.s.
Consider the (now Markovian) network of delayed Hawkes birth-death processes from Definition~\ref{multivariatedef}. 
Then the multivariate joint Z- and Laplace transform $\zeta(t,\boldsymbol  z,\boldsymbol  s)$ satisfies the following PDE: \begin{align}\nonumber&\phantom=\frac{\partial \zeta(t,\boldsymbol  z,\boldsymbol  s)}{\partial t}+\sum_{j=1}^d(r_js_j+z_j-1)\frac{\partial \zeta(t,\boldsymbol  z,\boldsymbol  s)}{\partial s_j}+\sum_{j=1}^d\mu_j(z_j-\beta_j(\boldsymbol  s))\frac{\partial \zeta(t,\boldsymbol  z,\boldsymbol  s)}{\partial z_j}\\
&+\sum_{j=1}^d\sum_{i=1}^d\mu_{ij}(z_j-z_i)\frac{\partial \zeta(t,\boldsymbol  z,\boldsymbol  s)}{\partial z_j}=-\zeta(t,\boldsymbol  z,\boldsymbol  s)\sum_{j=1}^dr_j\lambda_{j,0}s_j,\label{multivariatePDE}
\end{align} 
where $\beta_j(\boldsymbol  s)=\mathbb E[e^{-\boldsymbol  s^\top \boldsymbol  B_j}]$ is the multivariate Laplace transform of $\boldsymbol  B_j$.

Furthermore, given initial conditions $\boldsymbol  Q(0)=\boldsymbol 0$ and $\boldsymbol\Lambda(0)=\boldsymbol\lambda_0$, we have \begin{equation}\zeta(t,\boldsymbol  z,\boldsymbol  s)=\prod_{j=1}^d\exp\left(-\lambda_{j,0}\left(s_j(t)+r_j\int_0^ts_j(u)\ \mathrm du\right)\right),\end{equation} where $s_j(\cdot)$, $j\in[d]$, solve the system of ODEs 
\begin{align}s_j'(u)&=-r_js_j(u)-z_j(u)+1;\nonumber\\z_j'(u)&=\mu_j(\beta_j(s(u))-z_j(u))+\sum_{i=1}^d\mu_{ij}(z_i(u)-z_j(u)),\quad0\leq u\leq t,
\end{align} 
with boundary conditions $s_j(0)=s_j$ and $z_j(0)=z_j$.
\end{theorem}

\begin{theorem}\label{corjointmom}
For $q,Q\in\mathbb N_0$, let $\bar Q^q:=Q(Q-1)(Q-q+1)$ be the falling factorial, with $\bar Q^0:=1$ and $\bar Q^{-1}:=0$. Write $b_{kj}=\mathbb E[B_{kj}]$. Next, for $\boldsymbol g,\boldsymbol \ell\in\mathbb N_0^d$, write 
\begin{align}
  \binom{\boldsymbol g}{\boldsymbol \ell}:=\prod_{j=1}^d\binom{g_j}{l_j}.   
\end{align} 
Furthermore, for $\boldsymbol  q\in\mathbb N_0^d, \boldsymbol  Q\in\mathbb N_0^d$, write $\bar{\boldsymbol  Q}^{\boldsymbol  q}:=\prod_{j=1}^d\bar Q_j^{q_j}$. 
Let $\circ$ be the Hadamard product.
Then we have the following differential equation for the joint moments of $\boldsymbol Q(t),\boldsymbol\Lambda(t)$:
\begin{align}
 &\phantom=\nonumber\frac{\mathrm d}{\mathrm dt}\mathbb E\left[\bar{\boldsymbol  Q}^{\boldsymbol  q}(t)\boldsymbol\Lambda^{\boldsymbol g}(t)\right]
 +\|\boldsymbol g\circ \boldsymbol  r+\boldsymbol  q\circ\boldsymbol\mu\|_1\mathbb E\left[\bar{\boldsymbol  Q}^{\boldsymbol  q}(t)\boldsymbol\Lambda^{\boldsymbol g}(t)\right]
 -\sum_{j=1}^dq_j\mathbb E\left[\bar{\boldsymbol  Q}^{\boldsymbol  q-\boldsymbol  e_j}(t)\boldsymbol\Lambda^{\boldsymbol g+\boldsymbol  e_j}(t)\right]\\
 &\phantom=-\sum_{j=1}^d\sum_{k=1}^d\mu_jg_kb_{kj}\mathbb E\left[\bar{\boldsymbol  Q}^{\boldsymbol  q+\boldsymbol  e_j}(t)\boldsymbol\Lambda^{\boldsymbol g-\boldsymbol  e_k}(t)\right]\nonumber\\
&\phantom=+\sum_{j=1}^d\sum_{i=1}^d\mu_{ij}\left(q_j\mathbb E\left[\bar{\boldsymbol Q}^{\boldsymbol q}(t)\boldsymbol\Lambda^{\boldsymbol g}(t)\right]-q_i\mathbb E\left[\bar{\boldsymbol Q}^{\boldsymbol q+\boldsymbol e_j-\boldsymbol e_i}(t)\boldsymbol\Lambda^{\boldsymbol g}(t)\right]\right)\nonumber
 \\&=\sum_{j=1}^dg_jr_j\lambda_{j,0}\mathbb E\left[\bar{\boldsymbol  Q}^{\boldsymbol  q}(t)\boldsymbol\Lambda^{\boldsymbol g-\boldsymbol  e_j}(t)\right]+\sum_{j=1}^d\mu_j\sum_{\substack{\boldsymbol0\leq\boldsymbol{\ell}\leq\boldsymbol g\\\|\boldsymbol\ell\|_1\leq\|\boldsymbol g\|_1-2}}\binom{\boldsymbol g}{\boldsymbol\ell}\mathbb E\left[\boldsymbol  B_j^{\boldsymbol g-\boldsymbol\ell}\right]\mathbb E\left[\bar{\boldsymbol  Q}^{\boldsymbol  q+\boldsymbol  e_j}(t)\boldsymbol\Lambda^{\boldsymbol\ell}(t)\right].\label{ODEmomentscompact}\end{align}
\end{theorem}

Eqn.~\eqref{ODEmomentscompact} allows us to devise a recursive procedure to find the joint moments of arbitrary order.
Indeed, the left-hand side of \eqref{ODEmomentscompact} expresses a joint moment of order $n=\|(\boldsymbol  q,\boldsymbol g)\|_1$ as a linear ODE dependent on joint moments of equal order, whereas the right-hand side contains a forcing term, consisting of lower-order moments only. 
In general, we can find the $(n+1)$-th order moments by solving a linear system of ODEs with forcing constant dependent on the moments of order up to $n$. 
Since the system for the first-order moments does not contain unknown quantities, this provides us with a recursive procedure for expressing the moments of a network of delayed Hawkes birth-death processes in the moments of the mark random variables, in the exponential decay rates $r_i$, and in the departure and rerouting rates $\mu_i$, $\mu_{ij}$. 

\begin{remark}
To find the moments of order $n$, we need to solve a system of ODEs of dimension $\binom{n+2d-1}{n}$. 
The ODEs are found by substituting all possible $(\boldsymbol  q,\boldsymbol g)$ into \eqref{ODEmomentscompact} satisfying $\|(\boldsymbol  q,\boldsymbol g)\|_1=n$.
\end{remark}

\subsection{Transient behavior of the univariate delayed Hawkes birth-death process}\label{markovtransient}
We now specify to the univariate case with $h(t)=e^{-rt}$,  since in this setting we can be more specific about the moments of $(Q(t),\Lambda(t))$. 
Specifying \eqref{ODEmomentscompact} to the univariate case $d=1$, we obtain the following ODE: 
 \begin{align}
 \nonumber&\phantom=\frac{\mathrm d}{\mathrm dt}\mathbb E\left[\bar Q^{q}(t)\Lambda^g(t)\right]+(gr+q\mu)\mathbb E\left[\bar Q^{q}(t)\Lambda^g(t)\right]-q\mathbb E\left[\bar Q^{q-1}(t)\Lambda^{g+1}(t)\right]\\&=\boldsymbol 1\{g\geq1\}gr\lambda_0\mathbb E\left[\bar Q^{q}(t)\Lambda^{g-1}(t)\right]+\boldsymbol 1\{g\geq1\}\mu\sum_{j=0}^{g-1}\binom gj\mathbb E\left[B^{g-j}\right]\mathbb E\left[\bar Q^{q+1}(t)\Lambda^{j}(t)\right].\label{gsqz}
 \end{align}
We wish to derive a system of ODEs for the joint moments of order $n\in\mathbb N$, which we accomplish by taking a combination of indices $g=k$, $q=n-k$, $k\in\{0,1,\ldots,n\}$, for which (\ref{gsqz}) reads 
\begin{align}
\nonumber&\phantom=\frac{\mathrm d}{\mathrm dt}\mathbb E\left[\bar Q^{n-k}(t)\Lambda^{k}(t)\right]+(kr+(n-k)\mu)\mathbb E\left[\bar Q^{n-k}(t)\Lambda^k(t)\right]-(n-k)\mathbb E\left[\bar Q^{n-k-1}(t)\Lambda^{k+1}(t)\right]\\\nonumber&-\mu kb_1\mathbb E\left[\bar Q^{n-k+1}(t)\Lambda^{k-1}(t)\right]\\
&=\boldsymbol 1\{k\geq1\}kr\lambda_0\mathbb E\left[\bar Q^{n-k}(t)\Lambda^{k-1}(t)\right]+\boldsymbol 1\{k\geq2\}\mu\sum_{j=0}^{k-2}\binom kj\mathbb E\left[B^{k-j}\right]\mathbb E\left[\bar Q^{n-k+1}(t)\Lambda^{j}(t)\right],
\end{align} 
where $b_1=\mathbb E[B]$. 
Letting $$Z^{(n+1)}(t):=\begin{bmatrix}\mathbb E\left[\bar Q^{n}(t)\right]&\mathbb E\left[\bar Q^{n-1}(t)\Lambda(t)\right]&\cdots&\mathbb E\left[\bar Q^{1}(t)\Lambda^{n-1}(t)\right]&\mathbb E\left[\Lambda^{n}(t)\right]\end{bmatrix}^\top,$$ $$A^{(n+1)}=\begin{bmatrix}-a_0^{(n)}&n&0&\cdots&0&0\\\mu b_1&-a_1^{(n-1)}&n-1&\cdots&0&0\\0&2\mu b_1&-a_2^{(n-2)}&\ddots&0&0\\\vdots&\vdots&\ddots&\ddots&\vdots&\vdots\\0&0&0&\cdots&-a_{n-1}^{(1)}&1\\0&0&0&\cdots&n\mu b_1&-a_n^{(0)}\end{bmatrix},\quad C^{(n+1)}(t)=\begin{bmatrix}c_0^{(n)}(t)\\c_1^{(n-1)}(t)\\c_2^{(n-2)}(t)\\\vdots\\c_{n-1}^{(1)}(t)\\c_n^{(0)}(t)\end{bmatrix},$$ 
where $a_k^{(n-k)}=kr+(n-k)\mu=n\mu+k(r-\mu)$ and 
$$c_k^{(n-k)}(t)=\boldsymbol 1\{k\geq1\}kr\lambda_0\mathbb E\left[\bar Q^{n-k}(t)\Lambda^{k-1}(t)\right]+\boldsymbol 1\{k\geq2\}\mu\sum_{j=0}^{k-2}\binom kj\mathbb E\left[B^{k-j}\right]\mathbb E\left[\bar Q^{n-k+1}(t)\Lambda^{j}(t)\right]$$ 
it follows that 
\begin{equation}\frac{\mathrm d}{\mathrm dt}Z^{(n+1)}(t)=A^{(n+1)}Z^{(n+1)}(t)+C^{(n+1)}(t).\label{ODEZlinear}\end{equation}
Note that $A^{(n+1)}$ is a \emph{generalized Clement-Kac-Sylvester matrix}.
To solve this ODE, we need $C^{(n+1)}(t)$, which is a vector dependent on 
moments of order at most $n-1$, meaning that we can solve for the transient moments of the delayed Hawkes birth-death process recursively.
The proofs of the next two results are in Appendix~\hyperref[app A4]{E}.

\begin{theorem}\label{ODEZsol}
The solution to the ODE (\ref{ODEZlinear}) is \begin{equation}Z^{(n+1)}(t)=e^{A^{(n+1)}t}Z^{(n+1)}(0)+\int_0^te^{A^{(n+1)}(t-s)}C^{(n+1)}(s)\ \mathrm ds,\label{DHODEsol}\end{equation} 
where $Z^{(n+1)}(0)=\lambda_0^n\boldsymbol e_{n+1}$, with $\boldsymbol e_{n+1}$ the last standard unit vector in $\mathbb R^{n+1}$. The matrix exponential $e^{A^{(n+1)}t}$ can be calculated explicitly by \begin{equation}\label{matrixexpDH}e^{A^{(n+1)}t}=\sum_{k=0}^ne^{\lambda_k^{(n+1)}t}\prod_{\substack{j=0\\ j\neq k}}^n\frac{A^{(n+1)}-\lambda_j^{(n+1)}I_{n+1}}{\lambda_k^{(n+1)}-\lambda_j^{(n+1)}},\end{equation} where  $I_{n+1}$ is the $(n+1)\times(n+1)$ identity matrix and where \begin{equation}\label{A1eigenvalues}\lambda_k^{(n+1)}=-\frac n2(\mu+r)+\frac{n-2k}{2}\sqrt{(\mu-r)^2+4\mu b_1},\quad k=0,1,\ldots, n,\end{equation} are the eigenvalues of $A^{(n+1)}$. This implies that we have a stable system --- i.e., with $Z^{(n+1)}(t)$ converging, as $t\to\infty$ ---  if and only if the stability condition $b_1/r<1$ holds.
\end{theorem}


We are able to find the first-order moments in the stationary regime, by letting $t\to\infty$. 

\begin{theorem}\label{DHstationary}
Let $b_1:=\mathbb E[B]$. If the stability condition $b_1/r<1$ holds, then, as $t\to\infty$, 
\begin{equation}\label{stationaryconvergence}
\begin{bmatrix}\mathbb E[Q(t)]\\\mathbb E[\Lambda(t)]\end{bmatrix}\to\frac{r\lambda_0}{r-b_1}\begin{bmatrix}1/\mu\\1\end{bmatrix}.
\end{equation}
\end{theorem}

\subsection{Univariate delayed Hawkes birth-death processes in steady state}\label{subsectionsteadystate}

In the next corollary to Theorem~\ref{thstochasticdomination}, we describe the steady-state delayed Hawkes birth-death process $(\tilde Q(\infty),\tilde\Lambda(\infty))$ in the Markovian setting; its proof is in Appendix~\hyperref[app A4]{E}.

\begin{corollary}\label{steadystateequaldist}
In the Markovian setting with $h(t)=e^{-rt}$ and $J\sim\mathrm{Exp}(\mu)$, in steady state we have $Q(\infty)\stackrel{\mathcal D}=\tilde Q(\infty)$ and $\Lambda(\infty)\stackrel{\mathcal D}=\tilde \Lambda(\infty)$.
\end{corollary}

\begin{remark}
    By combining Corollary \ref{steadystateequaldist} with \cite{Infinite server queues}, Corollaries 3.8 and 3.9, we find $\mathrm{Var}(\tilde N(\infty))$, $\mathrm{Cov}(\tilde N(\infty),\tilde\Lambda(\infty))$, and $\mathbb E[\tilde \Lambda^g(\infty)]$ for any $g\in\mathbb N$.
\end{remark}


In stationarity, the distribution of population sizes and intensities \textit{at a fixed time instant} are the same for Hawkes and delayed Hawkes. It should be borne in mind, however, that the dynamics of the two processes {\it are} different in stationarity, since an arrival does not increase the intensity  instantaneously for delayed Hawkes.

Corollary \ref{steadystateequaldist} has an appealing \emph{informal} explanation. 
In stationarity, the stream of particles entering and leaving a Hawkes-fed birth-death process are `in equilibrium'. 
Hence, starting in the stationary distribution of the Hawkes-fed birth-death process, excitation caused by arriving particles (as we have for Hawkes) equals excitation caused by departing particles (as we have for delayed Hawkes). 
For the Hawkes process, under the stationary distribution, the inward stream in intensity (caused by excitation) equals the outward stream (caused by exponential decay).
Hence, if the delayed Hawkes process starts in the stationary distribution of Hawkes, increase in intensity caused by departures (equals increase in intensity that we would see for Hawkes) equals the decrease caused by exponential decay. 
This indicates that this distribution is also stationary for delayed Hawkes.

Corollary~\ref{steadystateequaldist} allows us to describe heavy-traffic behavior for the delayed Hawkes birth-death process, assuming marks having finite second moments; cf.\ Corollary~\ref{heavytailstraffic} and see Appendix~\hyperref[app A4]{E} for the proof.

\begin{corollary}\label{cor:gam}
Consider a Markovian delayed Hawkes birth-death process as in Corollary~\ref{steadystateequaldist}. 
Suppose that $b_2:=\mathbb E[B^2]<\infty$. 
Then we have, as $\rho=b_1/r\uparrow1$,
\begin{align*} 
(1-\rho)\tilde\Lambda(\infty)\stackrel{\mathcal D}\to\Gamma\left(\frac{2r\lambda_0}{b_2},\frac{2r}{b_2}\right),\quad(1-\rho)\tilde Q(\infty)&\stackrel{\mathcal D}\to\Gamma\left(\frac{2r\lambda_0}{b_2},\frac{2r\mu}{b_2}\right).
\end{align*}
\end{corollary}

\section{Discussion and concluding remarks} \label{discussion}
We have formally introduced the delayed Hawkes process, 
and a rich family of point processes having sojourn-time dependent excitation, containing Hawkes, delayed Hawkes and the ephemerally self-exciting process as special cases. 
The delayed Hawkes process arises naturally in applications and has turned out to be remarkably tractable, admitting a cluster process representation in the linear case enabling transform characterizations by a fixed-point equation and the analysis of heavy-tailed asymptotics.
The effect of delays has been made visible in a scaling limit that is markedly different from its classical, non-delayed counterpart. 
Furthermore, using a method that one can describe as \emph{genealogical coupling}, we have demonstrated that the delayed Hawkes birth-death process is stochastically dominated by a comparable Hawkes-fed birth-death process.  
In the Markovian case, we have provided a recursive procedure to calculate the moments of a network of delayed Hawkes birth-death processes explicitly. 

In future research, several directions can be envisioned. 
\begin{itemize}
\item In Theorem~\ref{FCLT}, we could only state our FCLTs on an interval bounded away from $0$. 
One may want to study the (complex) behavior on an interval $[0,\epsilon]$ including $0$ as well. 
\item As discussed in Remark~\ref{otheralpha}, in the scaling limit for $\alpha\in(\frac12,1)$, it would be interesting to identify $\mathbb E[N_\alpha^T(v)]$. 
In the same remark, we saw that for $\alpha\in(\frac12,1)$, we still find a Brownian limit, whereas for $\alpha=1$ and sojourn times taking values in the unit interval, the situation is more involved; in particular, one may ask whether it is reasonable to expect short-range dependence.
If there is non-Gaussian behavior for $\alpha=1$, one may want to look for a scaling limit in which one multiplies sojourn times by $T^{\alpha(T)}$, before one contracts time by a factor $T$; 
here, $\alpha(T)\to1$ as $T\to\infty$. 
This setting bears some similarities with the one considered in \cite{Rosenbaum1}, although in our case quantities unscaled by $1-\alpha(T)$ do not diverge, but instead become smaller and, in some sense, `collapse' to a Poisson process for $\alpha>1$. 
A similar regime that may be interesting is the one where $\alpha=1/2$, but where we have sojourn times $J_T=T\cdot J$, for some positive random variable $J$, so that (A3) is not satisfied in the limit $T\to\infty$.
\item An interesting line of study concerns statistical inference for delayed Hawkes processes. 
A considerable amount of literature exists on this topic for classical Hawkes processes, but it is open to what extent these results extend to delayed Hawkes processes. 
In the Markovian case our closed-form expressions for the moments can be used to identify moment estimators, whereas the non-Markovian case is anticipated to be substantially more challenging.
In this direction, the goodness-of-fit results reported in Appendix \hyperref[app B]{B} are promising.
\item It would be interesting to analyze the effect of delays on cluster durations for delayed Hawkes, following the recent results by Daw~\cite{Conditional uniformity}. 
We have not succeeded in extending Daw's arguments to our setting.
\item Our general family of models having sojourn-time dependent excitation encompasses the Hawkes, the delayed Hawkes, and the ephemerally self-exciting processes as special cases. 
It would be interesting to identify other relevant models belonging to this family.
\end{itemize}


\section*{Appendix A: Relegated proofs of Section~\ref{existence}}\label{app A}

\noindent\textit{Proof of Theorem~\ref{stabilityth uni}}.
We prove the theorem in the univariate case. 
From there, the multivariate result can be proved along the lines of \cite{Stabilitypaper}, Theorem~7, taking the randomness of the excitation functions into account in the same fashion as we do in the univariate case.  
To avoid repetition, we exclude the proof.

The proof uses the idea of the Picard proof method for the existence of differential equations, and follows \cite{Bacry}, Theorem~1. 
It is structured as follows. 
We can assume, w.l.o.g., that $L=1$, by writing $\phi(\cdot)=\phi(L^{-1}L\ \cdot)$. 
First, we prove the `existence' part. 
We take a bivariate Poisson process $M$ marked with random functions. 
With the aid of Lemma~\ref{Poisson construction}, we use Picard iteration, starting from the empty process, to construct a stationary process $N$ with finite mean intensity satisfying the desired dynamics.  
Second, we prove the `uniqueness' part, by proving that \textit{any} stationary process $\tilde N$ with finite mean intensity satisfying the desired dynamics also satisfies condition (ii) in the theorem. 
This means that we have stability, from which we deduce $\tilde N\stackrel{\mathcal D}=N$. 
Third, we prove stability under condition (i). 
Next, under condition (ii), we can take expectations with respect to $M$ in the proof of the stability part below, after which the proof is analogous to the one under condition (i); therefore it is omitted.

\textit{Existence}. We construct the process $N$ upon a basis being a product probability space $(\mathscr X,\mathcal A,\mathbb P)$ of (i) the canonical space of bivariate point processes on $\mathbb R\times\mathbb R_+$, with a probability measure $\mathbb P_M$ such that the identity mapping is a bivariate Poisson process of unit rate, and (ii) $(\Omega,\mathcal F)=L^1(\mathbb R_+)\cap L^\infty(\mathbb R_+)$, with a probability measure $\mathbb Q$ denoting the distribution of the random functions $h$. 
Such a random function exists by Kolmogorov's extension theorem. 
We denote the resulting marked Poisson process on $\mathbb R\times\mathbb R_+\times\Omega$ by $M$. 
Write $(\mathcal A_t)_{t\in\mathbb R}$ for the filtration induced by $M$, i.e., $\mathcal A_t=\sigma(S_tM_-)$. Write $\mathcal P(\mathcal A_t):=\bigvee_{s<t}\mathcal A_s$ for the corresponding predictable $\sigma$-algebra.  
As indicated in Section~\ref{existence}, we treat the first coordinate of $M$ as time.

We say that a point process $N$ is \textit{compatible} w.r.t.\ the left-shift operator $S_t$ if for all $t\in\mathbb R$, $\tilde\omega\in\mathcal A$, $S_tN(\tilde\omega)=N(S_t\tilde\omega)$, where $S_t\tilde\omega$ means that time is shifted in the basis space, meaning that the first coordinate of $M$ is shifted.

We approximate the desired process $(N(\cdot),\Lambda(\cdot))$ using Picard iteration. 
More specifically, we set $\Lambda_0\equiv 0$, and for $n\in\mathbb N_0$, \begin{align}
N_{n}(A\times B)&=\int_{A\times\mathbb R_+\times B}\mathbf1_{[0,\Lambda_n(\tau)]}(s) M(\mathrm d\tau\times\mathrm ds\times\mathrm d\omega),\quad &&A\times B\in\mathcal B(\mathbb R)\otimes\mathcal F,\nonumber\\
\Lambda_{n+1}(t)&=\phi\left(\int_{(-\infty,t)\times\Omega}h(t-\tau,\omega)N_n(\mathrm d\tau\times\mathrm d\omega)\right),\quad&&t\in\mathbb R.\label{Picard}
\end{align}
By induction, for every $n\in\mathbb N_0$, $N_n$ is adapted to $(\mathcal A_t)_{t\in\mathbb R}$, while $\Lambda_n$ is adapted to $(\mathcal P(\mathcal A_t))_{t\in\mathbb R}$.
Note that if $\phi(0)=0$, the zero solution is stationary; we typically work with functions such that $\phi(0)>0$. 
By construction, the processes $(N_n), (\Lambda_n)$ are $S_t$-compatible and increasing in $n$. 
Since the basis space on which the process is constructed is time-invariant, it follows that $(N_n), (\Lambda_n)$ are stationary.
Since $\phi$ is Lipschitz, for $n\geq1$ it holds that \begin{equation*}\mathbb E|\Lambda_{n+1}(0)-\Lambda_n(0)|\leq\mathbb E\int_{(-\infty,0)\times\Omega}|h(-\tau,\omega)|\ (N_n-N_{n-1})(\mathrm d\tau\times\mathrm d\omega),\end{equation*}
where the first coordinate of $N_n-N_{n-1}$ counts the number of points between $t\mapsto \Lambda_n(t)$ and $t\mapsto \Lambda_{n-1}(t)$. 
By Lemma~\ref{Poisson construction}, this process has $\Lambda_n-\Lambda_{n-1}$ as an $\mathcal H_t^M$-intensity. 
Hence,
\begin{align*}\mathbb E[\Lambda_{n+1}(0)-\Lambda_n(0)]&\leq\mathbb E\int_{(-\infty,0)\times\Omega}|h(-\tau,\omega)|\ (N_n-N_{n-1})(\mathrm d\tau\times\mathrm d\omega)\\
&=\int_{(-\infty,0)\times\Omega}|h(-\tau,\omega)|\ \mathbb E (N_n-N_{n-1})(\mathrm d\tau\times\mathrm d\omega)\\
&=\int_{(-\infty,0)\times\Omega}|h(-\tau,\omega)|\ (\mathrm d\tau\times\mathrm d\omega) \mathbb E [\Lambda_n(0)-\Lambda_{n-1}(0)]\\
&=\|\mathbb E|h|\|_{L^1}\mathbb E[\Lambda_{n}(0)-\Lambda_{n-1}(0)], 
\end{align*}
where we use Fubini's theorem, and where the second equality follows by stationarity of $(\Lambda_n)$. 
It follows that \[\sum_{n\geq0}\mathbb E[\Lambda_{n+1}(0)-\Lambda_n(0)]\leq\frac{\phi(0)}{1-\|\mathbb E|h|\|_{L^1}}<\infty,\] hence $(\Lambda_n)$ converges in $L^1$ to some limit process $\Lambda$. 
Using the same bounds, Markov's inequality gives $$\mathbb P\left(\Lambda_{n+1}(0)-\Lambda_n(0)\geq \|\mathbb E|h|\|_{L^1}^{n/2}\right)\leq\phi(0)\|\mathbb E|h|\|_{L^1}^{n/2},$$ and since $\sum_{n\geq0}\|\mathbb E|h|\|_{L^1}^{n/2}<\infty$, an application of Borel-Cantelli gives that $(\Lambda_n)$ converges a.s.\ as well, to the same limit $\Lambda$. 

Next, since $N_n-N_{n-1}$ is a point process itself, for any bounded $A\in\mathcal B(\mathbb R)$ of Lebesgue measure $\mathrm{Leb}(A)<\infty$, and $B\in\mathcal F$,
\begin{align*}
\sum_{n\geq0}\mathbb P\left(\int_{A\times B}\ (N_{n+1}-N_n)(\mathrm d\tau\times\mathrm d\omega)\neq0\right)&\leq\sum_{n\geq0}\mathbb E\int_{A\times B}\ (N_{n+1}-N_n)(\mathrm d\tau\times\mathrm d\omega)\\
&=\mathrm{Leb}(A)\mathbb Q(B)\sum_{n\geq0}\mathbb E[\Lambda_{n+1}(0)-\Lambda_n(0)],
\end{align*}
which is finite, using that $\mathbb Q$ is a probability measure. 
Hence, by Borel-Cantelli, $N_n$ is a.s.\ eventually constant on any bounded $A\times B\in\mathcal B(\mathbb R)\times \mathcal F$, whence it converges to some process $N$. 
The left-shift operator is continuous, whence $$S_tN(\tilde\omega)=S_t\lim_{n\to\infty}N_n(\tilde\omega)=\lim_{n\to\infty}S_tN_n(\tilde\omega)=\lim_{n\to\infty}N_n(S_t\tilde\omega)=N(S_t\tilde\omega),$$
i.e., $N$ inherits the $S_t$-compatibleness of $(N_n)_{n\geq0}$.

To finish the proof of the existence part, we verify that the limit processes $N$, $\Lambda$ satisfy the stated dynamics. 
First, by Fatou's lemma, for all $A\in\mathcal B(\mathbb R)$, $B\in\mathcal F$ of bounded measure, it holds that
\begin{align*}
\phantom{\leq\ }\mathbb E&\int_{A\times B}\left|N(\mathrm d\tau\times\mathrm d\omega)-M(\mathrm d\tau\times [0,\Lambda(\tau)]\times\mathrm d\omega)\right|\\
&\leq\liminf_{n\to\infty}\mathbb E\int_{A\times B}\left|M(\mathrm d\tau\times [0,\Lambda_n(\tau)]\times\mathrm d\omega)-M(\mathrm d\tau\times [0,\Lambda(\tau)]\times\mathrm d\omega)\right|\\
&=\mathrm{Leb}(A)\mathbb Q(B)\liminf_{n\to\infty}\mathbb E|\Lambda_n(0)-\Lambda(0)|=0,
\end{align*} 
where we use stationarity of the intensity processes as we did before. 
Note that the limits in the previous display actually exists, so that we can replace the limit inferiors by limits. 
Hence, $N$ is a modification of a process with conditional intensity $\Lambda(\cdot)$. 
For the process $\Lambda(\cdot)$, note that 
\begin{align*}
\phantom=\ &\mathbb E\left|\Lambda(0)-\phi\left(\int_{(-\infty,0)\times\Omega}h(-\tau,\omega)\ N(\mathrm d\tau\times\mathrm d\omega)\right)\right|\\
&\leq\mathbb E|\Lambda(0)-\Lambda_{n+1}(0)|+\mathbb E\int_{(-\infty,0)\times\Omega}h(-\tau,\omega)\ (N-N_{n})(\mathrm d\tau\times\mathrm d\omega)\\
&=\mathbb E|\Lambda(0)-\Lambda_{n+1}(0)|+\|\mathbb E|h|\|_{L^1}\mathbb E|\Lambda(0)-\Lambda_{n}(0)|,
\end{align*}
where we apply Lemma~\ref{Poisson construction}, the Lipschitz condition, the triangle inequality, and Fubini's theorem.
Hence, by letting $n\to\infty$ and by using stationarity, we see that $\Lambda(\cdot)$ is a modification of the process satisfying dynamics \eqref{intensity nonlinear uni}.

\textit{Uniqueness}. To prove uniqueness of the stationary solution $\tilde N$ with finite mean intensity $\tilde\Lambda$, we show that such a process satisfies initial condition (ii) given in the theorem. 
From the stability part, it then follows that $S_t\tilde N\stackrel{\mathcal D}\to N$. 
By stationarity, $S_t\tilde N\stackrel{\mathcal D}=\tilde N$, so $\tilde N \stackrel{\mathcal D}= N$.

Indeed, by a change of variables, 
\begin{align*}
\mathbb E_M i_c(t)&=\tilde \Lambda\int_{t-c}^t\int_{(-\infty,0)\times\Omega}|h(s-\tau,\omega)|\ (\mathrm d\tau\times\mathrm d\omega)\ \mathrm ds\leq c\tilde \Lambda\int_{t-c}^\infty \mathbb E|h(\tau,\omega)|\ \mathrm d\tau;
\end{align*}
note that this upper bound tends to $0$ as $t\to\infty$ by dominated convergence and Fubini, and that we have $\mathbb E_M i_c(t)\leq c\tilde \Lambda\|\mathbb E|h|\|_{L^1}$ for all $t\in\mathbb R$. 
This verifies initial condition (ii).

\textit{Stability}. Let $\tilde N$ be a bivariate point process marked by random functions with dynamics \eqref{intensity nonlinear uni} on $\mathbb R_+$, satisfying initial condition (i). 
In particular, we do \textit{not} assume that it also satisfies dynamics \eqref{intensity nonlinear uni} on $\mathbb R_-$. 
We prove that the finite-dimensional distributions of $S_t\tilde N$ converge to those of $S_tN$. 
Then \cite{DaleyVereJones}, Theorem 11.1.VII gives stability:  $S_t\tilde N_+\stackrel{\mathcal D}\to N_+$, as $t\to\infty$. 

We prove convergence of finite-dimensional distributions by proving that for every $c\in(0,t)$, 
\begin{align*}
 \mathbb P\left(N\{\tau\}\neq \tilde N\{\tau\}\text{ for some }\tau\in(t-c,t)\big|\mathcal H_0^{\tilde N}\right)\to0   
\end{align*} as $t\to\infty$. 
Here, we assume that $N$ and $\tilde N$ are constructed using the same marked bivariate Poisson process $M$ of unit rate. 
This is justified as follows.  
It can be proved that the $\mathcal H_t^{\tilde N}$-intensity 
\begin{align}
  \tilde\Lambda(t)=\phi\left(\int_{(-\infty,t)\times\Omega}h(t-\tau,\omega)\ \tilde N(\mathrm d\tau\times\mathrm d\omega)\right) 
\end{align} of $\tilde N$ is such that $t\mapsto \mathbb E[\tilde \Lambda(t)|\mathcal H_0^{\tilde N}]$ is a.s.\ locally integrable; this is proved in the same way as in \cite{Stabilitypaper}, Theorem~1. 
For this we need the assumption $\|\mathbb E|h|\|_{L^\infty}<\infty$. 
It follows that $\tilde N$ is nonexplosive, a.s. 
Then \cite{Massoulie}, Lemma~2, implies existence of some marked bivariate Poisson process $M$ of unit rate from which $\tilde N$ can be constructed using Lemma~\ref{Poisson construction}.  

In order to prove convergence of finite-dimensional distributions, we consider 
$$f(t)=\mathbb E\left[|\Lambda(t)-\tilde\Lambda(t)|\,\Big|\,\mathcal H_0^{\tilde N}\right]\mathbf1\{t\geq0\},$$ 
which is a.s.\ locally integrable because  $t\mapsto \mathbb E[\tilde \Lambda(t)|\mathcal H_0^{\tilde N}]$ is. 
Here, $\mathcal H_t^{\tilde N}$ is the sigma-algebra generated by the history of $\tilde N$ up to time $t$. 
Also consider the integrated version of $f$: 
\begin{align*}F(t)&:=\int_{t-c}^tf(\tau)\ \mathrm d\tau=\mathbb E\left[\int_{t-c}^t\ \mathrm d|N-\tilde N|(\tau)\,\Big|\,\mathcal H_0^{\tilde N}\right]\\&\phantom:\geq\mathbb P\left(N\{\tau\}\neq \tilde N\{\tau\}\text{ for some }\tau\in(t-c,t)\,\big|\,\mathcal H_0^{\tilde N}\right),
\end{align*} 
where again $c\in(0,t)$. 
By the last inequality, it suffices to prove that $F(t)\to0$ as $t\to\infty$.

With $\Lambda$ denoting the average intensity of $\Lambda(t)$, it holds for $t\geq0$ that
\begin{align*}
f(t)&\leq\int_{(-\infty,0)\times\Omega}|h(t-\tau,\omega)|\ \tilde N(\mathrm d\tau\times\mathrm d\omega)+\Lambda\int_{(-\infty,0)\times\Omega}|h(t-\tau,\omega)|\ (\mathrm d\tau\times\mathrm d\omega)\\
&+\int_{(0,t)\times\Omega}|h(t-\tau,\omega)|f(\tau)\ (\mathrm d\tau\times\mathrm d\omega).
\end{align*}
Integrating from $t-c$ to $t$ gives, after some more bounding,  
$$F(t)\leq j_c(t)+\int_0^t\mathbb E|h(\tau)|F(t-\tau)\ \mathrm d\tau,$$ where $j_c(t)=i_c(t)+c\Lambda\int_{t-c}^\infty \mathbb E|h(\tau)|\ \mathrm d\tau$. 
This is a Volterra integral inequality of the second kind. 
Since $\|\mathbb E|h|\|_{L^1}<1$, Picard iteration gives $$F(t)\leq\int_{0}^tj_c(t-\tau)\left(\sum_{n\geq0}(\mathbb E|h|)^{n*}(\tau)\right)\ \mathrm d\tau.$$
Note that $\sum_{n\geq0}(\mathbb E|h|)^{n*}(\tau)$ can be bounded in $L^1$ by Young's convolution inequality. 
Also, by our assumption (i), it follows that $j_c(t)$ is bounded a.s.\ and converges to $0$ as $t\to\infty$. 
By dominated convergence, $F(t)\to0$ as $t\to\infty$, finishing the proof of the stability part.
$\hfill\Box$

\section*{Appendix B: Relegated details of Section~\ref{scalinglimit}}\label{app B}

We claim at the start of Section~\ref{scalinglimit} that it is possible to distinguish between a Hawkes and a delayed Hawkes process using statistical techniques. 
In particular, suppose that one generates realizations on $[0,T]\ni t$, with $T=50{,}000$, of a univariate, linear, exponential delayed Hawkes pure-birth process $N(t)$ having conditional intensity 
\begin{equation}
    \Lambda(t) = \lambda_0 + \sum_{t_i < t} \alpha e^{-r(t - t_i)},
\end{equation}
where $t_i - J_i$ are the event times of $N(t)$, with $J_i \stackrel{\mathrm{i.i.d.}}{\sim} \mathrm{Exp}(\mu)$. 
In other words, $(t_i)$ correspond to the death times of the birth-death process $Q$ associated with $N$. 
We choose parameters $(\lambda_0, \alpha, r, \mu) = (1/6, 3, 3.6, 1/6)$, where the first three parameters imply that the expected stationary arrival intensity of $N$ equals $1$.

We fit $N$ to a parametric null hypothesis consisting of univariate, linear, exponential Hawkes processes. 
In particular, for the parameter space $\Theta = (0,10)^3$, we consider the parametric null hypothesis
\begin{equation}\label{H0 exp}
H_0^{\mathrm{Exp}}: N \stackrel{d}{=} N_\theta^{\mathrm{Exp}} \text{ for some } \theta \in \{(\lambda_0, \alpha, r) \in \Theta : \alpha < r \},
\end{equation}
where $N_\theta^{\mathrm{Exp}} = N_{\lambda_0, \alpha, r}^{\mathrm{Exp}}$ is a univariate, linear, exponential Hawkes process having intensity
\begin{equation}\label{int exp}
\lambda^{\mathrm{Exp}}_{\theta}(t) = \lambda^{\mathrm{Exp}}_{\lambda_0, \alpha, r}(t) = \lambda_0 + \sum_{t_i < t} \alpha e^{-r (t - t_i)},
\end{equation}
where $t_i$ denote the event times of $N_\theta^{\mathrm{Exp}}$.

We apply the asymptotically correct goodness-of-fit test described in \cite{GoFpaper}, Algorithm~1, using $n = \mathrm{ceil}(\sqrt{T}/4)$ and an Andersen-Darling test in step (v) of their algorithm; these choices are motivated in \cite{GoFpaper}. 
Out of $1{,}000$ simulated sample paths, we reject $494$, $768$ and $949$ times using significance levels of $0.01$, $0.05$, and $0.20$, respectively. 
Hence, we can clearly detect the deviation of the delayed Hawkes process from the non-delayed null hypothesis empirically.



\section*{Appendix C: Relegated proofs of Section~\ref{chapternonmarkov}}\label{app A5}
\noindent\textit{Proof of Theorem~\ref{heavytailsunith}}.
In this proof, we first assume that $B\in\mathrm{APT}(-\alpha)$, so that $B$ is of class $\mathscr R(-\alpha)$ with $\ell(x)$ a function converging to a positive constant. 
Under this assumption, we prove that also $Q(t)\in\mathscr R(-\alpha)$. 
Then we argue that essentially the same proof holds to show that $B\in\mathscr R(-\alpha)$ implies $Q(t)\in\mathscr R(-\alpha)$, and we indicate what needs to be changed in the proof.

By specifying \eqref{jointtansform to clustertransform} to the univariate case, and setting $s=0$, we express the Z-transform of $Q(t)$ as 
\begin{equation}
\mathbb E\left[z^{Q(t)}\right]=\exp\left(-\lambda_0\int_0^t(1-\eta(u,z))\ \mathrm du\right),\label{Z^nt in eta}
\end{equation} 
where $\eta(u,z):=\mathbb E[z^{S^Q(u)}]$, the Z-transform of the birth-death cluster process $S^Q$. 
It satisfies 
\begin{equation}
\eta(u,z)=\mathscr J(u)z+\int_0^u\beta\left(\int_w^uh(s-w)(1-\eta(u-s,z))\ \mathrm ds\right)\ \mathrm d\bar{\mathscr J}(w),\label{fixedpointeta}
\end{equation} 
which follows by specifying \eqref{defphij} to the univariate delayed Hawkes setting, and where $\mathscr J$, $\bar{\mathscr J}$ denote the survival function and CDF, respectively, of the generic sojourn time random variable $J$.
In the remainder of the proof, we invoke a Tauberian theorem to relate the behavior of a regularly varying function at infinity to the behavior of its Laplace-Stieltjes transform at $0$. 
This relation for $\beta$ is substituted into \eqref{fixedpointeta}, after which we analyze expansions for $\eta(u,z)$ and $\mathbb E[z^{Q(t)}]$. 
By invoking the Tauberian theorem in the reverse direction, we conclude that $Q(t)$ is also of class $\mathscr R(-\alpha)$.

As indicated, we first assume that $\mathbb P(B>x)x^{\alpha}\to C$ for some $C>0$.
Then it follows from the Tauberian theorem \cite{Bingham}, Theorem~8.1.6, that $\beta(s)-1+sb_1\sim-C\Gamma(1-\alpha)s^\alpha$ as $s\downarrow0$. 
Hence, as $z\uparrow 1$, 
\begin{align}\nonumber&\beta\left(\int_w^uh(s-w)(1-\eta(u-s,z))\ \mathrm ds\right)-1+b_1\int_w^uh(s-w)(1-\eta(u-s,z))\ \mathrm ds\\&\sim-C\Gamma(1-\alpha)\left(\int_w^uh(s-w)(1-\eta(u-s,z))\ \mathrm ds\right)^\alpha.\label{betaexpansion}
\end{align}
Substituting this into \eqref{fixedpointeta} yields, as $z\uparrow1$, 
\begin{align}1-\eta(u,z)&\sim1-\mathscr J(u)z-\int_0^u\bigg{\{}1-b_1\int_w^uh(s-w)(1-\eta(u-s,z))\ \mathrm ds\nonumber\\&-C\Gamma(1-\alpha)\left(\int_w^uh(s-w)(1-\eta(u-s,z))\ \mathrm ds\right)^\alpha\bigg{\}}\ \mathrm d\bar{\mathscr J}(w)\nonumber\\&=\mathscr J(u)(1-z)+\int_0^u\bigg\{b_1\int_w^uh(s-w)(1-\eta(u-s,z))\ \mathrm ds\nonumber\\&+C\Gamma(1-\alpha)\left(\int_w^uh(s-w)(1-\eta(u-s,z))\ \mathrm ds\right)^\alpha\bigg\}\ \mathrm d\bar{\mathscr J}(w).\label{1-eta asymp}
\end{align}

Next, expand $1-\eta(u,z)=\mathbb E[S(u)](1-z)+o(1-z)$, as $z\uparrow1$. 
Write $\mathbb E[S(u)]=R_1(u)$ for the leading term. 
Substituting this into \eqref{1-eta asymp} and comparing terms of order $1-z$, we see that $R_1$ satisfies 
\begin{align}R_1(u)&=\mathscr J(u)+b_1\int_0^u\int_w^uh(s-w)R_1(u-s)\ \mathrm ds\ \mathrm d\bar{\mathscr J}(w)\nonumber\\&=\mathscr J(u)+b_1\int_0^uR_1(u-s)\int_0^sh(s-w)\ \mathrm d\bar{\mathscr J}(w)\ \mathrm ds\nonumber\\&=\mathscr J(u)+b_1(R_1*\bar h)(u),\label{R1u}
\end{align} 
where $\bar h$ is defined by $\bar h(s):=\int_0^s h(s-w)\ \mathrm d\bar{\mathscr J}(w)$, and where $*$ denotes the convolution operator.
This is a Volterra equation of the second kind, and by Picard iteration we obtain, for $u\geq0$, 
\begin{equation}
R_1(u)=\sum_{n\geq0}b_1^n(\bar h^{n*}*\mathscr J)(u).\label{eqR1}
\end{equation} 
The next term in the expansion of $1-\eta(u,z)$ is of the form $R_\alpha(u)(1-z)^{\alpha}$. 
When we substitute $1-\eta(u,z)=\mathbb E[S(u)](1-z)+R_\alpha(u)(1-z)^{\alpha}+o((1-z)^{\alpha})$ into \eqref{1-eta asymp} and compare terms of order $(1-z)^{\alpha}$, we obtain 
\begin{align}
R_\alpha(u)=b_1(R_\alpha*\bar h)(u)+C\Gamma(1-\alpha)\int_0^u\left(\int_w^uh(s-w)R_1(u-s)\ \mathrm ds\right)^\alpha\ \mathrm d\bar{\mathscr J}(w).
\end{align} 
This is again a Volterra equation of the second kind; by Picard iteration we obtain \begin{equation}
R_\alpha(u)=C\Gamma(1-\alpha)\sum_{n\geq0}b_1^n\left(\bar h^{n*}*\left(\int_0^\cdot\left(\int_w^\cdot h(s-w)R_1(\cdot-s)\ \mathrm ds\right)^\alpha\ \mathrm d\bar{\mathscr J}(w)\right)\right)(u).\label{PicardR_alpha}
\end{equation} 
From \eqref{h switching order int} with $\alpha=0$, we infer that $\|\bar h\|_{L^1}=\|h\|_{L^1}$. 
Hence, by applying Young's convolution inequality $n$ times with $r=p=\infty,q=1$ to each term of \eqref{eqR1}, and by recognizing a geometric series, $\|\bar h\|_{L^1}b_1=\|h\|_{L^1}b_1<1$ implies that $R_1$ is a bounded function of $u$. 
Since $\|h\|_{L^1}<\infty$, the inner integral in \eqref{PicardR_alpha} is finite, whence $R_\alpha$ is also a bounded function of $u$. 

We now substitute the expansion $1-\eta(u,z)=\mathbb E[S(u)](1-z)+R_\alpha(u)(1-z)^{\alpha}+o((1-z)^{\alpha})$ into \eqref{Z^nt in eta}, which gives, after expanding the exponential functions, 
\begin{align}
\mathbb E\left[z^{Q(t)}\right]&\sim\exp\left(-\lambda_0\int_0^t\left(R_1(u)(1-z)+R_\alpha(u)(1-z)^{\alpha}\right)\ \mathrm du\right)\nonumber\\&=1-\lambda_0(1-z)\int_0^tR_1(u)\ \mathrm du-\lambda_0(1-z)^{\alpha}\int_0^tR_\alpha(u)\ \mathrm du+o((1-z)^{\alpha})\label{expansionznt}.
\end{align} 
By using the Tauberian theorem \cite{Bingham}, Theorem 8.1.6, the other way around, it then follows that $Q(t)\in\mathscr R(-\alpha)$, as claimed. 

We now indicate what we have to change in the proof if we assume that $B\in\mathscr R(-\alpha)$, so that $\mathbb P(B>x)=\ell(x)x^{-\alpha}$ for some slowly varying function $\ell$. 
Note that the constant $-C\Gamma(1-\alpha)$ in (\ref{1-eta asymp}) should in that case be replaced by $\ell(1/I(u,z;w))$, where $I(u,z;w)=\int_w^uh(s-w)(1-\eta(u-s,z))\ \mathrm ds$.

For small $\delta\in(0,\alpha-1)$, we use Potter's Theorem (i.e., \cite{Bingham}, Theorem~1.5.6) to conclude that for $z$ sufficiently close to $1$ and for some $A>1$, \begin{equation}
\ell\left(\frac1{I(u,z;w)}\right)\bigg/\ell\left(\frac1{1-z}\right)\leq A\max\left\{\left(\frac{1-z}{I(u,z;w)}\right)^\delta,\left(\frac{1-z}{I(u,z;w)}\right)^{-\delta}\right\}.\label{potter}
\end{equation}
Our assumptions on $h$ imply that given $\epsilon>0$, there exists $K>0$ such that \begin{align}
\sup_{v\in[0,u]}\mathbb P(S^Q(v)>K)<\epsilon,    
\end{align} 
whence for $0<z<1$ and $0\leq v\leq u$ we have $\eta(v,z)=\mathbb E[z^{S^Q(v)}]\geq(1-\epsilon)z^K$. 
Hence, we have, as $z\uparrow1$, 
$$\frac1{I(u,z;w)}\geq\frac1{(1-(1-\epsilon)z^K)\int_0^{u}h(s)\ \mathrm ds}\to\frac1{\epsilon\int_0^uh(s)\ \mathrm ds}.$$ 
Given some threshold $D>0$ such that (\ref{potter}) holds for $\ell(x)/\ell(y)$ for all $x,y\geq D$, cf.\ \cite{Bingham}, Theorem 1.5.6, we choose $\epsilon>0$ sufficiently small to assure that ${\epsilon\int_0^uh(s)\ \mathrm ds}\leq 1/(2D)$, so that we can find some $z^*\in(0,1)$ such that $z\geq z^*$ implies that  $1/I(u,z;w)\geq D$ for all $s,w$. 
We also have $(z-1)^{-1}\geq D$ for $z\geq 1-1/D$. 

When we have $B\in\mathscr R(-\alpha)$ instead of $B\in\mathrm{APT}(-\alpha)$, we replace in \eqref{betaexpansion} the factor $-C\Gamma(1-\alpha)$ by $\ell(1/I(u,z;w))$. 
For $z>\max\{z^*,1-\frac1D\}$, we apply the bound \eqref{potter} and a similar Potter bound for 
\begin{align}\left.
 \ell\left(\frac1{1-z}\right)\right/\ell\left(\frac1{I(u,z;w)}\right);   
\end{align} 
then we have an upper and a lower bound for the asymptotic expansion of $1-\eta(u,z)$, to which we conduct an analysis analogous to the case $B\in\mathrm{APT}(-\alpha)$, yielding $Q(t)\in\mathscr R(-\alpha)$ both when we use the upper bound as if it were the true expansion, and when we use the lower bound. We conclude that $Q(t)\in\mathscr R(-\alpha)$.
$\hfill\Box$

\medskip 

\noindent\textit{Proof of Corollary~\ref{heavytailstraffic}}.
Letting $t\to\infty$ in \eqref{expansionznt}, we have 
\begin{equation}
1-\mathbb E\left[z^Q\right]\sim\lambda_0(1-z)\int_0^\infty R_1(u)\ \mathrm du+\lambda_0(1-z)^\alpha\int_0^\infty R_\alpha(u)\ \mathrm du.
\end{equation}
By applying Young's convolution inequality to each term of \eqref{eqR1} and by recognizing a geometric series, we observe that $\int_0^\infty R_1(u)\ \mathrm du$ is of order $({1-\rho})^{-1}$, as $\rho\uparrow1$. 
Similarly, we use \eqref{PicardR_alpha} to conclude that $\int_0^\infty R_\alpha(u)\ \mathrm du$ is of order $({1-\rho})^{-\alpha-1}$.

Note that $\mathbb E[Q]=\lambda_0\int_0^\infty R_1(u)\ \mathrm du$, so $\mathbb E[Q]=\mathcal O\left((1-\rho)^{-1}\right)$, as $\rho\uparrow1$, i.e., $(1-\rho)Q$ stays bounded as $\rho\uparrow 1$. 
More specifically, using $1-z^{1-\rho}=(1-\rho)(1-z)+\mathcal O\left((1-z)^2\right)$, as $z\uparrow1$, we have, up to $\mathcal O\left((1-z)^2\right)$ terms,
\begin{align}
\nonumber1-\mathbb E\left[z^{(1-\rho)Q}\right]&\sim\lambda_0\left(1-z^{(1-\rho)}\right)\int_0^\infty R_1(u)\ \mathrm du+\lambda_0\left(1-z^{(1-\rho)}\right)^\alpha\int_0^\infty R_\alpha(u)\ \mathrm du\\&=(1-\rho)(1-z)\lambda_0\int_0^\infty R_1(u)\ \mathrm du+(1-\rho)^\alpha(1-z)^\alpha\lambda_0\int_0^\infty R_\alpha(u)\ \mathrm du.\label{1-rho N exp}
\end{align}
From this expansion, it is clear that $(1-\rho)\lambda_0\int_0^\infty R_1(u)\ \mathrm du<\infty$ as $\rho\uparrow1$. 
The second term in \eqref{1-rho N exp} diverges, as $\rho\uparrow1$, which implies that $X:=\lim_{\rho\uparrow 1}(1-\rho)Q$ satisfies $\mathbb E\left[X^\alpha\right]=\infty$. 
$\hfill\Box$

\section*{Appendix D: Supplement to Section~\ref{chapternonmarkov}:\\ Cluster size distributions for gamma-distributed marks}\label{app A3}
In this appendix, we study the distribution of the cluster size of the delayed Hawkes process, that is, the total number of descendants of a single immigrant, including the immigrant itself. 
Note that the offspring size is given by a Poisson random variable with parameter equal $B\int_J^\infty h(t-J)\ \mathrm dt=B\varrho$, where $J$ is the sojourn time of the parent, and where we set $\varrho:=\|h\|_{L^1}$. 
In particular, the offspring distribution is the same as the one for a Hawkes process having the same parameters. 
This implies that the total size of a cluster is the same for both processes, and is given by the total progeny size of a Galton-Watson branching process, which can be determined with the aid of the hitting time theorem, see, e.g., \cite{hittingtimetheorem}.

\begin{lemma}[Hitting time theorem]\label{HTth}
The total progeny size $Z$ of a Galton-Watson branching process with offspring distribution $X$ has a distribution with probability mass function 
\begin{equation}
\mathbb P(Z=n)=\frac1n\mathbb P\left(\sum_{k=1}^nX_k=n-1\right),
\end{equation} 
where $(X_k)_{k\in\mathbb N}$ is an i.i.d.\ sequence of random variables having the same distribution as $X$.
\end{lemma}

For unmarked Hawkes processes, it is a well-known result that $Z\sim\mathrm{Borel}(\varrho)$, i.e., 
$$\mathbb P(Z=n)=\frac{e^{-\varrho n}(\varrho n)^{n-1}}{n!}.$$ 
Even without the probabilistic context, it can be proved that those Borel probabilities sum to unity by setting $x=-\varrho e^{-\varrho}\in(-e^{-1},0)$ for $\varrho\in(0,1)$, and by considering the Taylor expansion around $0$ of the principal branch of the Lambert W function. 

We now consider a marked (delayed) Hawkes process under the stability condition $\mathbb E[B]\varrho<1$. 
In this case, the offspring size follows a mixed-Poisson type distribution. 
To make use of Lemma~\ref{HTth}, we want this distribution to be such that i.i.d.\ sums belong to a well-known parametric family. 
This is the case for gamma-distributed marks. 
In fact, the assumption of gamma-distributed marks is not too restrictive, for the set of mixtures of gamma distributions is dense in the set of continuous probability distributions on $[0,\infty)$. 

\begin{proposition}\label{progth}
Let $\alpha,c>0$ be such that $\alpha\varrho/c<1$. 
Consider a (delayed) Hawkes process with $\Gamma(\alpha,c)$ distributed marks, i.e., the marks admit a density 
$$f_B(x)=\frac{c^\alpha x^{\alpha-1}e^{-cx}}{\Gamma(\alpha)}.$$ 
Then the total cluster size $Z$ is finite a.s.\ and has probability mass function \begin{equation}
\mathbb P(Z=n)=\frac1n\binom{(\alpha+1)n-2}{n-1}\left(\frac{c}{c+\varrho}\right)^{\alpha n}\left(\frac{\varrho}{c+\varrho}\right)^{n-1},\quad n\in\mathbb N,
\end{equation} where, for $x,y\in\mathbb R$, with $x>y-1$, we use the generalized binomial coefficient 
\begin{align}
    \binom xy=\frac{\Gamma(x+1)}{\Gamma(y+1)\Gamma(x-y+1)}.
\end{align}
\end{proposition}
\begin{proof}
Let $X$ denote the offspring random variable. 
Since, in self-evident notation, $X\,|\,B\sim\mathrm{Pois}(B\varrho)$, 
\begin{align*}
\mathbb P(X=n)
&=\mathbb E[\mathbb P(X=n|B)]=\int_0^\infty\frac{e^{-\varrho x}(\varrho x)^n}{n!}\frac{c^\alpha x^{\alpha-1}e^{-cx}}{\Gamma(\alpha)}\ \mathrm dx
=\frac{c^\alpha\varrho^n}{\Gamma(\alpha)n!}\int_0^\infty e^{-(c+\varrho)x}x^{\alpha+n-1}\ \mathrm dx\\
&=\frac{\Gamma(\alpha+n)}{\Gamma(\alpha)n!}\frac{c^\alpha\varrho^n}{(c+\varrho)^{\alpha+n}}
=\binom{\alpha+n-1}{n}{p}^\alpha(1-{p})^n
,
\end{align*} 
where ${p}:=c/(c+\varrho)$.
Hence, $X\sim\mathrm{NB}\left(\alpha,{p}\right)$, i.e., $X$ follows the \textit{generalized} negative binomial distribution; note that $\alpha>0$ is not necessarily integer. 
It follows that if $X_1,\ldots, X_n$ are i.i.d.\ copies of $X$, then $\sum_{k=1}^nX_k\sim\mathrm{NB}\left(\alpha n,{p}\right)$. 
The result now follows by an application of Lemma~\ref{HTth}.
\end{proof}

When $\alpha=1$, the gamma distribution reduces to an exponential distribution, and we obtain 
\begin{equation}\label{Catalan}
\mathbb P(Z=n)=C_{n-1}p^n(1-p)^{n-1},
\end{equation} 
where $C_n=\binom{2n}n/(n+1)$ is the $n$th Catalan number and $p$ as defined in the proof of Proposition~\ref{progth}. 
Note that the ephemerally self-exciting process with intensity jump $\varrho$ and expiration rate $c$ has the same progeny distribution, see \cite{Ephemeral}, Proposition~3.3. 
This is no coincidence. 
Letting $B$ be the expiration time of the ephemeral excitation, $X|B\sim\mathrm{Pois}(B\varrho)$, where $B\sim\mathrm{Exp}(c)$, showing that the offspring random variable $X$ has the same probabilistic behavior under the ephemerally self-exciting process and the (delayed) Hawkes process with exponentially distributed marks.


\section*{Appendix E: Relegated proofs of Section~\ref{markov}}\label{app A4}

\noindent\textit{Proof of Theorem~\ref{theorem7characterization}}.
For $\boldsymbol  k\in\mathbb N_0^d$, $\boldsymbol\lambda\in\mathbb R_+^d$, let 
\begin{align*}
F(t,\boldsymbol  k,\boldsymbol\lambda)&=\mathbb P(\boldsymbol  Q(t)=\boldsymbol  k,\boldsymbol\Lambda(t)\leq\boldsymbol\lambda),\quad f(t,\boldsymbol  k,\boldsymbol\lambda)=\frac{\partial^dF(t,\boldsymbol  k,\boldsymbol\lambda)}{\partial\lambda_1\cdots\lambda_d},\\\xi(t,\boldsymbol  k,\boldsymbol  s)&=\int_{\mathbb R_+^d}e^{-\boldsymbol  s^\top\boldsymbol\lambda}f(t,\boldsymbol  k,\boldsymbol\lambda)\ \mathrm d\boldsymbol\lambda.\\\intertext{Note that, with this notation,}\zeta(t,\boldsymbol  z,\boldsymbol  s)&=\sum_{\boldsymbol k\in\mathbb N_0^d}\boldsymbol z^{\boldsymbol k}\xi(t,\boldsymbol  k,\boldsymbol  s).
\end{align*} 
Let $\circ$ be the Hadamard product, 
and let $\boldsymbol  e_j$ be the $j$-th standard unit vector in $\mathbb R^d$. 
We consider the Markovian dynamics between times $t$ and $t+\Delta t$. 
Let  $\boldsymbol  k\in\mathbb N_0^d$ and $\boldsymbol\lambda\in\mathbb R_+^d$. 
Write $[\boldsymbol 0,\boldsymbol\lambda]:=[0,\lambda_1]\times\cdots\times[0,\lambda_d]$. 
Note that we may enter state $\boldsymbol k$ either due to an arrival in coordinate $j$, leaving state $\boldsymbol  k-\boldsymbol  e_j$; due to a departure in coordinate $j$, leaving state $\boldsymbol  k+\boldsymbol  e_j$; or due to rerouting from coordinate $j$ to $i$, leaving state $\boldsymbol  k+\boldsymbol  e_j-\boldsymbol e_i$. 
Therefore, as $\Delta t\downarrow0$, 
\begin{align*}F(t+\Delta t,\boldsymbol  k,\boldsymbol\lambda-\boldsymbol  r\circ(\boldsymbol\lambda-\boldsymbol\lambda_0)\Delta t)
&=\sum_{j=1}^d\int_{[\boldsymbol 0,\boldsymbol\lambda]}y_j\Delta tf(t,\boldsymbol  k-\boldsymbol  e_j,\boldsymbol  y)\ \mathrm d\boldsymbol  y\\
&+\sum_{j=1}^d(k_j+1)\mu_j\Delta t\int_{[\boldsymbol 0,\boldsymbol\lambda]}\mathbb P(\boldsymbol  B_j\leq\boldsymbol\lambda-\boldsymbol  y)f(t,\boldsymbol  k+\boldsymbol  e_j,\boldsymbol  y)\ \mathrm d\boldsymbol  y\\
&+\sum_{j=1}^d\sum_{i=1}^d(k_j+1)\mu_{ij}\Delta tF(t,\boldsymbol  k+\boldsymbol  e_j-\boldsymbol e_i,\boldsymbol  \lambda)\\
&+F(t,\boldsymbol  k,\boldsymbol\lambda)\left(1-\sum_{j=1}^dk_j\mu_j\Delta t-\sum_{j=1}^d\sum_{i=1}^dk_j\mu_{ij}\Delta t\right)\\
&-\sum_{j=1}^d\int_{[\boldsymbol 0,\boldsymbol\lambda]}y_j\Delta tf(t,\boldsymbol  k,\boldsymbol  y)\ \mathrm d\boldsymbol  y+o(\Delta t).\end{align*}
Subtracting $F(t,\boldsymbol  k,\boldsymbol\lambda)$ from both sides, dividing by $\Delta t$ and taking the limit as $\Delta t\downarrow0$ gives us 
\begin{align*}
&\phantom=\frac{\partial F(t,\boldsymbol  k,\boldsymbol\lambda)}{\partial t}-\left[\frac{\partial F(t,\boldsymbol  k,\boldsymbol\lambda)}{\partial \lambda_1},\cdots,\frac{\partial F(t,\boldsymbol  k,\boldsymbol\lambda)}{\partial \lambda_d}\right](\boldsymbol  r\circ(\boldsymbol\lambda-\boldsymbol\lambda_0))\\
&=\sum_{j=1}^d\int_{[\boldsymbol 0,\boldsymbol\lambda]}y_jf(t,\boldsymbol  k-\boldsymbol  e_j,\boldsymbol  y)\ \mathrm d\boldsymbol  y+\sum_{j=1}^d(k_j+1)\mu_j\int_{[\boldsymbol 0,\boldsymbol\lambda]}\mathbb P(\boldsymbol  B_j\leq\boldsymbol\lambda-\boldsymbol  y)f(t,\boldsymbol  k+\boldsymbol  e_j,\boldsymbol  y)\ \mathrm d\boldsymbol  y\\
&+\sum_{j=1}^d\sum_{i=1}^d(k_j+1)\mu_{ij}F(t,\boldsymbol  k+\boldsymbol  e_j-\boldsymbol e_i,\boldsymbol  \lambda)-\sum_{j=1}^dk_j\mu_jF(t,\boldsymbol  k,\boldsymbol\lambda)\\
&-\sum_{j=1}^d\sum_{i=1}^dk_j\mu_{ij}F(t,\boldsymbol  k,\boldsymbol \lambda)-\sum_{j=1}^d\int_{[\boldsymbol 0,\boldsymbol\lambda]}y_jf(t,\boldsymbol  k,\boldsymbol  y)\ \mathrm d\boldsymbol  y.
\end{align*}
Next, we take the partial derivative with respect to the intensity $\lambda_j$ of each coordinate; i.e., we apply the differential operator $\frac{\partial^d}{\partial\lambda_1\cdots\lambda_d}$ to both sides of the last equation. 
Here we apply Leibniz' integral rule and we use our assumption $\mathbb P(B_{ij}\leq0)=0$ for all $i,j\in[d]$. This yields 
\begin{align}
&\phantom=\frac{\partial f(t,\boldsymbol  k,\boldsymbol\lambda)}{\partial t}
-\sum_{j=1}^dr_j\frac{\partial}{\partial \lambda_j}(\lambda_j f(t,\boldsymbol  k,\boldsymbol\lambda))+\sum_{j=1}^dr_j\lambda_{j,0}\frac{\partial f(t,\boldsymbol  k,\boldsymbol\lambda)}{\partial \lambda_j}\nonumber\\
&=\sum_{j=1}^d\lambda_jf(t,\boldsymbol  k-\boldsymbol  e_j,\boldsymbol\lambda)
+\sum_{j=1}^d(k_j+1)\mu_j\int_{[\boldsymbol 0,\boldsymbol\lambda]}\frac{\partial^d\mathbb P(\boldsymbol  B_j\leq\boldsymbol\lambda-\boldsymbol  y)}{\partial\lambda_1\cdots\partial \lambda_d}f(t,\boldsymbol  k+\boldsymbol  e_j,\boldsymbol  y)\ \mathrm d\boldsymbol  y\label{eqjointdensity}\\
&+\sum_{j=1}^d\sum_{i=1}^d(k_j+1)\mu_{ij}f(t,\boldsymbol  k+\boldsymbol  e_j-\boldsymbol e_i,\boldsymbol  \lambda)-\sum_{j=1}^d(k_j\mu_j+\lambda_j)f(t,\boldsymbol  k,\boldsymbol\lambda)-\sum_{j=1}^d\sum_{i=1}^dk_j\mu_{ij}f(t,\boldsymbol  k,\boldsymbol \lambda).\nonumber\end{align}
Our next step is transforming to $\xi(t,\boldsymbol  k,\boldsymbol  s)$ by applying the integral operator $\int_{\mathbb R_+^d}e^{-\boldsymbol  s^\top\boldsymbol \lambda}\cdot\ \mathrm d\boldsymbol\lambda$ to both sides of \eqref{eqjointdensity}. 
We can do this term by term; the calculations rely on integration by parts, Tonelli's theorem, and swapping the order of differentiation and integration. 
We obtain
\begin{align}
 \nonumber&\phantom=\frac{\partial \xi(t,\boldsymbol  k,\boldsymbol  s)}{\partial t}
 +\sum_{j=1}^d\left((r_js_j-1)\frac{\partial \xi(t,\boldsymbol  k,\boldsymbol  s)}{\partial s_j}
 +\frac{\partial \xi(t,\boldsymbol  k-\boldsymbol  e_j,\boldsymbol  s)}{\partial s_j}\right)
 +\sum_{j=1}^dr_j\lambda_{j,0}s_j\xi(t,\boldsymbol  k,\boldsymbol  s)\\
 &=\sum_{j=1}^d(k_j+1)\mu_j\beta_j(\boldsymbol  s)\xi(t,\boldsymbol  k+\boldsymbol  e_j,\boldsymbol  s)-\sum_{j=1}^dk_j\mu_j\xi(t,\boldsymbol  k,\boldsymbol  s)\nonumber\\
 &+\sum_{j=1}^d\sum_{i=1}^d(k_j+1)\mu_{ij}\xi(t,\boldsymbol  k+\boldsymbol  e_j-\boldsymbol e_i,\boldsymbol s)-\sum_{j=1}^d\sum_{i=1}^dk_j\mu_{ij}\xi(t,\boldsymbol  k,\boldsymbol s).
\end{align} 
Now multiplying by $\boldsymbol  z^{\boldsymbol  k}$ and summing over $\boldsymbol  k\in\mathbb N_0^d$ gives \eqref{multivariatePDE}. 

The final part of the theorem follows by the method of characteristics, similarly as in \cite{Infinite server queues}, Theorem~3.1, by parametrising $s_j$ and $z_j$ by $u\in[0,t]$, with $s_j(t)=s_j$ and $z_j(t)=z_j$, after which we change variables to $u'=t-u$.
$\hfill\Box$

\medskip

\noindent\textit{Proof of Theorem~\ref{corjointmom}}.
We rewrite \eqref{multivariatePDE} to the joint transform by substituting its definition
$$\zeta(t,\boldsymbol  z,\boldsymbol  s)=\mathbb E\left[\boldsymbol  z^{\boldsymbol  Q(t)}e^{-\boldsymbol  s^\top\boldsymbol\Lambda(t)}\right]=\mathbb E\left[\prod_{j=1}^dz_j^{Q_j(t)}e^{-s_j\Lambda_j(t)}\right].$$
This gives us the PDE
\begin{align}
&\hspace{-5mm}\phantom=\frac{\mathrm d}{\mathrm dt}\mathbb E\left[\prod_{l=1}^dz_l^{Q_l(t)}e^{-s_l\Lambda_l(t)}\right]
-\sum_{j=1}^d(r_js_j+z_j-1)\mathbb E\left[\prod_{l=1}^dz_l^{Q_l(t)}\Lambda_j(t)e^{-s_l\Lambda_l(t)}\right]\nonumber\\
&\phantom=+\sum_{j=1}^d\mu_j(z_j-\beta_j(\boldsymbol  s))\mathbb E\left[Q_j(t)z_j^{Q_j(t)-1}e^{-s_j\Lambda_j(t)}\prod_{\substack{l=1\\l\neq j}}^dz_l^{Q_l(t)}e^{-s_l\Lambda_l(t)}\right]\nonumber\\
&\phantom=+\sum_{j=1}^d\sum_{i=1}^d\mu_{ij}(z_j-z_i)\mathbb E\left[Q_j(t)z_j^{Q_j(t)-1}e^{-s_j\Lambda_j(t)}\prod_{\substack{l=1\\l\neq j}}^dz_l^{Q_l(t)}e^{-s_l\Lambda_l(t)}\right]\nonumber\\
&=-\sum_{j=1}^dr_j\lambda_{j,0}s_j\mathbb E\left[\prod_{l=1}^dz_l^{Q_l(t)}e^{-s_l\Lambda_l(t)}\right].\label{multivariatePDE2}
\end{align} 
We differentiate \eqref{multivariatePDE2} $\boldsymbol g\in\mathbb N_0^d$ times with respect $\boldsymbol s$, meaning that we differentiate $g_j$ times with respect to $s_j$, for each $j\in[d]$. 
After this, we set $\boldsymbol s=\boldsymbol 0$. 
Similarly, we differentiate $\boldsymbol  q\in\mathbb N_0^d$ times with respect to $\boldsymbol  z$, after which we set $\boldsymbol  z=\boldsymbol  1$.  
This yields the following ODE for 
$\mathbb E\left[\prod_{l=1}^d\bar Q_l^{q_l}(t)\Lambda_l^{g_l}(t)\right]$:
\begin{align}
&\hspace{-5mm}\phantom=\nonumber\frac{\mathrm d}{\mathrm dt}\mathbb E\left[\prod_{l=1}^d\bar Q_l^{q_l}(t)\Lambda_l^{g_l}(t)\right]
  +\sum_{j=1}^d(g_jr_j+q_j\mu_j)\mathbb E\left[\prod_{l=1}^d\bar Q_l^{q_l}(t)\Lambda_l^{g_l}(t)\right]\\
&\phantom=-\sum_{j=1}^dq_j\mathbb E\left[\bar Q_j^{q_j-1}(t)\Lambda_j^{g_j+1}(t)\prod_{\substack{l=1\\l\neq j}}^d\bar Q_l^{q_l}(t)\Lambda_l^{g_l}(t)\right]\nonumber\\
\nonumber&\phantom=-\sum_{j=1}^d\mu_j\left(\sum_{k=1}^dg_kb_{kj}\mathbb E\left[\bar Q_j^{q_j+1}(t)\Lambda_k^{g_k-1}(t)\prod_{\substack{l=1\\l\neq j}}^d\bar Q_l^{q_l}(t)\prod_{\substack{l=1\\l\neq k}}^d\bar \Lambda_l^{g_l}(t)\right]\right)\\
&\phantom=+\sum_{j=1}^d\sum_{i=1}^d\mu_{ij}\left(q_j\mathbb E\left[\prod_{l=1}^d\bar Q_l^{q_l}(t)\Lambda_l^{g_l}(t)\right]-q_i\mathbb E\left[\bar Q_i^{q_i-1}(t)\bar Q_j^{q_j+1}(t)\prod_{\substack{l=1\\l\neq i,j}}^d\bar Q_l^{q_l}(t)\prod_{l=1}^d\bar \Lambda_l^{g_l}(t)\right]\right)
\nonumber\\
&=\sum_{j=1}^dg_jr_j\lambda_{j,0}\mathbb E\left[\bar Q_j^{q_j}(t)\Lambda_j^{g_j-1}(t)\prod_{\substack{l=1\\l\neq j}}^d\bar Q_l^{q_l}(t)\Lambda_l^{g_l}(t)\right]\nonumber\\
&\phantom=+\sum_{j=1}^d\mu_j\sum_{\substack{\boldsymbol0\leq\boldsymbol{\ell}\leq\boldsymbol g\\\|\boldsymbol \ell\|_1\leq\|\boldsymbol g\|_1-2}}\prod_{l=1}^d\binom{g_l}{\ell_l}\mathbb E\left[B_{lj}^{g_l-\ell_l}\right]\mathbb E\left[\bar Q_j^{q_j+1}(t)\Lambda_j^{\ell_j}(t)\prod_{\substack{k=1\\k\neq j}}^d\bar Q_k^{q_k}(t)\bar \Lambda_k^{\ell_k}(t)\right],
\label{ODEmomentsdevious}
\end{align}
We can write \eqref{ODEmomentsdevious} more compactly as \eqref{ODEmomentscompact}.
$\hfill\Box$

\medskip

\noindent\textit{Proof of Theorem~\ref{ODEZsol}}.
First, \eqref{DHODEsol} is immediate from \eqref{ODEZlinear}. 
In order to calculate $e^{A^{(n+1)}}$, we need to exploit the structure of $A^{(n+1)}$: the superdiagonal is of the form $c_1[n:1]$, the subdiagonal of the form $c_2[1:n]$, and the diagonal of the form $c_3[0:n]+c_4[n:0]=c_4n+(c_3-c_4)[0:n]$. 
Here, we write $[k:m]:=\{k,k\pm1,k\pm2,\ldots,m\mp1,m\}$ for the set of integers between $k,m\in\mathbb Z$.

In fact, $A^{(n+1)}$ is a generalization of the Clement-Kac-Sylvester matrix. 
Using \cite{Chu}, \S3, its characteristic polynomial $p_{n+1}(w)=\det\big(A^{(n+1)}-wI_{n+1}\big)$ is given by 
$$p_{n+1}(w)=\prod_{k=0}^n\left(-w-n\mu+\frac{n(\mu-r)}{2}+\frac{n-2k}2\sqrt{(\mu-r)^2+4\mu b_1}\right),$$ 
hence the eigenvalues of $A^{(n+1)}$ are given by \eqref{A1eigenvalues}. 
Finally, formula \eqref{matrixexpDH} follows from Lagrange-Sylvester interpolation; see \cite{LagrangeSylvester}, Theorem~8.1.

For stability, take some $n\in\mathbb N$, and note that we have convergence of the moments $Z^{(n+1)}(t)$ if and only if $e^{A^{(n+1)}t}\to0$ as $t\to\infty$, which holds if and only if 
$$\lambda^{(n+1)}_{\max}=\lambda^{(n+1)}_0=-\frac n2\left(\mu+r-\sqrt{(\mu-r)^2+4b_1}\right)=-\frac n2\left(\mu+r-\sqrt{(\mu+r)^2-4\mu(r-b_1)}\right)<0,$$ 
which in turn holds if and only if $4\mu(r-b_1)>0$ so if and only if $b_1/r<1$.
$\hfill\Box$

\medskip

\noindent\textit{Proof of Theorem~\ref{DHstationary}}.
This can be proved by solving the system of ODEs for $n=1$ by Theorem~\ref{ODEZsol}, using \begin{align}
A^{(2)}=\begin{bmatrix}-\mu&1\\\mu b_1&-r\end{bmatrix}, \:\:\:\:C^{(2)}(s)=\begin{bmatrix}0\\r\lambda_0\end{bmatrix},\:\:\:\:Z^{(2)}(0)=\begin{bmatrix}0\\\lambda_0\end{bmatrix},
\end{align}
and letting $t\to\infty$. 
Alternatively, use \eqref{ODEZlinear}, set the derivative equal to $0$, and solve for $Z^{(2)}(\infty)$.
$\hfill\Box$

\medskip

\noindent\textit{Proof of Corollary~\ref{steadystateequaldist}}.
We know from Theorem~\ref{thstochasticdomination} that $Q(t)\geq_{\mathrm{st}}\tilde Q(t)$ and $\Lambda(t)\geq_{\mathrm{st}}\tilde \Lambda(t)$ for all $t\geq0$. 
Furthermore, from Theorem~\ref{DHstationary}  and \cite{Infinite server queues}, Corollary~3.9, we know that 
$$\mathbb E[Q(\infty)]=\mathbb E[\tilde Q(\infty)]=\frac{\lambda_0r}{\mu(r-b_1)}\quad\text{and}\quad \mathbb E[\Lambda(\infty)]=\mathbb E[\tilde\Lambda(\infty)]=\frac{\lambda_0r}{r-b_1}.$$

Hence, to prove the claim, it suffices to prove that if $X(\cdot), Y(\cdot)$ are stochastic processes on $[0,\infty)$ such that for all $t\geq0$, $X(t)\geq_{\mathrm{st}}Y(t)$, while 
\[X(t)\stackrel{\mathcal D}\to X,\quad Y(t)\stackrel{\mathcal D}\to Y,\]
as $t\to\infty$, with $\mathbb E[X]=\mathbb E[Y]$, then $X\stackrel{\mathcal D}=Y$.

Indeed, let $F_{X(t)}$ be the CDF of $X(t)$ and $F_{Y(t)}$ be the CDF of $Y(t)$. 
Then, for all $t\geq0,z\in\mathbb R$, $F_{X(t)}(z)\leq F_{Y(t)}(z)$ since $X(t)\geq_{\mathrm{st}}Y(t)$. 
Furthermore, for each continuity point $z$ of $F_X$, $F_{X(t)}(z)\to F_{X}(z)$; similarly, for each continuity point $z$ of $F_Y$, $F_{Y(t)}(z)\to F_{Y}(z)$. 
Hence, for all but at most countably many points $z$, 
$$F_X(z)=\lim_{t\to\infty}F_{X(t)}(z)\leq\lim_{t\to\infty}F_{Y(t)}(z)=F_Y(z).$$ 
If this inequality does not hold for some $z$, then by right-continuity it does not hold for a continuum of values $[z,z+\epsilon]$. 
By contradiction, $F_X(z)\leq F_Y(z)$ for all $z\in\mathbb R$. Since $X,Y\geq0$, it follows that 
$$0=\mathbb E[X]-\mathbb E[Y]=\int_0^\infty(1-F_X(z))\ \mathrm dz-\int_0^\infty(1-F_Y(z))\ \mathrm dz=\int_0^\infty(F_Y(z)-F_X(z))\ \mathrm dz.$$ 
Since the integrand is nonnegative for all $z\geq0$, it follows that $F_X(z)=F_Y(z)$ for almost all $z\geq0$. 
Inequality at a point would again imply inequality on an interval of positive measure. 
Hence, $F_X=F_Y$, i.e., $X\stackrel{\mathcal D}= Y$.
$\hfill\Box$

\medskip

\noindent\textit{Proof of Corollary~\ref{cor:gam}}.
The result follows from Corollary~\ref{steadystateequaldist}, in combination with \cite{Infinite server queues}, Theorems~6.4 and~6.6.
$\hfill\Box$

\newpage
\clearpage\thispagestyle{empty}
\ \ \bigskip

\begin{center}
\textbf{\uppercase{Online supplement to ``Delayed Hawkes birth-death processes''}}    
\end{center}

\vb

    
{\small \noindent{\textsc{Abstract.}}\ \ In this online supplement to our paper ``Delayed Hawkes birth-death processes'', we prove Theorem~4. 
For context, notation and definitions, see the main paper.}

\bigskip


\vb


               

\setcounter{equation}{0}
\setcounter{theorem}{0}
\setcounter{section}{0}
\setcounter{definition}{0}
\setcounter{lemma}{0}
\setcounter{remark}{0}

\renewcommand{\thesection}{\Alph{section}}
\renewcommand{\thesubsection}{\thesection.\Roman{subsection}}
\renewcommand{\theequation}{\thesection.\Roman{equation}}
\renewcommand{\thetheorem}{\thesection.\Roman{theorem}}
\renewcommand{\thelemma}{\thesection.\Roman{lemma}}
\renewcommand{\thedefinition}{\thesection.\Roman{definition}}
\renewcommand{\theremark}{\thesection.\Roman{remark}}

\section{Proof of Theorem~4}
\subsection{Joint transform characterization}

The proof is a suitable modification of the work done in \cite{multivariateKLM-app}. 
We exploit the cluster representation provided in \cite{DHpaper-app}, Definition~2. 
First, we summarize the multidimensional notation as introduced in \cite{multivariateKLM-app}. 
We need this notation to gain insight in the clustering structure, and to state and prove the results we are after.

Before we come to that, we note that this cluster representation is useful for several reasons. 
First, modulo the time shift corresponding to the arrival times, clusters generated by immigrants in the same coordinate are i.i.d. 
Next, cluster processes are generated independently across source components. 
Finally, each event from the same source component generates offspring by the same iterative procedure, since each child itself determines a cluster, i.e., there is \textit{self-similarity}.

We operationalize these ideas as follows. 
Let $j\in[d]$. 
For each immigrant $(T_r^{(0)},J_r^{(0)},j)$, we denote the $d$-dimensional cluster process it generates as $\boldsymbol S_j^{\boldsymbol N}(\cdot)$. 
We also consider the \textit{birth-death cluster} $\boldsymbol S_j^{\boldsymbol Q}(\cdot)$ and the \textit{rate cluster} $\boldsymbol S_j^{\boldsymbol \lambda}(\cdot)$ generated by this immigrant, which give the number of remaining offspring (including the parent) and remaining intensity increases caused by the immigrant arrival. 
From now on, for $t\geq T_r^{(0)}$, we interpret $u=t-T_r^{(0)}$ as the time elapsed since the arrival of the corresponding immigrant. 
We are dealing with $d$-dimensional cluster processes, of which we write the components as 
\begin{equation}
\boldsymbol S_j^{\boldsymbol N}(u)=\begin{bmatrix}S_{1\leftarrow j}^{\boldsymbol N}(u)\\\vdots\\S_{d\leftarrow j}^{\boldsymbol N}(u)\end{bmatrix},\quad\boldsymbol S_j^{\boldsymbol Q}(u)=\begin{bmatrix}S_{1\leftarrow j}^{\boldsymbol Q}(u)\\\vdots\\S_{d\leftarrow j}^{\boldsymbol Q}(u)\end{bmatrix},\quad \boldsymbol S_j^{\boldsymbol \lambda}(u)=\begin{bmatrix}S_{1\leftarrow j}^{\boldsymbol \lambda}(u)\\\vdots\\S_{d\leftarrow j}^{\boldsymbol \lambda}(u)\end{bmatrix}.\label{eqclusteringprocesses}
\end{equation} 
Here, $S_{i\leftarrow j}^{\boldsymbol N}(u)$ records the number of events in component $i$ up to time $u$ with as oldest ancestor the immigrant generating $\boldsymbol S_j^{\boldsymbol N}(\cdot)$, including the immigrant itself when $i=j$. 
Similarly, $S_{i\leftarrow j}^{\boldsymbol Q}(u)$ records the number of non-expired events in component $i$ up to time $u$ with as oldest ancestor the immigrant generating $\boldsymbol S_j^{\boldsymbol Q}(\cdot)$, including the ancestor itself if $i=j$ and if the ancestor has not yet left the system. 
Finally, $S_{i\leftarrow j}^{\boldsymbol \lambda}(u)$ records aggregated  change in the intensity of component $i$ caused by jumps with excitation functions $h_{im,J,\omega}$, following arrivals in component $m$ with lifetime $J$ within the cluster $\boldsymbol S_j^{\boldsymbol \lambda}(\cdot)$ generated by an immigrant in component $j$. 
For each $S_{i\leftarrow j}^{\boldsymbol \star}(\cdot)$, $\boldsymbol \star\in\{\boldsymbol N,\boldsymbol Q,\boldsymbol \lambda\}$, note that changes in $i$ within the cluster generated by an immigrant in $j$ might propagate through other dimensions $m\in[d]$ due to the multivariate setting.

An immigration event in some coordinate $j$ generates first-generation offspring in \textit{all} coordinates, which in turn constitute clusters themselves, called \textit{subclusters}, which are second-generation offspring of the immigrant. 
To analyze the self-similarity inherent in this process, for $\boldsymbol \star\in\{\boldsymbol N,\boldsymbol Q,\boldsymbol \lambda\}$, we define the matrix process 
\begin{equation}
\boldsymbol S^{\boldsymbol \star}(\cdot):=\begin{bmatrix}\boldsymbol S_1^{\boldsymbol \star}(\cdot)\ |&\cdots&|\ \boldsymbol S_d^{\boldsymbol \star}(\cdot)\end{bmatrix}=\begin{bmatrix}S_{1\leftarrow 1}^{\boldsymbol \star}(\cdot)&\cdots&S_{1\leftarrow d}^{\boldsymbol \star}(\cdot)\\\vdots&\ddots&\vdots\\S_{d\leftarrow 1}^{\boldsymbol \star}(\cdot)&\cdots&S_{d\leftarrow d}^{\boldsymbol \star}(\cdot)\end{bmatrix}=:\begin{bmatrix}\boldsymbol S_{(1)}^{\boldsymbol \star}(\cdot)\\\vdots\\\boldsymbol S_{(d)}^{\boldsymbol \star}(\cdot)\end{bmatrix}.\label{eqclusteringprocesses2}
\end{equation} 
Note that the $j$th column $\boldsymbol S_j^{\boldsymbol \star}(\cdot)$ of this matrix process corresponds to offspring events originating in coordinate $j$, while the $i$th row $\boldsymbol S_{(i)}^{\boldsymbol \star}(\cdot)$ describes offspring events arriving in component $i$.

Using the clustering processes defined in \eqref{eqclusteringprocesses} and \eqref{eqclusteringprocesses2}, we can state distributional equalities for the component processes $N_i,Q_i,\lambda_i$. 
Indeed, letting $I_j(\cdot)$ be the immigration process in coordinate $j\in[d]$, i.e., a homogeneous Poisson process of rate $\lambda_{j,0}$, we have 
\begin{align}
N_i(t)&\stackrel{\mathcal D}=\sum_{j=1}^d\sum_{k=1}^{I_j(t)}S_{i\leftarrow j}^{\boldsymbol N}(t-T_k);\nonumber\\Q_i(t)&\stackrel{\mathcal D}=\sum_{j=1}^d\sum_{k=1}^{I_j(t)}S_{i\leftarrow j}^{\boldsymbol Q}(t-T_k);\label{multdistequalitiesNQlambda}\\\nonumber\lambda_i(t)&\stackrel{\mathcal D}=\lambda_{i,0}+\sum_{j=1}^d\sum_{k=1}^{I_j(t)}S_{i\leftarrow j}^{\boldsymbol \lambda}(t-T_k).
\end{align} 
Similar distributional equations can be formulated for the cluster processes, using the observation that each cluster itself generates subclusters. 
To exploit this structure, letting $\boldsymbol X(\cdot)$ be an $\mathbb R_+^d$-valued time-dependent process and $P\geq0$, for $j\in[d]$ we define the functional 
\begin{equation}\mathcal A_j(P,\boldsymbol X(\cdot))(u)=P+\sum_{m=1}^d\sum_{k=1}^{K_{mj,J,\omega}(u)}X_m(u-T_k),
\end{equation} 
where $T_k$ are arrival times in component $m$, and where $K_{mj,J,\omega}(\cdot)$ denotes an inhomogeneous Poisson process of rate $h_{mj,J,\omega}$. 
Here, it is understood that $J\sim J_j$, which is the same for each target coordinate $m$, and it is understood that the excitation functions $h_{mj,J,\omega}$ are conditionally independent. 
Finally, whenever $P$ is an expression of $J$, it is understood that $J$ is again the lifetime of the immigrant in coordinate $j$ under consideration, i.e., the same $J$ as appearing in $K_{mj,J,\omega}(\cdot)$.  
Using this functional, we have the following distributional equalities for the cluster processes: 
\begin{align}
S_{i\leftarrow j}^{\boldsymbol N}(u)&\stackrel{\mathcal D}=\mathcal A_j\left(\mathbf1\{i=j\},\boldsymbol S_{(i)}^{\boldsymbol N}(\cdot)\right)(u);\nonumber\\S_{i\leftarrow j}^{\boldsymbol Q}(u)&\stackrel{\mathcal D}=\mathcal A_j\left(\mathbf1\{i=j\}\mathbf 1\{J>u\},\boldsymbol S_{(i)}^{\boldsymbol Q}(\cdot)\right)(u)\label{distributionalequalitiesmathcalA};\\S_{i\leftarrow j}^{\boldsymbol \lambda}(u)&\stackrel{\mathcal D}=\mathcal A_j\left(h_{ij,J,\omega}(\cdot),\boldsymbol S_{(i)}^{\boldsymbol \lambda}(\cdot)\right)(u).\nonumber
\end{align}

As in the Markovian case, to characterize the probabilistic behavior of the joint process $(\boldsymbol Q(\cdot),\boldsymbol\lambda(\cdot))$, we wish to characterize its joint Z- and Laplace transform. 
We define such a transform for general multivariate joint processes with first $d$-dimensional component $\mathbb N_0^d$-valued and second $d$-dimensional component $\mathbb R_+^d$-valued.

\begin{definition}
Let $(\boldsymbol X(\cdot),\boldsymbol Y(\cdot))$ be a stochastic process taking values in $\mathbb N_0^d\times\mathbb R_+^d$. 
For any $t\in\mathbb R_+$, the \textit{joint transform} of $(\boldsymbol X(u),\boldsymbol Y(u))$ is defined by 
\begin{equation}
\mathcal J_{\boldsymbol X,\boldsymbol Y}(u)\equiv\mathcal J_{\boldsymbol X,\boldsymbol Y}(u,\boldsymbol s,\boldsymbol z):=\mathbb E\left[\boldsymbol z^{\boldsymbol X(u)}e^{-\boldsymbol s^\top\boldsymbol Y(u)}\right]=\mathbb E\left[\prod_{i=1}^d z_i^{X_i(u)}e^{-s_iY_i(u)}\right],
\end{equation} 
where $\boldsymbol s\in\mathbb R_+^d$ and $\boldsymbol z\in[-1,1]^d$. 
The expectation is w.r.t.\ the filtration at $t=0$. 
We call the space of such transforms $\mathbb J$; we write $\mathcal J_{\boldsymbol X,\boldsymbol Y}(\cdot)\in\mathbb J$.

Furthermore, when we have an $\mathbb N_0^{d\times d}\times\mathbb R_+^{d\times d}$-valued matrix stochastic process $(\boldsymbol X(\cdot),\boldsymbol Y(\cdot))$ with $j$th column processes $(\boldsymbol X_j(\cdot),\boldsymbol Y_j(\cdot))$, then we define $\mathbb J^d$ as the $d$-dimensional analogue of $\mathbb J$, with $\mathcal J_{\boldsymbol X,\boldsymbol Y}(\cdot)\in\mathbb J^d$ defined by 
\begin{equation}
\boldsymbol{\mathcal J}_{\boldsymbol X,\boldsymbol Y}(u):=\begin{bmatrix}\mathcal J_{\boldsymbol X_1,\boldsymbol Y_1}(u)\\\vdots\\\mathcal J_{\boldsymbol X_d,\boldsymbol Y_d}(u)\end{bmatrix},
\end{equation} 
where $\boldsymbol{\mathcal J}_{\boldsymbol X_j,\boldsymbol Y_j}(\cdot)\in\mathbb J$. 
\end{definition}

As indicated before the definition, our aim is to characterize $\mathcal J_{\boldsymbol Q,\boldsymbol \lambda}(\cdot)$ defined by $\mathcal J_{\boldsymbol Q,\boldsymbol \lambda}(t)=\mathbb E\left[\boldsymbol z^{\boldsymbol Q(t)}e^{-\boldsymbol s^\top\boldsymbol \lambda(t)}\right]$, with initial conditions $\boldsymbol Q(0)=0$ and $\boldsymbol \lambda(0)=\boldsymbol \lambda_0$. 
We start by using the distributional equalities (\ref{multdistequalitiesNQlambda}) to express $\mathcal J_{\boldsymbol Q,\boldsymbol \lambda}(\cdot)$ in the joint transform of $(\boldsymbol S_j^{\boldsymbol Q}(\cdot),\boldsymbol S_j^{\boldsymbol \lambda}(\cdot))$

\begin{theorem}\label{jointtransformprocesstocluster}
The joint transform $\mathcal J_{\boldsymbol Q,\boldsymbol \lambda}(\cdot)$ satisfies 
\begin{equation}
\mathcal J_{\boldsymbol Q,\boldsymbol \lambda}(t,\boldsymbol s,\boldsymbol z)=\prod_{j=1}^d\exp\left(-\lambda_{j,0}\left(t+s_j-\int_0^t\mathcal J_{\boldsymbol S_j^{\boldsymbol Q},\boldsymbol S_j^{\boldsymbol \lambda}}(u,\boldsymbol s,\boldsymbol z)\ \mathrm du\right)\right).
\end{equation}
\end{theorem}
\begin{proof}
The proof proceeds as the proof of \cite{multivariateKLM-app}, Theorem~1, since that proof only depends on their distributional equalities \cite{multivariateKLM-app}, eqn.\ (13), which are the same as the equalities \eqref{multdistequalitiesNQlambda} in our case.
\end{proof}

This theorem shows that we can characterize the probabilistic behavior of $(\boldsymbol Q(\cdot),\boldsymbol\Lambda(\cdot))$ if we can characterize the joint transform of $(\boldsymbol S_j^{\boldsymbol Q}(\cdot),\boldsymbol S_j^{\boldsymbol \lambda}(\cdot))$ for any $j\in[d]$. 
This will be the subject of the next subsection, where we show that these transforms can be identified as a fixed point of a certain mapping. 
Furthermore, we will show that iterates of that mapping converge to the fixed point, for \textit{any} starting point, thereby giving an iterative procedure to approximate those joint transforms.

\subsection{Fixed-point theorem and convergence results}
We expressed $\mathcal J_{\boldsymbol Q,\boldsymbol \lambda}(\cdot)$ in terms of $\mathcal J_{\boldsymbol S_j^{\boldsymbol Q},\boldsymbol S_j^{\boldsymbol \lambda}}(\cdot)$, $j\in[d]$. 
Hence, to obtain an full characterization of $\mathcal J_{\boldsymbol Q,\boldsymbol \lambda}(\cdot)$, we need a method to determine $\mathcal J_{\boldsymbol S_j^{\boldsymbol Q},\boldsymbol S_j^{\boldsymbol \lambda}}(\cdot)$. Write $\mathcal G_j(\cdot)=\mathcal J_{\boldsymbol S_j^{\boldsymbol Q},\boldsymbol S_j^{\boldsymbol \lambda}}(\cdot)\in\mathbb J$. 
We aim to find $\mathcal G_j$ for all $j\in[d]$. 
This is equivalent to finding the vector-valued transform 
\begin{equation}
\mathbb J^d\ni\boldsymbol{\mathcal G}(\cdot):=\boldsymbol{\mathcal J}_{\boldsymbol S^{\boldsymbol Q},\boldsymbol S^{\boldsymbol \lambda}}(\cdot)=\begin{bmatrix}\mathcal J_{\boldsymbol S_1^{\boldsymbol Q},\boldsymbol S_1^{\boldsymbol \lambda}}(\cdot)\\\vdots\\\mathcal J_{\boldsymbol S_d^{\boldsymbol Q},\boldsymbol S_d^{\boldsymbol \lambda}}(\cdot)\end{bmatrix}=\begin{bmatrix}\mathcal G_1(\cdot)\\\vdots\\\mathcal G_d(\cdot)\end{bmatrix}.
\end{equation} 
Next, we define a mapping $\phi$ for which we will prove that $\boldsymbol{\mathcal G}(\cdot)$ is a fixed point, and for which iterates of an arbitrary $\boldsymbol{\mathcal J}^{(0)}(\cdot)\in\mathbb J^d$ will converge to $\boldsymbol{\mathcal G}(\cdot)$.

\begin{definition}\label{defphi}
Let $\phi:\mathbb J^d\to\mathbb J^d$ be the mapping defined by 
\begin{equation}
\boldsymbol{\mathcal J}(\cdot)=\begin{bmatrix}\mathcal J_1(\cdot)\\\vdots\\\mathcal J_d(\cdot)\end{bmatrix}\mapsto\begin{bmatrix}\phi_1(\mathcal J_1,\ldots,\mathcal J_d)(\cdot)\\\vdots\\\phi_d(\mathcal J_1,\ldots,\mathcal J_d)(\cdot)\end{bmatrix}=\begin{bmatrix}\phi_1(\boldsymbol{\mathcal J})(\cdot)\\\vdots\\\phi_d(\boldsymbol{\mathcal J})(\cdot)\end{bmatrix}=\phi(\boldsymbol{\mathcal J})(\cdot),
\end{equation} 
where for $j\in[d]$ 
\begin{align}\phi_j(\boldsymbol{\mathcal J})(u)&\equiv\phi_j(\boldsymbol{\mathcal J})(u,\boldsymbol s,\boldsymbol z)\nonumber\\&=\mathbb E_{J,\omega}\left[z_j^{\mathbf1\{J>u\}}\prod_{i=1}^de^{-s_ih_{ij,J,\omega}(u)}\prod_{m=1}^d\exp\left(-\int_0^uh_{mj,J,\omega}(v)\left(1-\mathcal J_m(u-v,\boldsymbol s,\boldsymbol z)\right)\ \mathrm dv\right)\right].
\end{align}
\end{definition}

We need to prove that this mapping is well-defined, i.e., that for $\boldsymbol{\mathcal J}_{\boldsymbol X,\boldsymbol Y}(\cdot)\in\mathbb J^d$, also $\phi(\boldsymbol{\mathcal J}_{\boldsymbol X,\boldsymbol Y})(\cdot)\in\mathbb J^d$. 
We postpone this until after the next lemma and theorem, which prove that $\boldsymbol{\mathcal G}(\cdot)$ is a fixed point of $\phi$. 
After we have proved that $\phi$ is well-defined, we will prove that it is continuous w.r.t.\ some appropriate topology. 
Thereafter, we show that iterates of an arbitrary $\boldsymbol{\mathcal J}^{(0)}(\cdot)\in\mathbb J^d$ under $\phi$ will converge to  $\boldsymbol{\mathcal G}(\cdot)$. 
This final result describes a method to determine $\mathcal J_{\boldsymbol S_j^{\boldsymbol Q},\boldsymbol S_j^{\boldsymbol \lambda}}(\cdot)$ explicitly, for any $j\in[d]$, hence completing our characterization of $\mathcal J_{\boldsymbol Q,\boldsymbol \lambda}(\cdot)$.

Below, when we prove that $\boldsymbol{\mathcal G}(\cdot)$ is a fixed point of $\phi$, we need to specify when offspring events arrive exactly, given our knowledge that the offspring events arrive before time $u$, where $u$ is the remaining time after the arrival of the source event. 
Remember that offspring events arrive according to an inhomogeneous Poisson process by the cluster representation \cite{DHpaper-app}, Definition~2. 
This implies that those offspring events are positioned in $[0,u]$ according to the normalized restriction to $[0,u]$ of the intensity measure corresponding to the inhomogeneous Poisson process. 
For $v\in[0,u]$, let $P_{ij,J,\omega}(v|u)$ be the probability that an offspring event in coordinate $i$ caused by an immigrant in coordinate $j$ was already generated before time $v$, conditional on being generated before time $u$, and conditional on $J,\omega$.

\begin{lemma}\label{clusterlemma}
Consider the cluster process $\boldsymbol S_j^{\boldsymbol\star}$ for $\boldsymbol \star\in\{\boldsymbol N,\boldsymbol Q,\boldsymbol \lambda\}$ generated by an immigrant event $(T^{(0)},J^{(0)},j)$ in component $j\in[d]$, and let $u=t-T^{(0)}$ be the time elapsed since its arrival. Then the following statements hold.
\begin{enumerate}
\item Subclusters are i.i.d.\ modulo the time shift: for each $m\in[d]$, modulo the time shifts $T_k$ corresponding to the arrival times of the first generation events $(T^{(1)}_k)_{k\in\mathbb N}$, the sequence $\left(\boldsymbol S^{\boldsymbol\star}_m(u-T_k)\right)_{k\in[n]}$ is i.i.d., conditional on $\{K_{mj,J^{(0)}}=n\}$ for some $n\in\mathbb N$.
\item For $v\leq u$ the probability $P_{ij,J,\omega}(v|u)$ is differentiable with derivative \begin{equation}p_{ij,J,\omega}(v|u)=\frac{h_{ij,J,\omega}(v)}{\int_0^uh_{ij,J,\omega}(s)\ \mathrm ds}.\end{equation}
\end{enumerate}
\end{lemma}
\begin{proof}
The proof is analogous to the proof of \cite{multivariateKLM-app}, Lemma~3. 
Note that we use that $h_{ij,J}$ is a.s.\ piecewise continuous, for almost all realization of $J\sim J_j$, in order to be able to differentiate the probability $P_{ij,J,\omega}(v|u)$.
\end{proof}

\begin{theorem}\label{multifixedpoint}
The vector of time-dependent joint transforms $\boldsymbol{\mathcal G}(\cdot):=\boldsymbol{\mathcal J}_{\boldsymbol S^{\boldsymbol Q},\boldsymbol S^{\boldsymbol \lambda}}(\cdot)$ satisfies $\boldsymbol{\mathcal G}(\cdot)=\phi(\boldsymbol{\mathcal G})(\cdot)$.
\end{theorem}
\begin{proof}
The idea of this proof is the same as that of \cite{multivariateKLM-app}, Theorem~2. 
It suffices to prove, for arbitrary $j\in[d]$ and $u\geq0$, that $\mathcal G_j(u)=\phi_j(\boldsymbol{\mathcal G})(u)$. 
In the proof we keep $\boldsymbol s$ and $\boldsymbol z$ fixed. 
Write $\boldsymbol K_{j,J}(u)=\begin{bmatrix}K_{1j,J}&\cdots&K_{dj,J}\end{bmatrix}^\top$ for the random (i.e., $\omega$-dependent) vector of Poisson processes of rate $h_{ij,J,\omega}$, where the random excitation functions $h_{ij,J,\omega}$ are assumed to be conditionally independent for $i\in[d]$. 
By using the distributional equalities (\ref{distributionalequalitiesmathcalA}), and with $J\sim J_j$ the lifetime of the immigrant in coordinate $j$, 
\begin{align}
\nonumber\mathcal G_j(u)&=\mathbb E_{J,\omega}\left[\mathbb E\left[\prod_{i=1}^dz_i^{S_{i\leftarrow j}^{\boldsymbol Q}(u)}e^{-s_iS_{i\leftarrow j}^{\boldsymbol \lambda}(u)}\bigg|J,\omega\right]\right]\\\nonumber&=\mathbb E_{J,\omega}\left[\sum_{\boldsymbol n\in\mathbb N_0^d}\mathbb E\left[\prod_{i=1}^dz_i^{S_{i\leftarrow j}^{\boldsymbol Q}(u)}e^{-s_iS_{i\leftarrow j}^{\boldsymbol \lambda}(u)}\bigg|\boldsymbol K_{j,J}(u)=\boldsymbol n\right]\mathbb P\left(\boldsymbol K_{j,J}(u)=\boldsymbol n\right|J,\omega)\right]\\&=\mathbb E_{J,\omega}\left[c(u)\sum_{\boldsymbol n\in\mathbb N_0^d}\mathbb E\left[\prod_{i=1}^dz_i^{\sum\limits_{m=1}^d\sum\limits_{k=1}^{n_m}S_{i\leftarrow m}^{\boldsymbol Q}(u-T_k)}e^{-s_i\sum\limits_{m=1}^d\sum\limits_{k=1}^{n_m}S_{i\leftarrow m}^{\boldsymbol \lambda}(u-T_k)}\right]\mathbb P\left(\boldsymbol K_{j,J}(u)=\boldsymbol n\right|J,\omega)\right],\nonumber
\end{align} 
where $$c(u):=z_j^{\mathbf1\{J>u\}}\prod_{i=1}^de^{-s_ih_{ij,J,\omega}(u)}.$$ 
Now we use the i.i.d.\ nature of the subclusters as in Lemma~\ref{clusterlemma} and the fact that next-generation offspring of a parent is distributed according to a Poisson process, to rewrite the inner expectation as a product over the source components of first-generation events. 
We let $T^{(mj)}$ be the r.v.\ with density $p_{mj,J,\omega}(v|u)$ as given in Lemma~\ref{clusterlemma}. 
These times are distributed as $T_k$ if those were sampled by $K_{mj,J,\omega}$. 
This leads to 
\begin{align*}
&\phantom=\mathbb E\left[\prod_{i=1}^dz_i^{\sum\limits_{m=1}^d\sum\limits_{k=1}^{n_m}S_{i\leftarrow m}^{\boldsymbol Q}(u-T_k)}e^{-s_i\sum\limits_{m=1}^d\sum\limits_{k=1}^{n_m}S_{i\leftarrow m}^{\boldsymbol \lambda}(u-T_k)}\right]\\&=\prod_{m=1}^d\mathbb E\left[\prod_{i=1}^dz_i^{S_{i\leftarrow m}^{\boldsymbol Q}(u-T^{(mj)})}e^{-s_iS_{i\leftarrow m}^{\boldsymbol \lambda}(u-T^{(mj)})}\right]^{n^m}=\prod_{m=1}^d\mathcal G_m\left(u-T^{(mj)}\right)^{n_m},
\end{align*} 
whence 
\begin{align*}
\mathcal G_j(u)=\mathbb E_{J,\omega}\left[c(u)\sum_{\boldsymbol n\in\mathbb N_0^d}\prod_{m=1}^d\left(\int_0^up_{mj,J,\omega}(v|u)\mathcal G_m\left(u-v\right)\ \mathrm dv\right)^{n_m}\mathbb P\left(\boldsymbol K_{j,J}(u)=\boldsymbol n\right|J,\omega)\right].
\end{align*} 
By Lemma~\ref{clusterlemma} we know that 
$$p_{mj,J,\omega}(v|u)=\frac{h_{mj,J,\omega}(v)}{\int_0^uh_{mj,J,\omega}(s)\ \mathrm ds},$$ 
while using that $K_{mj}(\cdot)$ are Poisson processes with intensity $h_{mj,J,\omega}$, we calculate 
$$\mathbb P\left(\boldsymbol K_{j,J}(u)=\boldsymbol n\right|J,\omega)=\prod_{m=1}^d\frac{\left(\int_0^uh_{mj,J,\omega}(s)\ \mathrm ds\right)^{n_m}}{n_m!}\exp\left(-\int_0^uh_{mj,J,\omega}(s)\ \mathrm ds\right).$$ 
Combining the previous three displays, 
\begin{align*}
\mathcal G_j(u)&=\mathbb E_{J,\omega}\left[c(u)\sum_{\boldsymbol n\in\mathbb N_0^d}\prod_{m=1}^d\frac1{n_m!}\left(\int_0^uh_{mj,J,\omega}(v)\mathcal G_m\left(u-v)\right)\ \mathrm dv\right)^{n_m}\exp\left(-\int_0^uh_{mj,J,\omega}(s)\ \mathrm ds\right)\right]\\&=\mathbb E_{J,\omega}\left[c(u)\prod_{m=1}^d\exp\left(\int_0^uh_{mj,J,\omega}(s)\left(\mathcal G_m(u-v)-1\right)\ \mathrm ds\right)\right].
\end{align*} 
Plugging in the definition of $c(u)$ again, the theorem follows.
\end{proof}

\begin{lemma}\label{lemmaphiwelldefined}
The mapping $\phi$ from Definition~\ref{defphi} is well-defined, i.e., for $\boldsymbol{\mathcal J}_{\boldsymbol X,\boldsymbol Y}(\cdot)\in\mathbb J^d$, we have $\phi(\boldsymbol{\mathcal J}_{\boldsymbol X,\boldsymbol Y})(\cdot)\in\mathbb J^d$.
\end{lemma}
\begin{proof}
The proof is analogous to the proof of \cite{multivariateKLM-app}, Lemma~1. 
Our analogue of their Theorem~2 is Theorem~\ref{multifixedpoint}. 
Also, we should replace $B_{ij}g_{ij}$ in their proof by $h_{ij,J,\omega}$.
\end{proof}

For the fixed-point theorem, we wish to show that iterates $\phi^n(\boldsymbol{\mathcal J}^{(0)})(u)$ for some arbitrary $\boldsymbol{\mathcal J}^{(0)}(\cdot)\in\mathbb J^d$ converge to a unique limit, namely the value $\boldsymbol{\mathcal G}(u)$ that we are after. 
To this end, we need an appropriate notion of distance on $\mathbb J^d$. 
We define a norm $\|\cdot\|_{\mathbb J^d}$ as a uniform Euclidean norm by \begin{equation}
\|\boldsymbol{\mathcal J}\|_{\mathbb J^d}=\sup_{\substack{u\in[0,t],\boldsymbol s\in\mathbb R_+^d\\\boldsymbol z\in[-1,1]^d}}\|\boldsymbol{\mathcal J}(u,\boldsymbol s,\boldsymbol z)\|=\sup_{u,\boldsymbol s,\boldsymbol z}\|\boldsymbol{\mathcal J}(u,\boldsymbol s,\boldsymbol z)\|,
\end{equation} 
where $\|\cdot\|$ denotes the Euclidean norm on $\mathbb R^d$.

\begin{lemma}\label{phicontinuouslemma}
The mapping $\phi$ is continuous w.r.t.\ $\|\cdot\|_{\mathbb J^d}$, if we work on a bounded interval $[0,t]$.
\end{lemma}
\begin{proof}
The proof is similar to the proof of \cite{multivariateKLM-app}, Lemma~2. 
Take $\boldsymbol{\mathcal J}(\cdot),\tilde{\boldsymbol{\mathcal J}}(\cdot)\in\mathbb J^d$. 
It suffices to prove continuity in each coordinate separately, i.e., for each $j\in[d]$ we prove that given $\epsilon>0$, we can find $\delta>0$ such that 
$$\|\boldsymbol{\mathcal J}-\tilde{\boldsymbol{\mathcal J}}\|_{\mathbb J^d}<\delta$$ 
implies that $\|\phi_j(\boldsymbol{\mathcal J})-\phi_j(\tilde{\boldsymbol{\mathcal J}})\|_{\mathbb J}<\epsilon$. 
We have 
\begin{align*}
&\|\phi_j(\boldsymbol{\mathcal J})-\phi_j(\tilde{\boldsymbol{\mathcal J}})\|_{\mathbb J}=\sup_{u,\boldsymbol s,\boldsymbol z}\left|\phi_j(\boldsymbol{\mathcal J})(u,\boldsymbol s,\boldsymbol z)-\phi_j(\tilde{\boldsymbol{\mathcal J}})(u,\boldsymbol s,\boldsymbol z)\right|\\&\leq\sup_{u,\boldsymbol s,\boldsymbol z}\mathbb E_{J,\omega}\Bigg|\exp\left(\sum_{m=1}^d\left(s_mh_{mj,J,\omega}(u)+\int_0^uh_{mj,J,\omega}(v)\left(1-\mathcal J_m(u-v,\boldsymbol s,\boldsymbol z)\right)\ \mathrm dv\right)\right)\\&\phantom{\leq\sup_{u,\boldsymbol s,\boldsymbol z}\mathbb E_{J}\ }-\exp\left(\sum_{m=1}^d\left(s_mh_{mj,J,\omega}(u)+\int_0^uh_{mj,J,\omega}(v)\left(1-\tilde{\mathcal J}_m(u-v,\boldsymbol s,\boldsymbol z)\right)\ \mathrm dv\right)\right)\Bigg|\\&\leq\sup_{u,\boldsymbol s,\boldsymbol z}\mathbb E_{J,\omega}\left|\sum_{m=1}^d\int_0^uh_{mj,J,\omega}(v)\left(\mathcal J_m(u-v,\boldsymbol s,\boldsymbol z)-\tilde{\mathcal J}_m(u-v,\boldsymbol s,\boldsymbol z)\right)\ \mathrm d v\right|\\&\leq\sum_{m=1}^d\mathbb E_{J,\omega}\left[\int_0^t h_{mj,J,\omega}(v)\sup_{u,\boldsymbol s,\boldsymbol z}\left|\mathcal J_m(u-v,\boldsymbol s,\boldsymbol z)-\tilde{\mathcal J}_m(u-v,\boldsymbol s,\boldsymbol z)\right|\ \mathrm d v\right]\\&\leq\sum_{m=1}^d\mathbb E_{J,\omega}\|h_{mj,J,\omega}\|_{\infty}\int_0^t\sup_{u,\boldsymbol s,\boldsymbol z}\left|\mathcal J_m(u,\boldsymbol s,\boldsymbol z)-\tilde{\mathcal J}_m(u,\boldsymbol s,\boldsymbol z)\right| \ \mathrm dv\\&\leq d\max_{m,j\in[d]}\mathbb E_{J,\omega}\|h_{mj,J,\omega}\|_\infty t\delta,
\end{align*} 
where the first inequality follows by the triangle inequality and the fact that $|z_i|\leq1$; the second by the mean value theorem applied to $x\mapsto e^x$, using that $\mathcal J_m(u-v,\boldsymbol s,\boldsymbol z)\leq 1$; the third by three more triangle inequalities and positivity of the integrand; the fourth since $h_{mj,J_\omega}(v)<\|h_{mj,J_\omega}\|_\infty$ a.e.; and the fifth is obvious. 
Note that $\max_{m,j\in[d]}\mathbb E_{J,\omega}\|h_{mj,J,\omega}\|_\infty<\infty$ by \cite{DHpaper-app}, Definition~2. 
It follows that if $$\delta<\frac{\epsilon}{d\max_{m,j\in[d]}\mathbb E_{J,\omega}\|h_{mj,J,\omega}\|_\infty t},$$ then $\|\phi_j(\boldsymbol{\mathcal J})-\phi_j(\tilde{\boldsymbol{\mathcal J}})\|_{\mathbb J}<\epsilon$.
\end{proof}

\begin{remark}
The restriction that we should work on bounded intervals $[0,t]$ in Lemma~\ref{phicontinuouslemma} is no obstacle. 
Whenever we want to use the bound appearing in the proof of this lemma, or want to use continuity of $\phi$, we do this to find some value $\boldsymbol{\mathcal G}(u)$. 
Then we can just take any $t\geq u$, and apply the lemma.
\end{remark}

We now state the convergence result. 
For some joint transform $\boldsymbol{\mathcal J}^{(0)}(\cdot)\in\mathbb J^d$. We construct the sequence $(\boldsymbol{\mathcal J}^{(n)}(\cdot))_{n\in\mathbb N_0}$ by setting $\boldsymbol{\mathcal J}^{(n)}(\cdot):=\phi(\boldsymbol{\mathcal J}^{(n-1)})(\cdot)$. 
By Lemma~\ref{lemmaphiwelldefined}, we know that $\boldsymbol{\mathcal J}^{(n)}(\cdot)\in\mathbb J^d$ for all $n\in\mathbb N_0$.

\begin{theorem}\label{theoremmultconvergence}
For any $\boldsymbol{\mathcal J}^{(0)}(\cdot)\in\mathbb J^d$, the sequence $(\boldsymbol{\mathcal J}^{(n)}(u))_{n\in\mathbb N_0}$ converges pointwise to the fixed point
$\boldsymbol{\mathcal G}(u):=\boldsymbol{\mathcal J}_{\boldsymbol S^{\boldsymbol Q},\boldsymbol S^{\boldsymbol \lambda}}(u)$. 
That is, as $n\to\infty$, for any $u\leq t$, \begin{equation}
\boldsymbol{\mathcal J}^{(n)}(u)\equiv\boldsymbol{\mathcal J}^{(n)}(u,\boldsymbol s,\boldsymbol z)\to\boldsymbol{\mathcal J}_{\boldsymbol S^{\boldsymbol Q},\boldsymbol S^{\boldsymbol \lambda}}(u,\boldsymbol s,\boldsymbol z)\equiv\boldsymbol{\mathcal J}_{\boldsymbol S^{\boldsymbol Q},\boldsymbol S^{\boldsymbol \lambda}}(u).
\end{equation}
\end{theorem}
\begin{proof}
The proof is analogous to the proof of \cite{multivariateKLM-app}, Theorem~3. 
We should refer to Lemma~\ref{phicontinuouslemma} and Theorem~\ref{multifixedpoint} wherever they refer to their Lemma~2 and Theorem~2, respectively.
\end{proof}

Theorem~\ref{theoremmultconvergence} describes how $\boldsymbol{\mathcal J}_{\boldsymbol S^{\boldsymbol Q},\boldsymbol S^{\boldsymbol \lambda}}(u)$ can be approximated. 
Together with Theorem~\ref{jointtransformprocesstocluster}, this gives a full characterization of the time-dependent joint transform of $(\boldsymbol Q(\cdot),\boldsymbol\lambda(\cdot))$. 
This also gives rise to a numerical procedure to characterize $(\boldsymbol Q(\cdot),\boldsymbol\lambda(\cdot))$: take some $\boldsymbol{\mathcal J}^{(0)}(\cdot)\in\mathbb J^d$, and approximate $\boldsymbol{\mathcal J}_{\boldsymbol S^{\boldsymbol Q},\boldsymbol S^{\boldsymbol \lambda}}(u)$ for a grid of values $\{u_1,\ldots, u_N\}\subset[0,T]$, by iterating $\phi$ until convergence obtains. 
Using this, it is possible to approximate $\mathcal J_{\boldsymbol Q,\boldsymbol \lambda}(t,\boldsymbol s,\boldsymbol z)$ with the aid of Theorem~\ref{jointtransformprocesstocluster}. 
Thereafter, this joint transform can either be inverted numerically  to determine the multivariate CDF, or it can be differentiated numerically to determine joint moments.


{\small
\begin{thebibliography}{99}%

\bibitem{ACL15}
  {\sc Y. A\"it-Sahalia, J.~A. Cacho-Diaz}, and {\sc R.~J.~A. Laeven} (2015).
  Modeling financial contagion using mutually exciting jump processes.
  {\it Journal of Financial Economics} {\bf 117}, pp.\ 585--606.

\bibitem{ALP14}
  {\sc Y. A\"it-Sahalia, R.~J.~A. Laeven}, and {\sc L. Pelizzon} (2014).
  Mutual excitation in Eurozone sovereign CDS.
  {\it Journal of Econometrics} {\bf 183}, pp.\ 151--167.



\bibitem{GoFpaper}
  {\sc J. Baars}, {\sc S.~U. Can}, and {\sc R.~J.~A. Laeven} (2025).  
  Asymptotically distribution-free goodness-of-fit testing for point processes.  
  Preprint. Available at \url{https://arxiv.org/abs/2503.24197v1}. 
  
\bibitem{Supplement} 
  {\sc J. Baars}, {\sc R. J. A. Laeven}, and {\sc M. Mandjes} (2025).
  Online supplement to ``Delayed Hawkes birth-death processes''.

\bibitem{Bacry}
  {\sc E. Bacry}, {\sc S. Delattre}, {\sc M. Hoffmann}, and {\sc J. F. Muzy} (2013).
  Some limit theorems for Hawkes processes and application to financial statistics.
  {\it Stochastic Processes and their Applications} {\bf 123}, pp.\ 2475--2499. 
  
\bibitem{finance2}
  {\sc E. Bacry} and {\sc J. F. Muzy} (2014).
  Hawkes model for price and trades high-frequency dynamics.
 {\it Quantitative Finance} {\bf 14}, pp.\ 1147--1166.

\bibitem{finance}
  {\sc L. Bauwens} and {\sc N. Hautsch} (2009).
  Modelling financial high frequency data using point processes.
  In book: {\it Handbook of Financial Time Series},  pp.\ 953--979.
  
\bibitem{Bingham}
  {\sc N. H. Bingham}, {\sc C. M. Goldie}, and {\sc J. L. Teugels} (1989).
  {\it Regular Variation}. 
  Cambridge University Press {\bf 27}, Cambridge.

\bibitem{LagrangeSylvester}
  {\sc V. C. Borkar} and {\sc M. A. Salman} (2016). 
  The exact methods to compute the matrix exponential.
 {\it IOSR Journal of Mathematics} {\bf 12}, pp.\ 72--86.



\bibitem{Stabilitypaper}
  {\sc P. Brémaud} and {\sc L. Massoulié} (1996).
  Stability of nonlinear Hawkes processes.
  {\it The Annals of Probability} {\bf 24}, pp.\ 1563--1588.
  
\bibitem{CCC22}
  {\sc P. Cattiaux}, {\sc L. Colombani}, and {\sc M. Costa} (2022).
  Limit theorems for Hawkes processes including inhibition.
  {\it Stochastic Processes and their Applications} {\bf 149}, pp.\ 404--426.

\bibitem{covid}
  {\sc W. Chiang}, {\sc X. Liu}, and {\sc G. Mohler} (2022).
  Hawkes process modeling of COVID-19 with mobility leading indicators and spatial covariates.
  {\it International Journal of Forecasting} {\bf 38}, pp.\ 505--520.

\bibitem{Chu}
  {\sc W. Chu} (2010).
  Fibonacci polynomials and Sylvester determinant of tridiagonal matrix.
  {\it Applied Mathematics and Computation} {\bf 216}, pp.\ 1018--1023.

\bibitem{elementary}
  {\sc L.~R. Cui}, {\sc A.~G. Hawkes}, and {\sc H. Yi} (2020).
  An elementary derivation of moments of Hawkes processes.
  {\it Advances in Applied Probability} {\bf 52}, pp.\ 102--137.

\bibitem{DaleyVereJones}
  {\sc D.~J. Daley} and {\sc D. Vere-Jones} (2003).
  \textit{An Introduction to the Theory of Point Processes},
  Vol I and II, 2nd ed. 
  Springer-Verlag, New York.


  
\bibitem{Conditional uniformity}
  {\sc A. Daw} (2023).
  Conditional uniformity and Hawkes processes.
  {\it Mathematics of Operations Research},
  Articles in Advance.

\bibitem{Ephemeral}
  {\sc A. Daw} and {\sc J. Pender} (2022).
  An ephemerally self-exciting point process.
  {\it Advances in Applied Probability} {\bf 54}, pp.\ 340--403.

\bibitem{matrix method}
  {\sc A. Daw} and {\sc J. Pender} (2023).
  Matrix calculations for moments of Markov processes.
  {\it Advances in Applied Probability} {\bf 55}, pp.\ 126--150.

\bibitem{dawpenderqueues}
  {\sc A. Daw} and {\sc J. Pender} (2018).
  Queues driven by Hawkes processes.
  {\it Stochastic Systems} {\bf 8}, pp.\ 192--229.
  
\bibitem{DZ11}
  {\sc A. Dassios} and {\sc H. Zhou} (2011).
  A dynamic contagion process.
  {\it Advances in Applied Probability} {\bf 43}, pp.\ 814--846.

\bibitem{Hawkesnetwork}
  {\sc S. Delattre}, {\sc N. Fournier}, and {\sc M. Hoffmann} (2016).
  Hawkes processes on large networks.
  {\it The Annals of Applied Probability} {\bf 26}, pp.\ 216--261.
  
\bibitem{social media2}
  {\sc N. Du}, {\sc Y. Wang}, {\sc L. Song}, {\sc H.  Zhang}, and {\sc L. Ma} (2013).
  Hawkes processes for clickstream data and the emergence of collective attention.
  {\it Proceedings of the 22nd international conference on World Wide Web}, pp.\ 609--620.


\bibitem{Hawkes}
  {\sc A. G. Hawkes} (1971).
  Spectra of some self-exciting and mutually exciting point processes.
  {\it Biometrika} {\bf 58}, pp. 83--90.

\bibitem{Hawkes2}
  {\sc A. G. Hawkes} and {\sc D. Oakes} (1974).
  A cluster process representation of a self-exciting process.
  {\it Journal of Applied Probability} {\bf 11}, pp.\ 493--503.

\bibitem{hittingtimetheorem}
  {\sc R. Van der Hofstad} and {\sc M. Keane} (2008).
  An elementary proof of the hitting time theorem.
  {\it The American Mathematical Monthly} {\bf 115}, pp.\ 753--756.

\bibitem{Horst}
  {\sc U. Horst} and {\sc W. Xu} (2021).
  Functional limit theorems for marked Hawkes point measures.
  {\it Stochastic Processes and their Applications} {\bf 134}, pp.\ 94--131.

\bibitem{Ikefuji}
  {\sc M. Ikefuji}, {\sc R. J. A. Laeven}, {\sc J. R. Magnus} and {\sc Y. Yue} (2022). 
  Earthquake risk embedded in property prices: {E}vidence from five {J}apanese cities. 
  {\it Journal of the American Statistical Association} {\bf 117}, pp.\ 82--93.

\bibitem{self-correcting0}
  {\sc V. Isham} and {\sc M. Westcott} (1979).
  A self-correcting point process.
  {\it Stochastic Processes and their Applications} {\bf 8}, pp.\ 335--347.


\bibitem{Rosenbaum1}
  {\sc T. Jaisson} and {\sc M. Rosenbaum} (2015).
  Limit theorems for nearly unstable Hawkes processes.
  {\it The Annals of Applied Probability} {\bf 25}, pp.\ 600--631.


\bibitem{Rosenbaum2}
 {\sc T. Jaisson} and {\sc M. Rosenbaum} (2016).
  Rough fractional diffusions as scaling limits of nearly unstable heavy tailed Hawkes processes.
  {\it The Annals of Applied Probability} {\bf 26}, pp.\  2860--2882.

\bibitem{multivariateKLM}
  {\sc R. Karim}, {\sc R. J. A. Laeven}, and {\sc M. Mandjes} (2021).
  Exact and asymptotic analysis of general multivariate Hawkes processes and induced population processes.
  Preprint. Available at \url{https://arxiv.org/abs/2106.03560}.

\bibitem{ldpKLM}
  {\sc R. Karim}, {\sc R. J. A. Laeven} and {\sc M. Mandjes} (2025).
   Compound multivariate Hawkes processes: {L}arge deviations and rare event simulation. 
   {\it Bernoulli} {\bf 31}, pp.\ 3113--3138.
  
\bibitem{Kirchner}
  {\sc M. Kirchner} (2017).
  An estimation procedure for the Hawkes process.
  {\it Quantitative Finance} {\bf 17}, pp.\ 571--595. 
  
\bibitem{shotnoise}
  {\sc D. T. Koops}, {\sc O. J.  Boxma}, and {\sc M. Mandjes} (2017).
  Networks of $\cdot/G/\infty$ server queues with shot-noise-driven arrival intensities.
  {\it Queueing Systems} {\bf 86}, pp.\ 301--325.
  
\bibitem{Infinite server queues}
  {\sc D. T. Koops}, {\sc M. Saxena}, {\sc O. J.  Boxma}, and {\sc M. Mandjes} (2018).
  Infinite-server queues with Hawkes input.
  {\it  Journal of Applied Probability} {\bf 55}, pp.\ 920--943.
  
\bibitem{Massoulie}
  {\sc L. Massoulié} (1998).
  Stability results for a general class of interacting point processes dynamics, and applications.
  {\it Stochastic Processes and their Applications} {\bf 75}, pp.\ 1--30.

\bibitem{Neveu}
  {\sc J. Neveu} (1965).
  \textit{Mathematical Foundations of the Calculus of Probability}, 1st ed.
  Holden-Day series in probability and statistics.


\bibitem{OgataLewis}
  {\sc Y. Ogata} (1981).
  On Lewis' simulation method for point processes.
  {\it IEEE Transactions on Information Theory} {\bf 27}, pp.\ 23--31.

\bibitem{seismology}
  {\sc Y. Ogata} (1988).
  Statistical models for earthquake occurrences and residual analysis for point processes.
  {\it Journal of the American Statistical Association} {\bf  83}, pp.\ 9--27.

\bibitem{self-correcting1}
  {\sc Y. Ogata} and {\sc D. Vere-Jones} (1984).
  Inference for earthquake model.
  {\it Stochastic Processes and their Applications} {\bf 17}, pp.\ 337--347.

\bibitem{self-correcting2}
  {\sc Y. Ogata} and {\sc D. Vere-Jones} (1984).
  On the moments of a self-correcting process.
  {\it Journal of Applied Probability} {\bf 21}, pp.\ 335--342.

\bibitem{criminology}
  {\sc J. Olinde} and {\sc M. Short} (2020). 
  A self-limiting Hawkes process: Interpretation, estimation, and use in crime modeling.
  {\it 2020 IEEE International Conference on Big Data}.
  
\bibitem{AgeDependent}
  {\sc M. B. Raad}, {\sc S. Ditlevsen}, and {\sc E. Löcherbach} (2020).
  Stability and mean-field limits of age dependent Hawkes processes.
  {\it Annales de l'Institut Henri Poincaré Probabilités et Statistiques} {\bf 56}, pp.\ 1958--1990.

\bibitem{neuroscience}
  {\sc P. Reynaud-Bouret}, {\sc R. Lambert}, {\sc C. Tuleau-Malot}, {\sc T. Bessaih}, {\sc V. Rivoirard}, {\sc Y. Bouret}, and {\sc N. Leresche} (2018). 
  Reconstructing the functional connectivity of multiple spike trains using Hawkes models.
  {\it Journal of Neuroscience Methods} {\bf 297}, pp.\ 9--21.

\bibitem{social media}
  {\sc M. Rizoiu}, {\sc Y. Lee}, {\sc S. Mishra}, and {\sc L. Xie} (2017).
  A tutorial on Hawkes processes for events in social media.
  In book: {\it Frontiers of Multimedia Research}, pp.\ 191--218.





\bibitem{earthquakestressrelease}
  {\sc D. Vere-Jones} (1978).
  Earthquake predicton - a statistician's view.
  {\it Journal of Physics of the Earth} {\bf 26}, pp.\ 129--146.

\bibitem{Zhu}
  {\sc L. Zhu} (2013). 
  Central limit theorem for nonlinear Hawkes processes.
  {\it Journal of Applied Probability} {\bf  50}, pp.\ 760--771.

\end{thebibliography}
}


{\small
\begin{thebibliography}{99}%


\bibitem{DHpaper-app}
  {\sc J. Baars}, {\sc R. J. A. Laeven}, and {\sc M. Mandjes} (2025).
  Delayed Hawkes birth-death processes.
  Preprint. 

\bibitem{multivariateKLM-app}
  {\sc R. Karim}, {\sc R. J. A. Laeven}, and {\sc M. Mandjes} (2021).
  Exact and asymptotic analysis of general multivariate Hawkes processes and induced population processes.
  Preprint. Available at \url{https://arxiv.org/abs/2106.03560}.



\end{thebibliography}
}
\end{document}